\documentclass[11pt,twoside]{article}
\usepackage{bbm}
  \usepackage{url}
  \usepackage{amssymb,amsmath}
     \usepackage{graphicx}
   \usepackage[a4paper, margin=2cm]{geometry}

  \input xypic
  \xyoption{all}

 \usepackage{tikz}
\usetikzlibrary{arrows}

 \usepackage{times}

  \usepackage{amsmath, amsthm, amsfonts}

\numberwithin{equation}{section}
\newtheorem{theorem}{Theorem}[section]

\newtheorem{corollary}[theorem]{Corollary}

\newtheorem{lemma}[theorem]{Lemma}

\newtheorem{remark}[theorem]{Remark}

\def \inv{^{-1}}
\def  \supp{\mbox{supp}}
\def  \inj{\mbox{inj}}
\def \fc {\mathfrak{c}}
\def \half{\frac{1}{2}}

\def \p{\partial}
\def  \wkr {W^{k,2,\alpha}_{r,u_{(r)}} }

\def  \cwk {\mathcal{W}^{k,2,\alpha}_{u} }
\def  \wka {W^{k,2,\alpha}_{u} }

\def  \lka {L_{u}^{k-1,2,\alpha} }
\def \lkar {L^{k-1,2,\alpha}_{r,u_{(r)}} }
\def \ker {\mbox{ker}}

\def \E{{\mathcal E}}
\def \v {\vskip 0.1in}
\def \n {\noindent}

\def \im {\mbox{im}}

\begin{document}

  \begin{center}
 {\LARGE \bf The Exponential Decay of Gluing Maps \\ for $J$-Holomorphic map Moduli Spaces}
\v
  {\large An-Min Li and Li Sheng}\footnote{partially  supported by a NSFC grant}
   \footnote{anminliscu@126.com, lshengscu@gmail.com}

{Department of Mathematics, Sichuan University
        Chengdu, PRC}
        \end{center}

 % \tableofcontents
\v\v
\begin{abstract}
We prove the exponential decay of the derivative of the gluing maps with respect to the gluing parameter.
\end{abstract}

\section{\bf Introduction and Preliminary}\label{introduction}

The gluing analysis plays an important role in Gromov-Witten Invariants theory. In this paper we study the gluing estimates, in particular the estimates of the derivatives of the gluing maps with respect to the gluing parameter $r$. We describe now the problem and state our main result. We only consider the case of gluing one nodal, for general cases the estimates are the same.
\v
\subsection{\bf $J$-holomorphic maps from Riemann surface with one nodal point}\label{s_intro_1}
Let $(M,\omega,J)$ be a closed $C^{\infty}$ symplectic manifold of dimension $2m$ with $\omega$-tame almost complex structure $J$, where $\omega$ is a symplectic form. Then there is a Riemannian metric
\begin{equation}\label{definition_of_metrics}
G_J(v,w):=<v,w>_J:=\frac{1}{2}\left(\omega(v,Jw)+\omega(w,Jv)\right)
\end{equation}
for any $v, w\in TM$.  Following \cite{MS} we choose the complex linear connection
 $$
 \widetilde {\nabla}_{X}Y=\nabla_{X}Y-\tfrac{1}{2}J\left(\nabla_{X}J\right)Y $$
 induced by the Levi-Civita connection $\nabla$ of the metric $G_{J}.$ Let $(\Sigma,j)$ be a smooth Riemann surface. A map $u:\Sigma\longrightarrow M$ is called a $(j,J)$-holomorphic map if $du\circ j=J\circ du$. Alternatively
\begin{equation}\label{holo}
\bar\p_{j,J}(u):=\half\left(du + J(u) du\circ j\right)=0.
\end{equation}
Let $(\Sigma, j,{\bf y}, q)$ be a marked nodal Riemann surface of genus $g$ with $n$ marked points ${\bf y}=(y_1,...,y_n)$ and one nodal point $q$.  We write the marked nodal Riemann surface as
\begin{equation}\label{rie}
\left[\Sigma=\Sigma_{1}\wedge\Sigma_{2},j=(j_{1},j_{2}),{\bf y}=({\bf y}_{1},{\bf y}_{2}), q=(p_1,p_2)\right],\end{equation}
 where $(\Sigma_{i},j_{i},{\bf y}_{i}, q_i)$, $i=1,2$, are smooth Riemann surfaces of genus $g_i$ with $n_i$ marked points $\mathbf y_{i}$ and puncture $q_i$. We say that $q_1,q_2$ are paired to form $q$.  We assume that $(\Sigma_{i},j_{i},{\bf y}_{i},q_{i})$ is stable, i.e., $n_{i}+1>2-2g_{i}$, $i=1,2$. Let $z_i$ be a local complex coordinate around $q_i$, $z_i(q_i)=0$, $i=1,2$.
 Let
\begin{equation}\label{cylind coord}
z_1=e^{-s_1 - 2\pi \sqrt{-1} t_1},\;\;\;z_2=e^{s_2 +2\pi\sqrt{-1}  t_2}.
\end{equation}
$(s_i, t_i)$ are called the holomorphic cylindrical coordinates near $q_{i}$.  In terms of the holomorphic cylindrical coordinates we write
\begin{align}\label{rie cylind coord}
&\stackrel{\circ}{\Sigma}_{1}:=\Sigma_1\setminus\{q_1\}\cong\Sigma_{10}\cup\{[0,\infty)\times S^1\},\\
&\stackrel{\circ}{\Sigma}_{2}:=\Sigma_2\setminus\{q_2\}\cong\Sigma_{20}\cup\{(-\infty,0]\times S^1\}.\nonumber
\end{align}
Here $\Sigma_{i0}\subset \Sigma_i$, $i=1,2$, are compact surfaces with boundary. Put $\stackrel{\circ}{\Sigma}= \Sigma\setminus\{q_1,q_2\} =\stackrel{\circ}{\Sigma}_{1}\cup\stackrel{\circ}{\Sigma}_{2}$. We introduce the notations
\begin{equation}\label{D(R)}
\Sigma_i(R_0)=\Sigma_{i0}\cup \{(s_i,t_i)|\;|s_i| \leq R_0\},\;\; \;\;\;\;\;\Sigma(R_0)=\Sigma_1(R_0)\cup \Sigma_2(R_0).
\end{equation}
We choose a local coordinate system $(a_{1},a_{2})\in \mathbf{A}_1\times \mathbf{A}_2$ for complex structure on $\stackrel{\circ}{\Sigma}_{1}\cup\stackrel{\circ}{\Sigma}_{2}$ , where $\mathbf{A}_i\subset \mathbb R^{6g_{i}-6+2n_{i}}$ is an open set, and
$a_i=(j_{i},{\bf y}_{i}).$
\v
For any gluing parameter $(r,\tau)$ with $r\geq R_{0}$ and $\tau\in S^1$ we construct a surface
$\Sigma_{(r)} =\Sigma_1 \#_{(r)} \Sigma_2 $ as follows, where and later we
use $(r)$ to denote gluing parameters. We cut off the
part of $\Sigma_i$ with cylindrical
coordinate $|s_i|>\frac{3r}{2}$ and glue the remainders along
the collars of length $r$ of the cylinders with the gluing formulas:
\begin{equation}\label{gluing rie}
s_1=s_2 + 2r,\;\;\;t_1=t_2 + \tau.
\end{equation}
$\Sigma(R_{0})$ can naturally equate to the subset of $\Sigma_{(r)}$.
Then $((a_1,a_2),r, \tau)$ is a local coordinate system near $(\Sigma, j,{\bf y}, q)$ in the Teichm\"uller space $\mathbf{T}_{g,n}$.
  For  any $a=(a_{i},a_{i})\in \mathbf{A}_{1}\times \mathbf{A}_{2}$ with $a_{i}=(j_{i},\mathbf{y}_{i}),$  let $j_{(r),a}$ be the complex structure on $\Sigma_{(r)}$ satisfying
$$
j_{(r),a}|_{\Sigma_{i}(R_{0})}=j_{i},\;\;\;\;\;j_{(r),a}\left.\frac{\p}{\p s_{i}}\right|_{\Sigma_{(r)}\setminus \Sigma(R_{0})}=\left.\frac{\p }{\p t_{i}}\right|_{\Sigma_{(r)}\setminus \Sigma(R_{0})}
$$
where $z=e^{-r-2\pi \sqrt{-1}\tau}$. If no danger of confusion we  denote $j_{(r),a}$ by $j_{a}.$
\v
We may choose a family  of metrics $\mathbf{g}_i$ on $ \stackrel{\circ}{\Sigma}_i $ in the given conformal class $j_i$, depending on $a_i\in \mathbf{A}_i$ smoothly, such that, restricting to $\Sigma\setminus \Sigma(R_{0})$,
$$\mathbf{g}_i=(ds_i)^2+(dt_i)^2,$$ the standard cylinder metric. Then we define a metric $\mathbf{g}$ on $\stackrel{\circ}{\Sigma}$ as $\mathbf{g}=\mathbf{g}_1\oplus \mathbf{g}_2.$
\v
Let $u=(u_1,u_2)$, where $u_i:\Sigma_i\to M$ is $(j_i,J)$-holomorphic map such that $u_1(q_1)=u_2(q_2)$.
We choose the local normal coordinates $(x^{1},\cdots,x^{2m})$ in a neighborhood  $O_{u(q)}$ of $u(q)$ such that
$$
(x^{1},\cdots,x^{2m})(u(q))=0,\;\;\;\;\;\left.J(0)\frac{\p}{\p x^{i}}\right|_{0}=\left.\frac{\p}{\p x^{m+i}}\right|_{0},\;\;\;\;\;\;\;
\left.J(0)\frac{\p}{\p x^{m+i}} \right|_{0}=\left.-\frac{\p}{\p x^{i}}\right|_{0},\;\;\;\;\;\;i\leq m.
$$
In terms of the holomorphic cylindrical coordinates $(s_i,t_i)$ over each tube the linearized operator
$D_{u_i}$ takes the following form ( see Appendix
\ref{Linearized operator})
\begin{equation}
D_{u_i}=\frac{\partial}{\partial
s_i}+J_0\frac{\partial} {\partial t_i}+F_{u_i}^{1}+F_{u_i}^{2}\frac{\partial} {\partial t_i},
\end{equation}
where  $\sum\limits_{p+q=d} \left|\frac{\p^{d}F^{l}_{u_i}}{\partial s_{i}^ p \partial  t_{i}^q}\right|\rightarrow 0,$ for $l,i=1,2,\forall d\geq0$, exponentially and
uniformly in $t_i$ as $|s_i|\rightarrow \infty $. Here $J_{0}$ is the standard complex structure in $O_{u(q)}$ such that
$ J_{0}\frac{\p}{\p x^{i}} = \frac{\p}{\p x^{m+i}},\;\;
 J_{0}\frac{\p}{\p x^{m+i}} =-\frac{\p}{\p x^{i}}.$
 Therefore, the operator $H_{s_i}=J_0\frac{d}{dt_i}+ F_{u_i}^{1}+F_{u_i}^{2}\frac{\partial} {\partial t_i}
$ converges
to $H_{\infty}=J_0\frac{d}{dt_i}$.  Obviously, the operator $D_{u_i}$ is not Fredholm operator because over nodal end the
operator $H_{\infty}=J_0\frac{d}{dt_i}$ has zero eigenvalue. The $\ker
H_{\infty}$ consists of constant vectors in $T_{u(q)}M$.
To recover a Fredholm theory we use weighted spaces $W^{k,2,\alpha}(u^*TM)$. In this paper we take $k> 2.$
Fix a positive function $W$ on $\Sigma$ which has order equal
to $e^{\alpha |s|} $ on each end, where $\alpha$ is a small
constant such that $0<\alpha  <1$ and over each end
$H_{\infty}- \alpha = J_0\frac{d}{dt}- \alpha $ is invertible. We
will write the weight function simply as $e^{\alpha |s|}.$ For any
section $h\in C_{c}^{\infty}(\stackrel{\circ}{\Sigma};u^{\ast}TM)$ and section $\eta \in
C_{c}^\infty(\stackrel{\circ}{\Sigma}, u^{*}TM\otimes \wedge^{0,1}_jT^{*}\stackrel{\circ}{\Sigma})$ we define the norms
\begin{eqnarray}\label{norm_1_p_alpha}
&&\|h\|_{k,2,\alpha}=
 \left(\int_{\Sigma}e^{2\alpha|s|} \sum_{i=0}^{k} |\nabla^i h|^2
dvol_{\Sigma}\right)^{1/2},\\
\label{norm_p_alpha}
&&\|\eta\|_{k-1,2,\alpha}= \left(\int_{\Sigma}e^{2\alpha|s|}\sum_{i=0}^{k-1} |\nabla^i \eta|^2
dvol_{\Sigma}\right)^{1/2}.
\end{eqnarray}
Denote by $W^{k,2,\alpha}(\stackrel{\circ}{\Sigma};u^{\ast}TM)$ and
$W^{k-1,2,\alpha}(\stackrel{\circ}{\Sigma}, u^{*}TM\otimes \wedge^{0,1}_jT^{*}\stackrel{\circ}{\Sigma})$ the complete spaces with respect to the norms
\eqref{norm_1_p_alpha} and \eqref{norm_p_alpha}
respectively. The operator $$D_u:   W^{k,2,\alpha}(\stackrel{\circ}{\Sigma};u^{\ast}TM)\rightarrow W^{k-1,2,\alpha}(\stackrel{\circ}{\Sigma}, u^{*}TM\otimes \wedge^{0,1}_jT^{*}\stackrel{\circ}{\Sigma})$$
is a Fredholm operator as long as $\alpha$ does not lie in the spectrum
of the operator $H_{\infty}$.
\vskip 0.1in
\noindent
\v

We choose $R_0$ so large that  $u(\{|s_i|\geq \frac{r}{2}\})$ lie in $O_{u(q)}$ for any $r>R_0$. In this coordinate system we identify $T_xM$ with $T_{u(q)}M$ for all $x\in O_{u(q)}$. Any $h_0\in T_{u(q)}M$ may be considered as a vector field in the coordinate neighborhood.
We fix a smooth cutoff function $\varrho$:
\[
\varrho(s)=\left\{
\begin{array}{ll}
1, &\mbox{ if }\ |s|\geq \bar{d}, \\
0, &\mbox{ if }\ |s|\leq \frac{\bar{d}}{2}
\end{array}
\right.
\]
where $\bar{d}$ is a large positive number. Put
$$\hat{h}_0=\varrho h_0.$$
Then for $\bar{d}$ large enough $\hat{h}_0$ is a section in $C^{\infty}(\stackrel{\circ}{\Sigma}; u^{\ast}TM)$
supported in the tube $\{(s,t)||s|\geq \frac{\bar{d}}{2}, t \in {S} ^1\}$.
Denote
$${\mathcal W}^{k,2,\alpha}(\stackrel{\circ}{\Sigma};u^{\ast}TM)=\{h+\hat{h}_0 | h \in
W^{k,2,\alpha}(\stackrel{\circ}{\Sigma};u^{\ast}TM),h_0 \in \ker H_{\infty}\}.$$
 We define weighted Sobolev  norm  on ${\mathcal W}^{k,2,\alpha}$ by $$\| h+\hat{h}_{0}\|_{\mathcal{W},k,2,\alpha}=
 \|h\|_{k,2,\alpha} + |h_{0}| ,$$
 where $|h_{0}|=[G_{J}(h_{0},h_{0})_{u(q)}]^{\frac{1}{2}}$.
 If no danger of confusion, later we will denote
$$\mathcal{W}^{k,2,\alpha}_{u}=\mathcal{W}^{k,2,\alpha}(\stackrel{\circ}{\Sigma},u^*TM),\;\;\; W^{k,2,\alpha}_{u}= W^{k,2,\alpha}(\stackrel{\circ}{\Sigma},u^*TM),\;\;\;L_{u}^{k-1,2,\alpha}=W^{k-1,2,\alpha}(\stackrel{\circ}{\Sigma}, u^{*}TM\otimes \wedge^{0,1}_{j_o}T^{*}\stackrel{\circ}{\Sigma}).
$$
Let
$$\mathcal{B}_i=\{ v_{i}=\exp_{u_{i}}(h_{i})|\;h_{i}\in \mathcal{W}^{k,2,\alpha}_{u_i}\},$$
$$\mathcal{B}_{1}\times_{q}\mathcal B_{2}:=\left\{v=(v_1,v_2)\in \mathcal{B}_{1}\times \mathcal B_{2} \;|\;v_1(q_1)=v_2(q_2) \right\}.$$
For any $\rho>0$ set
\begin{equation}\label{eqn_g10}
O_{b_{o}}(\rho ):=\{v\in \mathcal{B}_{1}\times_{q} \mathcal{B}_{2} \;|\; v=\exp_{u}(h+\hat h_{0}),\;\|h+\hat h_{0}\|_{\mathcal W,k,2,\alpha}<\rho  \}.
\end{equation}
Let $\mathcal E_i$ be the infinite dimensional Bananch bundle over $\mathbf{A}_i\times\mathcal{B}_i$, whose fiber at $(a_i,v_i)$ is $W^{k-1,2,\alpha}(\stackrel{\circ}{\Sigma}, v_{i}^{*}TM\otimes \wedge^{0,1}_{j_{a_{i}}}T^{*}\stackrel{\circ}{\Sigma}).$.
We have a
Fredholm system $(\mathbf{A}_i\times\mathcal{B}_i, \mathcal{E}_i, \overline{\p}).$
We will fix a complex structure $a_o=(a_{o1}, a_{o2})$.

\subsection{\bf Pregluing}\label{s_intro_2}
Let $b_{oi}=(a_{oi}, u_{i}),i=1,2,$ where  $a_{oi}\in \mathbf{}_{i}$, $u_{i}:\Sigma_i\rightarrow M$ are $(j_{oi},J)$-holomorphic maps with $u_{1}(q_1)=u_{2}(q_2)$.  Denote $b_o=(b_{o1},b_{o2})=(a_o,u)=(j_o, {\bf y}_o, u)$. Let $r>4R_0.$ We glue the map $(u_1,u_2)$ to get a pregluing maps $u_{(r)}$, a family of approximate $(j_o,J)$-holomorphic maps as follows.
Set\\
\[
u_{(r)}=\left\{
\begin{array}{ll}
u_1 \;\;\;\;\; on \;\;\Sigma_{10}\bigcup\{(s_1,t_1)|0\leq s_1 \leq
\frac{r}{2}, t_1 \in S^1 \}    \\  \\
u_1(q)=u_2(q) \;\;\;\;\;
on \;\;\{(s_1,t_1)| \frac{3r}{4}\leq s_1 \leq
\frac{5r}{4}, t_1 \in S^1 \}  \\   \\
u_2 \;\;\;\;\; on \;\;\Sigma_{20}\bigcup\{(s_2,t_2)|0\geq s_2
\geq - \frac{r}{2}, t_2 \in S^1 \}     \\
\end{array}.
\right.
\]
To define the map $u_{(r)}$ in the remaining part we fix a smooth cutoff
function $\beta : {\mathbb{R}}\rightarrow [0,1]$ such that
\begin{equation}\label{def_beta}
\beta (s)=\left\{
\begin{array}{ll}
1 & if\;\; s \geq 1 \\
0 & if\;\; s \leq 0
\end{array}
\right.
\end{equation}
 and $\sqrt{1-\beta^2}$ is a smooth function,  $0\leq \beta^{\prime}(s)\leq 4$ and $\beta^2(\frac{1}{2})=\frac{1}{2}.$
 We define\\
$$u_{(r)}= u_1(q)+ \beta\left(3-\frac{4s_1}{r}\right)(u_1(s_1,t_1)-u_1(q)) +
\beta\left(\frac{4s_1}{r}-5\right)(u_2(s_1-2r,t_1-\tau)- u_2(q)).$$
\vskip 0.1in
\noindent
We introduce a notation $\beta_{i;R}(s_i)$. For any $R>0$ denote \begin{equation}\label{beta}
\beta_{1;R}(s_1)=\beta\left(\frac{1}{2}+\frac{r-s_1}{R}\right),\;\;\;\;\beta_{2;R}(s_{2})=\sqrt{1-\beta^2\left(\frac{1}{2}-\frac{s_{2}+r}{R}\right)} , \end{equation}
 where $\beta$ is the cut-off function defined in \eqref{def_beta}.
 Then we have
\begin{equation}\label{beta_rel.}\beta_{2;R}^2(s_1- 2r)=\left(1-\beta^2(\frac{1}{2}-\frac{s_{1}-r}{R})\right)=1-\beta_{1;R}^2(s_1).
\end{equation}
For any $\eta \in
C^{\infty}(\Sigma_{(r)};u_{(r)}^{\ast}TM\otimes \wedge_{j_{o}}^{0,1}T\Sigma_{(r)})$,
let
\begin{equation}\label{eqn_s.n.}
\eta_{i}(p) =\left\{
\begin{array}{ll}
\eta  & if\;\; p\in \Sigma_{i}(R_{0})\\
\beta_{i;2}(s_{i})\eta(s_{i},t_{i}) & if\;\; p\in \Sigma_{i}(r+1)\setminus \Sigma_{i}(R_{0}) \\
0 & if\;\; p\in\Sigma_{i}\setminus \Sigma_{i}(r+1)
\end{array}
\right..
\end{equation}
Then $\eta_{i}$  can be considered as
a section over $\Sigma_i$, i.e., $\eta_{i}\in C^{\infty}(\Sigma_{i};u_{i}^{\ast}TM\otimes \wedge_{j_{oi}}^{0,1}T\Sigma_{i}).$ If no danger of confusion  we will denotes \eqref{eqn_s.n.} by $\eta_{i}=\beta_{i;2}\eta.$
 Define
\begin{equation}\|\eta\|_{k-1,2,\alpha,r}=\|\eta_{1} \|_{\Sigma_1,k-1,2,\alpha} +
\|\eta_{2} \|_{\Sigma_2,k-1,2,\alpha}.
\end{equation}
We now define a norm $\|\cdot\|_{k,2,\alpha,r}$ on
$C^{\infty}(\Sigma_{(r)};u_{(r)}^{\ast}TM).$ For any section
$h\in C^{\infty}(\Sigma_{(r)};u_{(r)}^{\ast}TM)$ denote
\begin{equation}\label{norm h}
  h_0=\int_{ {S}^1}h(r,t)dt,\end{equation}
\begin{equation}\label{norm h-1}
h_1(s_1,t_1) = (h-\hat h_{0})(s_1,t_1)\beta_{1;2}(s_1),\;\;\;h_2(s_2,t_2)= (h-\hat h_{0})(s_2,t_2) \beta_{2;2}(s_{2}).
\end{equation}
 We define
 \begin{equation}\label{norm-h-2}
 \| h\|_{k,2,\alpha,r}=\|h_1\|_{\Sigma_1,k,2,\alpha} +
\|h_2\|_{\Sigma_2,k,2,\alpha}+|h_{0}|.
\end{equation}
Denote the resulting completed spaces by $$W_{r}^{k-1,2,\alpha}(\Sigma_{(r)};u_{(r)}^{\ast}TM\otimes \wedge_{j_{o}}^{0,1}T\Sigma_{(r)})\;\; \mbox{ and }\;\; W_{r}^{k,2,\alpha}(\Sigma_{(r)};u_{(r)}^{\ast}TM) $$  respectively. To simplify notations we will  denote
$$
 W^{k,2,\alpha}_{r,u_{(r)}} =W^{k,2,\alpha}(\Sigma_{(r)},u_{(r)}^*TM), \;\;\;\;L^{k-1,2,\alpha}_{r,u_{(r)}}=W_r^{k-1,2,\alpha}(\Sigma_{(r)}, u_{(r)}^{*}TM\otimes \wedge^{0,1}_{j_{o}}T^{*}\Sigma_{(r)}).
$$
 Set $\mathbb{D}=\{z=e^{-2r-\sqrt{-1}2\pi\tau} |R_0< r\leq \infty,\;0\leq \tau
\leq 1\}$ and for $(r)\in \mathbb{D}$ denote
$$\mathcal{B}_r =\left\{ v_{r}:\Sigma_{(r)}\to M\;|\;v_{r}=\exp_{u_{(r)}}h_{r},\;h_{r}\in W^{k,2,\alpha}_{r,u_{(r)}} \right\}.$$
For any $R>R_0,$ $\rho>0$ denote
$$
O_{b_{o}}(\mathrm{R},\rho ):= \bigcup_{r\geq R,\tau\in S^1} \left\{\left(e^{-2r-2\pi \sqrt{-1}\tau},v_{r}\right)|\;v_{r}=\exp_{u_{(r)}}h_{r}\in \mathcal B_{r}, \| h_{r}\|_{k,2,\alpha,r}<\rho   \right\}.
$$

\v
\subsection{\bf Local regularization}\label{s_intro_3}
We want to use the implicit function theorem to get $(j,J)$-holomorphic maps from $\Sigma_{(r)}$ to $M$.
When the transversality fails we need to take the "regularization". We explain this now. Fix $a_o=(a_{o1},a_{o2})$, where $a_{oi}=(j_{oi}, {\bf y}_{oi})$.

\v
Let $\bar{\mathcal E}$ be the infinite dimensional Bananch bundle over
$(\mathcal{B}_{1}\times_{q}\mathcal B_{2})\mid_{O_{b_o}(\rho)}$ whose fiber at $b\in O_{b_o}(\rho)$ is
$$\bar{\mathcal E}_{b}:= \left\{\beta(R_{0}-s_{i})\eta(s_{i},t_{i})|\eta \in \mathcal E_1\times \mathcal E_2\right\}.$$
$\bar{\mathcal E}$ can be viewed as a infinite dimensional Bananch bundle over $\mathcal{B}_r\mid_{O_{b_{o}}(\mathrm{R},\rho )}$ for $r>\mathrm{R}>R_0$. Denote by $\inj_{M}$ the injective radius of $(M,G_{J})$. Given $\xi\in W^{k,2,\alpha}_{r,u_{(r)}}$ with $\|\xi\|_{L^{\infty}}<\inj_M$, let
\begin{equation}\label{eqn_Phi}
\Phi_{u_{(r)}}(\xi):u_{(r)}^*TM\rightarrow (\exp_{u_{(r)}}\xi)^*TM \end{equation}
denote the complex bundle isomorphism, given by parallel transport along the
geodesics $s\rightarrow \exp_{u_{(r)}}(s\xi)$ with  respect to the connection $\tilde{\nabla}$.
There is a neighborhood $O_{b_o}(\mathrm{R},\rho)$, over which
$\bar{\E}$ is trivialized. Since $u_{(r)}|_{\Sigma(R_{0})}=u |_{\Sigma(R_{0})},$   there is a isomorphism $ \bar{\Phi}_{b_{o}, b}: \bar{\E}_{b_{o}}\rightarrow  \bar{\E}_{b}$ for any $b\in O_{b_o}(\mathrm{R},\rho)$, where  $\bar{\Phi}_{b_{o}, b}$ is induced by ${\Phi} $. We can choose a finite dimensional subspace $K_{b_{o}}=(K_{b_{o1}}, K_{b_{o2}}) \subset \E|_{b_{o}}=(\E_{b_{o1}},\E_{b_{o2}})$ such that
$$K_{b_{o1}} + \im D_{u_{1}} = \E|_{b_{o1}},\;\;\;K_{b_{o2}} + \im D_{u_{2}} = \E|_{b_{o2}}.$$
By a small perturbation we may assume that
every member of $K_{b_{oi}}$ is $C^{\infty}$ and supports in the compact subset $\Sigma_i(R_0)$ for some large number $R_0$,
Then $K_{b_{o}}$ can be considered a subspace of $L^{k-1,2,\alpha}_{r,u_{(r)}}$ in a natural way.
We define a thickned Fredholm system
$(K_{b_o}\times O_{b_{o}}(R,\rho ), K_{b_o}\times \E|_{O_{b_{o}}(R,\rho )}, \mathcal{S})$ with
\begin{equation}\label{local regu}
\mathcal{S}(\kappa,b) = \bar{\partial}_{j_o,J}v + \mathfrak{i}(\kappa,b), \;\;\;\;\;\;\;\;\kappa\in K_{b_o},
\end{equation}
where $ \mathfrak{i}(\kappa,b) =\bar{\Phi}_{b_o,b}\kappa$ and $b=(a_{o},v)$.
Denote by $D{\mathcal S}_{(\kappa,b)}$ the linearized operator of $\mathcal{S}$ at $(\kappa,b)$. Then
$$D{\mathcal S}_{(0,b_o)}|_{K_{b_o}\times  {W}^{k,2,\alpha}}:K_{b_o}\times
W^{k,2,\alpha}_{u}\to L_{u}^{k-1,2,\alpha}$$ is surjective.
Let $Q_{(0,b_o)}$ be a right inverse of $D{\mathcal S}_{(0,b_o )}.$
We call $(\kappa, b)$ a perturbed $(j_o,J)$-holomorphic map, if $(\kappa, b)$ satisfies $\mathcal{S}(\kappa,b)=0$. If no confusion, we denote $K_{b_{o}}$ by $K.$ Let $(\kappa_o,b_o)$ be a perturbed $(j_o,J)$-holomorphic map. Denote by $D{\mathcal S}_{(\kappa_o,b_o)}$ the linearized operator of $\mathcal{S}$ at $(\kappa_o,b_o)$, by $Q_{(\kappa_{o},b_o)}$ a right inverse of $D{\mathcal S}_{(\kappa_{o},b_o )}.$ Denote by $D{\mathcal S}_{(\kappa_o,b_{(r)})}$ the linearized operator of $\mathcal{S}$ at $(\kappa_o,b_{(r)})$.

\subsection{Some operators}
\v
Given
$\eta\in L^{k-1,2,\alpha}_{r,u_{(r)}}$ denote
\begin{equation}\label{def. of Q'-1}
\left(\eta_1(s_1,t_1),\eta_2(s_2,t_2)\right) = \left(\beta_{1;2}(s_1)\eta(s_1,t_1), \beta_{2;2}(s_2)\eta(s_2,t_2)\right),\end{equation}
\begin{equation}\label{def. of Q'-2}
Q_{(\kappa_{o},b_o)}(\eta_1,\eta_2)=(\kappa,h)=(\kappa,(h_1,h_2)),\;\;h_{i}\in W^{k,2,\alpha}(\Sigma_{i};u_{i}^{\ast}TM ),\end{equation}
where $\eta(s_i,t_i)$ denote the expression of $\eta$ in terms of the coordinates $(s_i,t_i)$.
We define  $h_{(r)}\in W^{k,2,\alpha}_{r,u_{(r)}}$ by
\begin{equation}\label{def_h}
h_{(r)}=\beta_{1;r}(s_1)h_{1}(s_1,t_1)+\beta_{2;r}(s_{1}-2r)h_{2}(s_{1}-2r,t_1-\tau).
\end{equation}
Note that, in the part $\{|s_i|\geq \frac{r}{2}\}$, $h_1$ and $h_2$ are identified as vectors in $T_{u(q)}M$, so the meaning of definition \eqref{def_h} is clear. Since $\mathfrak{i}(\kappa, b)$ supports in $\Sigma(R_0)$ for all $\kappa\in K$ and  $$u_{(r)}|_{\{ s_1 \leq \frac{r}{2}\}}=u_1|_{\{ s_1 \leq \frac{r}{2}\}},\;\;\;\;\;u_{(r)}|_{\{ |s_2|\leq \frac{r}{2}\}}=u_2|_{\{|s_2|\leq \frac{r}{2 }\}},$$
we have $ \mathfrak{i}(\kappa,b_{(r)}) = \mathfrak{i}(\kappa ,b )$ along $u_{(r)}$. Then we define an approximate right inverse
\begin{equation}
\label{def_approximate_right_inverse}
Q_{(\kappa_{o},b_{(r)})}^{\prime}\eta =(\kappa ,h_{(r)}).\end{equation}
It is easy to show that
$D{\mathcal S}_{(\kappa_{o}, b_{(r)})}\circ  Q'_{(\kappa_{o},b_{(r)})}$ is invertible when $r$ big enough ( cf. the proof of Lemma \ref{aright_inverse_after_gluing} ). Then a right inverse $ Q_{(\kappa_{o},b_{(r)})}$  of $D{\mathcal S}_{(\kappa_{o}, b_{(r)})}$ is given by
 \begin{equation}
 \label{express_right_inverse}
 Q_{(\kappa_{o},b_{(r)})}= Q'_{(\kappa_{o},b_{(r)})}( D{\mathcal S}_{(\kappa_{o}, b_{(r)})}\circ  Q'_{(\kappa_{o},b_{(r)})})\inv.
 \end{equation}
 For a fixed gluing parameter $(r)=(r,\tau)$
we define a map
$$I_r: \ker D \mathcal{S}_{( \kappa_{o}, b_{o})}\longrightarrow \ker D \mathcal{S}_{( \kappa_{o},b_{(r)})}$$ as follows. For any $(\kappa,h+\hat h_{0})\in \ker D \mathcal {S}_{(\kappa_o,b_o)},$
 where $h=(h_1 ,h_2  )\in W^{k,2,\alpha}_u,$ we set
\begin{equation}
\label{definition_h_ker}
 h_{(r)}=\beta_{1;r}(s_1)h_{1}(s_1,t_1)+\beta_{2;r}(s_{1}-2r)h_{2}(s_{1}-2r,t_{1}-\tau)+\hat h_{0},
\end{equation}
and define
\begin{align}\label{definition_I}
&I_{r}(\kappa,h+\hat h_{0})=(\kappa,h_{(r)})- Q_{(\kappa_{o},b_{(r)})}\circ D\mathcal{S}_{(\kappa_o,b_{(r)})}(\kappa,h_{(r)}).
\end{align}
It is easy to see that $I_{r}(\kappa,h+\hat h_{0})\in \ker D \mathcal{S}_{(\kappa_o,b_{(r)})}$. We can prove that when $r$ large enough  $I_r$ is an isomorphism ( cf \S\ref{isomorphism-1} ).
 \v
It is important to estimate the derivative of the gluing map with respect to $r$. To this end we need to take the derivative $\frac{\p}{\p r}$ for $ Q_{(\kappa_{o},b_{(r)})}$ and other operators. Note that both $Q_{(\kappa_{o},b_{(r)})}$ and $f_{(r)}$ are global operators, so we need a global estimate. On the other hand,
since the domain $\Sigma_{(r)}$ depends on $r$, in order to make the meaning of the derivative $\frac{\p}{\p r}$ for these operators clear we need transfer all operators defined over $\Sigma_{(r)}$ into a family of operators  defined over $\stackrel{\circ}{\Sigma}_1\cup \stackrel{\circ}{\Sigma}_2$, depending on $(r)$.
We first define three maps
$$H_{r}:L^{k-1,2,\alpha}_{r,u_{(r)}}\to L^{k-1,2,\alpha}_u,\;\;\;\;P_{r}:L^{k-1,2,\alpha}_u\to L^{k-1,2,\alpha}_{r,u_{(r)}},\;\;\;\phi_{r}:\mathcal W^{k,2,\alpha}_u\to W^{k,2,\alpha}_{r,u_{(r)}}$$
as following. Given $\eta\in L^{k-1,2,\alpha}_{r,u_{(r)}}$ define
$$H_{r}\eta=(\beta_{1;2}(s_1)\eta(s_1,t_1),\;\beta_{2;2}(s_2)\eta(s_2,t_2)),$$
where $\eta(s_{i},t_i)$ is the expression of $\eta$ in terms the coordinates $(s_{i},t_{i}).$
Given
$(\eta_1,\eta_2)\in L^{k-1,2,\alpha}_u$ define
\begin{equation} \label{eqn_s.n.1}
P_{r}(\eta_{1},\eta_{2})=\left\{
\begin{array}{ll}
\eta_{i} & if\;\; p\in  \Sigma(r/2) \\
\beta_{1;2}(s_1)\eta_{1}(s_1,t_1)+\beta_{2;2}(s_{1}-2r) \eta_{2}(s_{1}-2r,t_1-\tau) & if\;\; p\in  \Sigma_{(r)}\setminus \Sigma(r/2)
\end{array}
\right..
\end{equation}
  If no danger of confusion  we will denotes \eqref{eqn_s.n.1} by $P_{r}(\eta_{1},\eta_{2})=\sum\beta_{i;2}\eta_{i}.$
\v
Given
$(h_1+\hat h_0,h_2+\hat h_0)\in \mathcal  W^{k,2,\alpha}_u$ with $\supp\; h_{i}\subset \Sigma(3r/2)$, define
\begin{align*}\left.\phi_{r}\left(h_1+\hat h_0,h_{2}+\hat h_0\right)\right|_{\Sigma_i(r/2)}&=\left.\left(h_{i}+\hat h_0\right)(s_i,t_i)\right|_{\Sigma_i(r/2)},\\
\left.\phi_r\left(h_1+\hat h_0,h_2+\hat h_0\right)\right|_{ \frac{r}{2} \leq s_{1} \leq \frac{3r}{2}}&=\left.\left(h_1(s_1,t_1)+ h_2(s_1- 2r,t_1-\tau)+\hat h_0\right)\right|_{ \frac{r}{2} \leq s_{1} \leq \frac{3r}{2}}.
\end{align*}
By \eqref{beta_rel.} one can check that
\begin{equation}\label{equ.PH}
P_{r}  H_{r} =Id,\;\;\;\;\;\; H_{r}  P_{r} (\eta_{1},\eta_{2})=(\tilde \xi_{1},\tilde \xi_{2}).
\end{equation}
where
$$\tilde \xi_{1}=\beta_{1;2}(\beta_{1;2}\eta_1(s_1,t_1)+\beta_{2;2}\eta_{2}(s_1-2r, t_1-\tau)),\;\;\;\tilde \xi_{2}=\beta_{2;2}(\beta_{1;2}\eta_1(s_2+2r,t_2+\tau)+\beta_{2;2}\eta_{2}(s_2, t_2)).
$$
In particular, $H_{r}$ is injective and $P_{r}$ is surjective.
\v
Next we introduce the following three operators
$$(Q^{\prime})^{*}_{(\kappa_{o},b_{(r)})}:L^{k-1,2,\alpha}_{r,u_{(r)}}\to W^{k,2,\alpha}_u,\;\;\;\;Q^{*}_{(\kappa_{o},b_{(r)})}:L^{k-1,2,\alpha}_{r,u_{(r)}}\to W^{k,2,\alpha}_u,\;\;\;I^{*}_{r}:\ker D \mathcal {S}_{(\kappa_o,b_o)}\to  K_{b_{o}}\times \mathcal W^{k,2,\alpha}_u .$$
Given
$ \eta \in \lkar$, denote
\begin{equation} (\kappa ,(h_1,h_2))=Q_{( \kappa_{o}, b_{o})}H_r\eta
.\end{equation}
Set
\begin{equation}
h_{r}^{*}=(\beta_{1;r}(s_1)h_{1}(s_1,t_1),\;\beta_{2;r}(s_2)h_{2}(s_2,t_2))\in \wka.
\end{equation}
Define
\begin{align}\label{eqn_def_Q*}
&(Q^{\prime})^{*}_{(\kappa_{o},b_{(r)})}\eta =  (\kappa ,
h_r^*),\;\;\;\; Q^*_{(\kappa_{o},b_{(r)})}=(Q^{\prime})^{*}_{(\kappa_{o},b_{(r)})}  (D\mathcal S_{(\kappa_o,b_{(r)})}  Q^{\prime}_{(\kappa_{o},b_{(r)})})^{-1}.
\end{align}
Then we have maps
$$(Q^{\prime})^{*}_{(\kappa_{o},b_{(r)})}  P_r:L^{k-1,2,\alpha}_u\to W^{k,2,\alpha}_u,\;\;\;\;\;
Q^{*}_{(\kappa_{o},b_{(r)})}  P_r: L^{k-1,2,\alpha}_u\to W^{k,2,\alpha}_u.$$

\v\n
For any $(\kappa,\zeta+\hat \zeta_{0})\in \ker D \mathcal {S}_{(\kappa_o,b_o)},$
 where $\zeta=(\zeta_1 ,\zeta_2  )\in W^{k,2,\alpha}_u,$   we set
\begin{equation}
\label{definition_h_ker}
\zeta_{r}^{*}= (\zeta_{1}\beta_{1;r}+\hat{\zeta}_{0},\; \zeta_{2}\beta_{2;r}+\hat{\zeta}_{0} ).
\end{equation} Define
\begin{align}\label{def_I*}
&I^{*}_{r}(\kappa,\zeta+\zeta_{0})=(\kappa,\zeta_{r}^*)-Q^{*}_{(\kappa_{o},b_{(r)})}\circ D\mathcal{S}_{(\kappa_o,b_{(r)})}(Id,\phi_{r})(\kappa,\zeta^*_{r}).
\end{align}
By the definition we have
$$
I_{r}=(Id,\phi_{r})\circ I_{r}^{*},\;\;\;\;\;\;Q_{( \kappa_{o},b_{(r)})}=(Id,\phi_{r})\circ Q_{( \kappa_{o},b_{(r)})}^{*}.
$$
where  $Id$ identifies $\bar{\mathcal E}_{b_{o}}$ with $\bar{\mathcal E}_{b_{(r)}}.$
 \v For fixed $(r)$ we consider the family of maps:
\begin{align*}& \mathcal{F}_{(r)}: K  \times W^{k,2,\alpha}_{r,u_{(r)}}
\rightarrow L^{k-1,2,\alpha}_{r,u_{(r)}},\;\;\;\mathcal{F}_{(r)}(\kappa,h)=P_{b, b_{(r)}}\left(\bar{\partial}_{j_{o},J}v
+ \mathfrak{i}(\kappa,b)\right),
\end{align*}
where $b=(r,\tau,a_{o}, v),\;\;v=\exp_{u_{(r)}}h$  and
\begin{equation*}
P_{b, b_{(r)}}=\Phi_{u_{(r)}}(h)^{-1}:W_{r}^{k-1,2,\alpha}(\Sigma_{(r)}, v^{*}TM\otimes \wedge^{0,1}_{j_{o}}T^{*}\Sigma_{(r)})\to \lkar.
\end{equation*}
By implicit function theorem, there exists a small neighborhood $O_{r}$ of $0 \in \ker\;D{\mathcal S}_{(\kappa_{o},b_{(r)})}$ and a unique smooth map
$$f_{(r)}: O_{r}\rightarrow L^{k-1,2,\alpha}_{r,u_{(r)}}$$ such that for any $(\kappa,\zeta)\in O_{r},$
$$ dv+J\cdot dv\cdot j_{a_{o}}+ \mathfrak{i}\left(\kappa_{o}+\kappa_{r},b\right)=0,$$
where $b=(r,\tau,a_{o},v),$ $v=\exp_{u_{(r)}}(\xi_{(r)})$ and
$(\kappa_{r},\xi_{(r)})=(\kappa,\zeta) + Q_{( \kappa_{o},b_{(r)})}\circ f_{(r)}(\kappa,\zeta).$
Then combining with the isomorphism $I_r: \ker D \mathcal{S}_{( \kappa_{o}, b_{o})}\longrightarrow \ker D \mathcal{S}_{( \kappa_{o},b_{(r)})}$ we
get a gluing map
$$I_{r}+Q_{( \kappa_{o},b_{(r)})}\circ f_{(r)} \circ I_{r}$$
from $O$ into the  space of  perturbed $(j_{a_o},J)$-holomorphic maps, where $O$ is a neighborhood of $0$ in $\ker D \mathcal{S}_{(\kappa_o,b_o)}$.
We consider the operator
$$I^*_r +Q^*_{( \kappa_{o},b_{(r)})}f_{(r)}I_r: \ker D \mathcal{S}_{( \kappa_{o}, b_{o})}\to
K\times  {\mathcal{W}}^{k,2,\alpha}_u.$$
It is easy to see that, restricting to $\Sigma(R_0)$, we have $$I^*_r(\kappa,\zeta)+Q^*_{( \kappa_{o},b_{(r)})}f_{(r)}(I_r(\kappa,\zeta))=I_r(\kappa,\zeta)+Q_{( \kappa_{o},b_{(r)})}f_{(r)}(I_r(\kappa,\zeta)).$$

\v
\subsection{\bf Main result}\label{s_main_thm}

 The main result  of this paper is the following estimate.
\begin{theorem}\label{lem_est_rI}
Let $\Sigma$ be a marked nodal Riemann surface with one nodal point $q$. Let $(\kappa_o, b_o)=(\kappa_o, a_o, u)\in K\times O_{b_o}(\rho)$ be a perturbed $(j_o,J)$-holomorphic map from $\Sigma$ to $M$, where $u=(u_{1},u_{2}):\Sigma_1\cup \Sigma_2\to M$ with $u_{1}(q)=u_{2}(q)$ and
$$D{\mathcal S}_{(\kappa_o,b_o)}|_{K\times  \wka} :K\times  \wka \rightarrow \lka $$ is surjective.
Denote by $Q_{(\kappa_{o},b_o)}:\lka \rightarrow K \times \wka$ a right inverse of $D{\mathcal S}_{(\kappa_o,b_o)} .$ Let $ Q_{(\kappa_{o},b_{(r)})}$ be the right inverse of D${\mathcal S}_{(\kappa_o,b_{(r)})}$ defined in \eqref{express_right_inverse}. Then the following hold.
\v
Let $\fc\in (0,1)$ be a fixed constant. For any $0<\alpha<\frac{1}{100\fc}$,  there exists two positive  constants  $\mathsf C_{1}, \mathsf{d}$ such that for any $(\kappa,\zeta)\in \ker D \mathcal{S}_{(\kappa_{o},b_{o})}$ satisfying $\|(\kappa,\zeta)\|_{\mathcal W,k,2,\alpha}\leq \mathsf{d}$, we have the following estimate
 $$
\left \|\tfrac{\partial }{\partial r}\left(I^{*}_{r}(\kappa,\zeta)+Q^{*}_{( \kappa_{o},b_{(r)})}\circ f_{(r)} \circ I_{r}(\kappa,\zeta) \right) \right \|_{k-1,2,\alpha} \leq  \mathsf C_{1}e^{-(\fc-5\alpha)\tfrac{r}{4} }(\mathsf{d}+1),
$$
when $r$ large enough, where $\|\cdot\|_{ k,2,\alpha}$ is defined in \eqref{norm}.
\end{theorem}
\vskip
0.1in \noindent
As a consequence we have
\begin{corollary}\label{coordinate_decay}
 Let $l\in \mathbb Z^+$ be a fixed integer. There exists positive  constants  $\mathsf C_{o,l}, \mathsf{d}, R_{0}$ such that for any $(\kappa,\zeta)\in \ker D \mathcal{S}_{(\kappa_{o},b_{o})}$ with $\|(\kappa,\zeta)\|_{\mathcal W,k,2,\alpha}\leq \mathsf{d}$,
restricting to the compact set $\Sigma(R_0)$, the following estimate hold
 $$
\left \|\tfrac{\partial }{\partial r}\left(I_{r}(\kappa,\zeta)+Q_{( \kappa_{o},b_{(r)})}\circ f_{(r)} \circ I_{r}(\kappa,\zeta) \right) \right \|_{C^{l}(\Sigma(R_0))} \leq  \mathsf C_{o,l}e^{-(\fc-5\alpha)\tfrac{r}{4} }(\mathsf{d}+1).
$$
\end{corollary}
\vskip
0.1in \noindent

\v
In order to apply our result to the study of $J$-holomorphic map moduli spaces we should let $a=(a_1,a_2)$ vary in $\mathbf{A}_1\times \mathbf{A}_2$ and need to take a sum of several $K_{b_i}$. In the section \S\ref{Extension} we extend the Theorem \ref{lem_est_rI} to more general setting and to a neighborhood of $a_o$.
In our next paper \cite{LS-1} we use the Theorem to show that the Gromov-Witten invariants can be defined as an integral over top strata of virtual neighborhood. Furthermore, we prove that such invariants
satisfy all the Gromov-Witten axioms of Kontsevich and Manin.
\v

\v

\section{\bf Some important estimates}\label{s_est}

In this section we give some important estimates which will be used in this paper.
\subsection{\bf Exponential decay theorem for $J$-holomorphic maps}\label{s_est_1}

Denote $$B_{r}(0)=\{z\in \mathbb C|\; |z|\leq r\},\;\;\;\;\;\;\;\;A(r,R):=\{z\in \mathbb C|\; r\leq |z|\leq R\}.$$
Take the standard Euclidean metric on $B_{r}(0)$ and $A(r,R).$
\v
The following two theorems in \cite{MS} (see Lemma 4.3.1 and Lemma 4.7.3) are  fundamental results:
\begin{theorem}\label{theorem_2.1} Let $(M,J)$ be a compact almost complex manifold. Suppose that $M$ is equipped with any
Riemannian metric. Then there exists a constant $\hbar> 0$ such that the following
holds.
If $r > 0$ and $u : B_{r}(0) \to M$ is a $J$-holomorphic curve then
$$\int_{B_{r}(0)} |du|^2<\hbar\;\;\;\Rightarrow \;\;\;\;|du(0)|^2 < \frac{8}{\pi r^2}\int_{B_{r}(0)} |du|^2.$$
\end{theorem}
\v

\begin{theorem}\label{theorem_2.2}
Let $(M,\omega)$ be a compact symplectic manifold and $J$ be an  $\omega$-tame almost complex structure.
Then, for every $\fc<1$, there exist positive constants $\hbar = \hbar(M, \omega, J, \nu)$ and $\mathsf C= \mathsf C(\fc)$
such that every $J$-holomorphic curve $u : A(r, R) \to M$ with
$$
 E(u,A(r,R))<\hbar,
$$
satisfies for any $\log 2\leq T\leq \log \sqrt{R/r}$,
$$
E(u,A(e^{T}r,e^{-T}R))\leq \mathsf C e^{-2\fc T} E(u,A(r,R)).
$$
\end{theorem}
\v

\v
Now we use the cylinder coordinates $(s,t)$. Fix a constant $R_{0}\in (0,r/4)$. Take the standard complex structure $j$ and a smooth metric $\mathbf g$ on $(R_{0},2r-R_{0})\times S^1$ such that
 $$
 \mathbf g= ds^2+dt^2,\;\;\;\;\;\mbox{ in }\;R_{0}\leq s\leq 2r-R_{0}.
 $$
We can restate Theorems \ref{theorem_2.1} and \ref{theorem_2.2} as following.
\v

\begin{theorem}\label{tube_exponential_decay}
Fix a constant  $\fc\in (0,1)$. Let $(M,\omega)$ be a compact symplectic manifold and $J$ be an  $\omega$-tame almost complex structure.
Then  there exist positive constants $\hbar = \hbar(M, \omega, J, \nu)$ and $\mathsf C_1= \mathsf C_1(\fc)$
such that every    $(j,J)$-holomorphic map  $u:[R_{0},2r-R_{0}]\times S^1\rightarrow  M$ with
\begin{equation}
 E (u,R_{0}\leq s\leq 2r-R_{0}):=\int_{[R_{0},2r-R_{0}]\times S^1} |du|^2 <\hbar,
\end{equation}
 and any $   R_{0}+\log 2 \leq R\leq r,$ we have
 \begin{align*}
&E(u,R\leq s\leq 2r-R)\leq \mathsf{C}_{1}E(u,R_{0}\leq s\leq 2r-R_{0})e^{-2\fc(R-R_{0})},\;\; \\
&\left|\frac{\p u}{\p s} (s,t)\right|+\left|\frac{\p u}{\p t} (s,t)\right|\leq \mathsf{C}_{1}\sqrt{E(u,R_{0}\leq s\leq 2r-R_{0}})e^{-\fc(R-R_{0})},\;\;  \forall\; R+\frac{1}{2}\leq s\leq 2r-R-\frac{1}{2}.
\end{align*}
 \end{theorem}
\v
If we take $\mathbb R^{2m}$ instead of $M$, a similar result also hold. Let $\omega_{o}$ be the standard symplectic form in $\mathbb R^{2m}$.
Let  $u:[R_{0},2r-R_{0}]\times S^1\to \mathbb R^{2m}$ be a map  satisfying
$du+J_{0}\cdot du\cdot j= 0$. Denote $$e(R_{0}):=E\left(u,R_{0}\leq s\leq 2r-R_{0}\right):=\int_{[R_{0},2r-R_{0}]\times S^1} u^{*}\omega_{o}=\int_{[R_{0},2r-R_{0}]\times S^1} |\nabla u|^2ds dt$$ where  $|\nabla u|^2=\sum_{i=1}^{2m} \left(\frac{\p u^{i}}{\p s}\right)^2+\sum_{i=1}^{2m} \left(\frac{\p u^{i}}{\p t}\right)^2.$ We have
\begin{lemma}\label{decay_kernel_tube}
There is a constant
$\mathsf C$ depending only on    $\fc$ such that
for all  solution $u:[R_{0},2r-R_{0}]\times S^1\rightarrow \mathbb R^{2m}$ of the equations $du+J_{0}\cdot du\cdot j= 0$, and any $R_{0}+\log 2 \leq R\leq r,$ we have,
\begin{align}\label{Ee_decay}
 e(R)&\leq e^{-2\fc(R-R_{0})} e(R_{0}),\;\;\;\\ \label{Eu_decay}
\left|\frac{\p u}{\p s} (s ,t )\right|+\left|\frac{\p u}{\p t} (s ,t )\right|&\leq \mathsf{C}\sqrt{e(R_{0})}e^{-\fc(R-R_{0})},\;\;  \forall\; R+\frac{1}{2}\leq s \leq 2r-R-\frac{1}{2}.
\end{align}
\end{lemma}
\begin{proof} We give a sketch of the proof. For any loop
$\gamma: S^1\rightarrow \mathbb R^{2m}$ and any smooth map  $W: D_{1}(0)\rightarrow \mathbb R^{2m}$
satisfying $W(\p D_{1}(0) )=\gamma$,
we define an action functional by
\begin{equation*}
{\mathcal A}( \gamma)=-\int_{D_{1}(0)}W^* \omega_{o}.
\end{equation*}
Since $\omega_{o}$ is exact in $\mathbb R^{2m},$ the action functional  ${\mathcal A}( \gamma)$ is well defined. Denote by $L(\gamma)$ the length of $\gamma.$
Then isoperimetric inequality can be written as (see Page 85 of \cite{MS})
\begin{equation}\label{eqn_isop_ineq}
|\mathcal A(\gamma)|\leq \frac{1}{4\pi } L(\gamma)^2.
\end{equation}
Denote $u_{s}:S^{1}\to \mathbb R^{2m}$ given by $u_{s}(t)=u(s,t).$
 Then
 $$e(R)=|\mathcal{A}(u_{R})-\mathcal{A}(u_{2r-R})|.$$
By the same argument of \cite{MS} (see Page 105 of \cite{MS}), using inequalities \eqref{eqn_isop_ineq}    one can prove \eqref{Ee_decay}.
It is well known  that $|\nabla u|^2$ is subharmonic function. Then
\begin{equation}\label{eqn_subh}
|\nabla u|^2(p)\leq \frac{4}{\pi}\int_{D_{\frac{1}{2}}(p)} |\nabla u|^2,\;\;\;\; \forall \; p \in [R_{0}+1,2r-R_{0}-1]\times S^1 .
\end{equation}
Then \eqref{Eu_decay} follows from  \eqref{eqn_subh}.
\end{proof}

\begin{remark}
Theorem \ref{tube_exponential_decay} and Lemma \ref{decay_kernel_tube} also hold for $r=\infty.$
\end{remark}

\v

\subsection{Estimates for the equation $\bar{\p}_{j,J_{0}}\zeta=\chi$}\label{s_est_2}

Fix $\alpha \in (0,\frac{\fc}{100}).$   The multiplication by $e^{\alpha s}$ gives an isomorphism from $ W^{k,2,\alpha} (\mathbb R\times S^1;\mathbb R^{2m})$ to $W^{k,2}(\mathbb R\times S^1;\mathbb R^{2m} ) $
and
\begin{equation}
\label{eqn_norm_rel}
 \mathsf C\inv\|e^{\alpha s}f\|_{W^{k,2} (\mathbb R\times S^1;\mathbb R^{2m})}\leq \|f\|_{W^{k,2,\alpha} (\mathbb R\times S^1;\mathbb R^{2m})}\leq \mathsf C\|e^{\alpha s}f\|_{W^{k,2} (\mathbb R\times S^1;\mathbb R^{2m})},
\end{equation}
for some constant $ \mathsf C>0$ depending only on $k$ and $ \alpha .$
It is easy to check that
\begin{equation}\label{eqn_sol_rel}
 \bar{\p }_{j,J_{0}} h=\eta, \mbox{ if and only if }(\bar{\p }_{j,J_{0}}-\alpha) (e^{\alpha s} h)=e^{\alpha s}\eta.
\end{equation}
Obviously, $L:=J_{0}\frac{\partial }{\partial t} -\alpha$ is an invertible elliptic operator on $H^{1}(S^1)$. It is well known that  the operator
  $\bar{\p}_{j,J_{0}}-\alpha:W^{k,2 }(\mathbb R\times S^1,\mathbb R^{2m})\rightarrow W^{k-1,2 }(\mathbb R\times S^1,\mathbb R^{2m})$  has a   right inverse $$Q_{\alpha}:W^{k-1,2 }(\mathbb R\times S^1,\mathbb R^{2m})\to W^{k,2 }(\mathbb R\times S^1,\mathbb R^{2m})$$ with
 $$\|Q_{\alpha}\rho\|_{k,2}\leq \mathsf C(k,\alpha)\|\rho \|_{k-1,2},\;\;\;\forall \rho\in W^{k-1,2 }(\mathbb R \times S^1,\mathbb R^{2m})$$
 (see Proposition 3.4 in \cite{D}).
  Denote $$\mathbb R^{+}=\{ x\in \mathbb R| x\geq 0\},\;\;\;\mathbb R^{-}=\{ x\in \mathbb R| x\leq 0\}.$$
Denote the completed spaces of $C^{k-1,2}_{c}(\mathbb R^+ \times S^1,\mathbb R^{2m})$ with respect to the norm $\|\cdot\|_{k-1,2,\alpha}$ by $W^{k-1,2,\alpha}_{c}(\mathbb R^+ \times S^1,\mathbb R^{2m}).$
For any $\eta\in W^{k-1,2,\alpha}_{c}(\mathbb R^{+} \times S^1,\mathbb R^{2m})$, $e^{\alpha s} \eta$ can been seen as in $W^{k-1,2 }(\mathbb R \times S^1,\mathbb R^{2m})$ by extending  $e^{\alpha s} \eta|_{\mathbb R^{-} \times S^1}=0.$ Set $$Q_{J_{0},\alpha}\eta=\left.e^{-\alpha s}Q_{\alpha} \left(e^{\alpha s}\eta\right)\right|_{\mathbb R^{+} \times S^1}.$$
Then by \eqref{eqn_norm_rel} and  \eqref{eqn_sol_rel} we conclude that     $$\bar{\p}_{j,J_{0}} \left(Q_{J_{0},\alpha}\eta\right)=\eta,$$ and $Q_{J_{0},\alpha}\eta\in W^{k,2,\alpha}(\mathbb R^+ \times S^1,\mathbb R^{2m})$ with
\begin{equation}\label{eqn_2.9}\|Q_{J_{0},\alpha}\eta\|_{k,2,\alpha}\leq \mathsf C(k, \alpha )\|\eta \|_{k-1,2,\alpha},\;\;\;\forall \eta \in W_{c}^{k-1,2,\alpha}(\mathbb R^{+}  \times S^1,\mathbb R^{2m}).
\end{equation}
Then we have the following lemma.
\begin{lemma}\label{tube_ker_c-1}
Let $\eta\in \lka$ and $h +\hat h_{0}\in \cwk$ be a solution of $D_{u}(h+\hat h_{0})=\eta$ over $\Sigma\setminus \Sigma(R_{0})$.
Suppose that, for any $p,q\geq 0,$
\begin{equation}\label{exp_decay_Ss}
 \left|\frac{\p^{p+q} F^{l}_{u} }{\partial s_{1}^ p \partial t_{1} ^q}\right|\leq
C_{p,q}e^{-\fc |s_i|} ,\;\;\;\;\;\forall \; |s_i|\geq R_0,\;\; l=1,2
\end{equation} for some constant $C_{p,q}>0.$
Then  for any $0<\alpha< \frac{\fc}{2} $, there exists a constant $\mathsf{C} >0$ such that for any $R>\max\{R_{0},\bar d\}$ and  $R'>2+R$
 \begin{align}\label{t_ker_1}
  \left\|h \mid_{s_{1}\geq R' }\right\|_{k,2,\alpha} \leq \mathsf{C} \left( (e^{-(\fc-\alpha)(R'-  R)}  +  e^{-( \fc -\alpha) R} ) \left\|h+\hat h_{0}  \right\|_{\mathcal W,k,2,\alpha } + \left\|\eta\mid_{s_{1}\geq R }  \right\|_{k-1,2,\alpha }  \right)
  \end{align}
 In particular, if  $D_{u}$ has a bounded right inverse $Q_{u}:\lka\to \wka$ let $h=Q_{u}\eta$ be a solution of $D_{u}h = \eta$ over $(R_0,\infty)\times S^1$.   Then there exists a constant $\mathsf{C}' >0$ independent of $r$ such that
 	\begin{equation} \label{f_tube_estimate_curve}
	\left\|h\mid_{s_{1}\geq R'}\right\|_{k,2,\alpha} \leq \mathsf{C}' \left[\left({e^{-(\fc-\alpha)(R'-  R)}}+e^{-(\fc-\alpha) R }\right) \|\eta \|_{k-1,2,\alpha}+ \left\|\eta\mid_{  s_{1} \geq  R }  \right\|_{k-1,2,\alpha}\right].
	\end{equation}
\end{lemma}
\begin{proof} First we construct $\zeta$ such that, restricting to $  s  \geq  R+1$, \begin{equation}\label{eqn_a_zeta}
\bar{\p}_{j,J_{0}}(h -\zeta)=0.\end{equation}
   Denote  $\chi=\beta(s_{1}-R)\left(\eta-F_{u}^{1}(h+\hat h_{0})-F_{u}^{2}\frac{\p h}{\p t}\right) .$ Obviously, $\supp\;\chi \subset \{ s\geq R \} $ and $$\chi=\eta-F_{u}^{1}(h+\hat h_{0})-F_{u}^{2}\frac{\p h}{\p t},\;\;\;\;\;\mbox{ for any }   s\geq R+1.$$ Let $\zeta=  Q_{ J_{0},\alpha}(\chi).$ Then $\zeta$ satisfies \eqref{eqn_a_zeta} in $s\geq R+1.$ By \eqref{eqn_2.9} and using the the exponential decay of $F^{i}_{u},i=1,2$  we get
  \begin{equation}\label{eqn_zeta_o}
  \|\zeta  \|_{k,2,\alpha} \leq  C\|\chi\|_{k-1,2,\alpha}\leq C\left(e^{-(\fc-\alpha)R}  \left\|h +\hat h_{0} \right\|_{k,2,\alpha}+\|\eta|_{  s\geq  R}\|_{k-1,2,\alpha}\right).
  \end{equation}
     Since for any nonnegative integers $p,q$,
   \begin{equation}
   \bar{\p}_{j,J_{0}}\left(\frac{\p^{p+q}}{\p^{p} s\p^{q}t }(h -\zeta)\right)=0,\;\;\;\mbox{ in  }s\geq R+1
   \end{equation}
   using  Lemma \ref{decay_kernel_tube} with $r=\infty$   we  conclude that for any $l\leq k$
   $$
   |\p^{l}(h-\zeta)|(s,t)\leq C e^{-\fc(s-R)}  \|(h-\zeta)\|_{l,2} \leq C e^{-\fc(s-R)} \|h-\zeta\|_{k,2,\alpha},\;\;\;\;\forall s\geq R+2.
   $$
   By $\lim\limits_{s\to \infty} (h-\zeta)=0$ and  the integration  with respect to s, we have
  \begin{equation}\label{eqn_zeta_ao}
 \left\|(h-\zeta)\mid_{s\geq R'}\right\|_{k,2,\alpha}  \leq C e^{-(\fc-\alpha)(R'-R)} \|h-\zeta\|_{k,2,\alpha}.
\end{equation}
Then by \eqref{eqn_zeta_o} and \eqref{eqn_zeta_ao} we get
 \begin{align*}
&\left\|h \mid_{  s\ge R'}\right\|_{k,2,\alpha} \leq \left\|(h-\zeta)\mid_{s\geq R'}\right\|_{k,2,\alpha}+ \|\zeta |_{s\geq R'}\|_{k,2,\alpha} \\
& \leq C {e^{-(\fc-\alpha)(R'-  R)}}\left\|h-\zeta|_{s\geq R}  \right\|_{2,\alpha}+ C \left(e^{-( \fc -\alpha) R}  \left\|h+\hat h_{0}  \right\|_{k,2,\alpha}  +\|\eta|_{s\geq R}\|_{k-1,2,\alpha} \right) \\
& \leq C\left( {e^{-(\fc-\alpha)(R'-  R)}} +  e^{-( \fc -\alpha) R} \right)  \left\|h +\hat h_{0}   \right\|_{k,2,\alpha}+C\|\eta|_{s\geq R}\|_{k-1,2,\alpha}.
 \end{align*}
 We obtain the estimate \eqref{t_ker_1}.
\end{proof}

\section{\bf Gluing estimates}\label{gluing_analysis}\label{s_glu}
\v

Let $(\kappa_o, b_o)=(\kappa_o, a_o, u)=(\kappa_o, j_o, {\bf y}_o, u)$ be a perturbed $(j_o,J)$-holomorphic map from $\Sigma$ to $M$.
For any $(\kappa,h+\hat h_{0})\in K\times \mathcal{W}^{k,2,\alpha}_u,$
 where $h=(h_{1},h_{2})\in W^{k,2,\alpha}_u$, we define
\begin{equation} \label{norm}
\|(\kappa,h)\|_{k,2,\alpha}=|\kappa|+\|h_1\|_{k,2,\alpha} + \|h_2\|_{k,2,\alpha},\;\;\;\;
\|(\kappa,h+\hat h_{0})\|_{\mathcal{W},k,2,\alpha}=\|(\kappa,h)\|_{k,2,\alpha}+|h_{0}|.\end{equation}
 For any $(\kappa,h_{(r)})\in K\times W_{r,u_{(r)}}^{k,2,\alpha}$
we define
\begin{equation} \label{norm-1}
\|(\kappa,h_{(r)})\|_{k,2,\alpha,r}=|\kappa|+\|h_{(r)}\|_{k,2,\alpha,r}.
\end{equation}
In this section and the next section we derive gluing estimates. To simplify notations we let C denote a generic constant whose value may change from line-to-line, but is independent of $(r)$.

\v
\subsection{Estimates of right inverse}

First we recall the definition of $Q'_{( \kappa_{o}, b_{(r)})}.$ Given
$\eta\in L^{k-1,2,\alpha}_{r,u_{(r)}}$ denote
\begin{equation*}
\left(\eta_1(s_1,t_1),\eta_2(s_2,t_2)\right) = \left(\beta_{1;2}(s_1)\eta(s_1,t_1), \beta_{2;2}(s_2)\eta(s_2,t_2)\right),\end{equation*}
\begin{equation*}
Q_{(\kappa_{o},b_o)}(\eta_1,\eta_2)=(\kappa,h)=(\kappa,(h_1,h_2)),\;\;h_{i}\in W^{k,2,\alpha}(\Sigma_{i};u_{i}^{\ast}TM ).\end{equation*}
Then
\begin{equation*}
Q'_{( \kappa_{o}, b_{(r)})}\eta:=(\kappa, h_{(r)})=(\kappa,\beta_{1;r}(s_1)h_{1}(s_1,t_1)+\beta_{2;r}(s_{1}-2r)h_{2}(s_{1}-2r,t_1-\tau)).
\end{equation*}

\begin{lemma}\label{DQ'} For any $\eta\in \lkar ,$ we have
\begin{equation}\label{eqn_DQ'}
D \mathcal{S}_{(\kappa_{o}, b_{(r)})}\circ  Q'_{(\kappa_{o},b_{(r)})}\eta- \eta=\sum (\bar{\partial}\beta_{i;r}) h_{i}+\sum  \beta_{i;r}(F^{1}_{u_{(r)}}-F^{1}_{u_{i}})h_{i}\end{equation}
$$+ \sum \beta_{i;r}(F^{2}_{u_{(r)}}-F^{2}_{u_{i}})\p_{t}h_{i}+(\sum \beta_{i;r}\beta_{i;2}-1)\eta
.$$\end{lemma}
\n{\bf Proof:}
Since $D \mathcal{S}_{(\kappa_{o}, b_{o})}(\kappa, h)=d\mathfrak{i}_{(\kappa_o,b_o)}(\kappa,h)+D_{u}h=(\eta_1,\eta_{2})$ we have
\begin{equation}\label{app_DS_right_inverse}
 D{\mathcal S}_{(\kappa_{o}, b_{(r)})}\circ  Q'_{(\kappa_{o},b_{(r)})}\eta=\eta\;\;\;\;\;\;for \;\;|s_{i}|\leq \tfrac{r}{2}.\end{equation}
It suffices to calculate the left hand side in the annulus $\{\frac{r}{2}\leq |s_i|\leq \frac{3r}{2}\}.$ By choosing $r$ large enough we may assume that $\{\frac{r}{2}\leq |s_i|\leq \frac{3r}{2}\}\subset \Sigma\setminus \Sigma(R_0)$.
 Note that in this annulus
 $$d\mathfrak{i}_{(\kappa_o,b_o)}=0,\;\;\;D\mathcal S_{(\kappa_o,b_{o})}h_i= D_{u_i}h_i=\eta_{i},$$
 $$\beta_{1;r} D_{u_{1}}h_{1}+\beta_{2;r} D_{u_{2}}h_{2}= \sum_{i=1}^{2} \beta_{i;r}\beta_{i;2} \eta,\;\;
 D_{u_{i}}=\bar{\partial}_{j_{o},J_{0}}+F^{1}_{u_{i}}+ F^{2}_{u_{i}}\frac{\partial} {\partial t}. $$
We have, in this annulus,
\begin{align}\label{approximate_difference}
&D \mathcal{S}_{(\kappa_{o}, b_{(r)})}\circ  Q'_{(\kappa_{o},b_{(r)})}\eta- \eta =  \bar{\partial}_{j_{o},J_0}h_{(r)}+ F^{1}_{u_{(r)}}h_{(r)}+ F^{2}_{u_{(r)}}\p_{t}h_{(r)}- \eta \\
&=\sum \beta_{i;r}\bar{\partial}_{j_{o},J_0}h_{i} +\sum \beta_{i;r} F^{1}_{u_i}h_{i}
+\sum (\bar{\partial}\beta_{i;r}) h_{i}+  \sum \beta_{i;r}(F^{1}_{u_{(r)}}-F^{1}_{u_{i}})h_{i}-\eta \nonumber \\
   &\;\;\;+\sum \beta_{i;r}(F^{2}_{u_{(r)}}-F^{2}_{u_{i}})\p_{t}h_{i}+\sum \beta_{i;r} F^{2}_{u_{i}} \p_{t}h_{i}\nonumber\\
&= \sum (\bar{\partial}\beta_{i;r}) h_{i}+  \sum \beta_{i;r}(F^{1}_{u_{(r)}}-F^{1}_{u_{i}})h_{i}+ \sum \beta_{i;r}(F^{2}_{u_{(r)}}-F^{2}_{u_{i}})\p_{t}h_{i}+(\sum \beta_{i;r}\beta_{i;2}-1)\eta.\nonumber
\end{align}
$\Box$
\begin{lemma}\label{aright_inverse_after_gluing}
 Suppose that   $D{\mathcal S}_{(\kappa_o,b_o)}|_{K\times \wka}:K\times   \wka \rightarrow \lka $ is surjective.
Denote by $Q_{(\kappa_{o},b_o)}:\lka \rightarrow K\times \wka$ a bounded right inverse of $D{\mathcal
S}_{(\kappa_o,b_o )}.$ Then $D{\mathcal
S}_{(\kappa_{o}, b_{(r)})}$ is surjective for $r$ large enough. Moreover, there is a right inverse $ Q_{(\kappa_{o},b_{(r)})}$
such that
\begin{equation}
\label{right_estimate}
\| Q_{(\kappa_{o},b_{(r)})}\|\leq  \mathsf{C}
\end{equation}
 for some constant $ \mathsf{C}>0$ independent of $ r $.
\end{lemma}
\n{\bf Proof:} We first show that
\begin{eqnarray}
\label{approximate_right_inverse_estimate_1}
\|Q'_{( \kappa_{o}, b_{(r)})}\|\leq C \\
\label{approximate_right_inverse_estimate_2}
\|D{\mathcal S}_{( \kappa_{o}, b_{(r)})}\circ Q'_{( \kappa_{o}, b_{(r)})}-Id\|\leq \frac{2}{3}
\end{eqnarray}
for some constant $C>0$ independent of $ r $.  By the definition \eqref{norm h} and $0\leq \beta_{i;r}\leq1 $ we have
   \begin{align}\label{h_r_o}
   |(h_{(r)})_{0}| \leq &  e^{-\alpha r}\max_{t\in S^1} | e^{\alpha r}h_{(r)}(r,t)| \leq  e^{-\alpha r}\max_{t_i\in S^1} \sum | e^{\alpha r}h_{i}(r,t_i)|   \\ \nonumber
 \leq  & C e^{-\alpha r} \sum_{i=1,2}\|e^{\alpha |s_{i}|}h_{i}(s_{i},t_{i})|_{r-1\leq  s_{1}\leq r+1}\|_{k,2}    \leq C e^{-\alpha r} \sum \|h_{i}\|_{k,2,\alpha}
   \end{align}
where we used  the  Sobolev embedding theorem in the second inequality.
By $\|Q_{(\kappa_{o},b_o)}\|\leq C$  and the definition of $\|\cdot\|_{k,2,\alpha,r}$ we have
\begin{align*}
\|(\kappa,h_{(r)})\|_{k,2,\alpha,r} &= |\kappa|+ \sum\|\beta_{i;2} (h_{(r)}-(\hat{ h}_{(r)})_{0})\|_{k,2,\alpha}+|(h_{(r)})_{0}| \\
&\leq |\kappa|+ \sum\|\beta_{i;2} h_{(r)}\|_{k,2,\alpha}+C\sum \|h_{i}\|_{k,2,\alpha}   \\
 &\leq
2(C+1)  \|(\kappa,h_{1},h_{2})\|_{k,2,\alpha} \leq C \|(\eta_1,\eta_{2})\|_{k-1,2,\alpha}\leq C\|\eta\|_{k-1,2,\alpha,r},
\end{align*}
where we used \eqref{h_r_o} in the second inequality. Then \eqref{approximate_right_inverse_estimate_1} follows.
\v
We prove \eqref{approximate_right_inverse_estimate_2}. It follows from \eqref{eqn_DQ'} that
\begin{align}
\left\|D \mathcal{S}_{(\kappa_{o}, b_{(r)})}\circ  Q'_{(\kappa_{o},b_{(r)})}\eta-\eta\right\|_{k-1,2,\alpha,r}
&\leq \frac{1}{2}\|\eta\|_{k-1,2,\alpha,r}+\frac{C}{r}\sum \|h_i\|_{k,2,\alpha}  \\
&  \leq   \left(\frac{C}{r}+\frac{1}{2}\right)\|\eta\|_{k-1,2,\alpha,r}. \nonumber
\end{align}
where we used $ \frac{1}{2} \leq \sum \beta_{i;r}\beta_{i;2}\leq \sqrt{2},\;\;\; \sum_{i=1}^{2}|F^{i}_{u_{(r)}}|\leq Ce^{-\fc\frac{r}{2}}\mbox{ in }\;\left\{\tfrac{r}{2}\leq s_1\leq \tfrac{3r}{2 }\right\}$ in the first inequality,
and used $\|Q_{(\kappa_{o},b_o)}\|\leq C$
in the last inequality. Then \eqref{approximate_right_inverse_estimate_2} follows when  $r $ large enough. \v
The estimate \eqref{approximate_right_inverse_estimate_2} implies that
 $D{\mathcal S}_{(\kappa_{o}, b_{(r)})}\circ  Q'_{(\kappa_{o},b_{(r)})}$ is invertible, and   a right inverse $ Q_{(\kappa_{o},b_{(r)})}$  of $D{\mathcal S}_{(\kappa_{o}, b_{(r)})}$ is given by
 \begin{equation}
 \label{express_right_inverse-1}
 Q_{(\kappa_{o},b_{(r)})}= Q'_{(\kappa_{o},b_{(r)})}( D{\mathcal S}_{(\kappa_{o}, b_{(r)})}\circ  Q'_{(\kappa_{o},b_{(r)})})\inv.
 \end{equation}
Then the Lemma follows.  $\Box$

\v

\subsection{Isomorphism between $\ker D \mathcal{S}_{( \kappa_{o}, b_{o})}$ and $\ker D \mathcal{S}_{( \kappa_{o},b_{(r)})}$}\label{isomorphism-1}
\v
For any $(\kappa,h+\hat h_{0})\in \ker D \mathcal {S}_{(\kappa_o,b_o)},$
 where $h=(h_1 ,h_2  )\in W^{k,2,\alpha}_u,$ we set
\begin{equation}
 h_{(r)}=\beta_{1;r}(s_1)h_{1}(s_1,t_1)+\beta_{2;r}(s_{1}-2r)h_{2}(s_{1}-2r,t_{1}-\tau)+\hat h_{0},
\end{equation}
Recall that $I_{r}:\ker D \mathcal{S}_{( \kappa_{o}, b_{o})}\to \ker D \mathcal{S}_{( \kappa_{o},b_{(r)})}$ defined by
 $$
I_{r}(\kappa,h+\hat h_{0})=(\kappa,h_{(r)})- Q_{(\kappa_{o},b_{(r)})}\circ D\mathcal{S}_{(\kappa_o,b_{(r)})}(\kappa,h_{(r)}).
$$

\begin{lemma}\label{lem_est_I_r}
$I_{r}$ is an isomorphisms for $r$ large enough, and
$$\|I_{r}\|\leq \mathsf C ,$$
for some constant $\mathsf C>0$ independent of $r$.
\end{lemma}
\begin{proof} The proof is basically a similar gluing
argument as in \cite{D}. The proof is
divided  into 2 steps. \vskip 0.1in \noindent {\bf Step 1}.  We define a map $I'_{r}:\ker D\mathcal{S}_{( \kappa_{o},b_{(r)})}\longrightarrow \ker D\mathcal{S}_{( \kappa_{o}, b_{o})} $
 and show that $I'_{r}$ is injective for $r$ large enough. For any $(\kappa,h)\in \ker D\mathcal{S}_{( \kappa_{o},b_{(r)})}$ we denote by $h_{i}$ the restriction of $h$
 to the part $|s_{i}|\leq  r +1,$ we get a pair $(h_{1},h_{2}).$  Set
 \begin{equation}
 h_{0}=\int_{S^1}h\left(r,t\right)dt.
 \end{equation}
  We denote
  $$\tilde{h} =\left((h_{1}- \hat h_{0})\beta_{1;2} + \hat h_{0},\;\;(h_{2}- \hat h_{0})\beta_{2;2} + \hat h_{0}\right)$$
and define
$I'_{r}:\ker D\mathcal{S}_{( \kappa_{o},b_{(r)})}\longrightarrow  \ker D\mathcal{S}_{( \kappa_{o}, b_{o})}$ by
\begin{equation}
\label{definition_I'}
I'_{r}(\kappa,h)= (\kappa,\tilde{h})-Q_{( \kappa_{o}, b_{o})}\circ D\mathcal{S}_{( \kappa_{o}, b_{o})}(\kappa,\tilde{h}),
\end{equation}
 where  $Q_{( \kappa_{o}, b_{o})}$ denotes the right inverse of $D\mathcal{S}_{( \kappa_{o}, b_{o})}|_{K\times \wka}: K \times \wka\rightarrow \lka.$  Since
$D\mathcal{S}_{( \kappa_{o}, b_{o})}\circ Q_{( \kappa_{o}, b_{o})}=I,$ we have
$I'_{r}(\ker D\mathcal{S}_{( \kappa_{o},b_{(r)})})\subset  \ker D\mathcal{S}_{( \kappa_{o}, b_{o})}.$
 \v
Let $(\kappa,h)\in \ker D\mathcal{S}_{( \kappa_{o},b_{(r)})}$ such that $I'_{r}(\kappa,h)=0.$ First we prove $h_{0}=0.$ Since $d\mathfrak{i}_{(\kappa_{o},b_{o})}(\kappa,h)$ and $D_{u_{i}} (\beta_{i;2}(h_{i}-\hat{h}_{0}))$ have compact support and $F^l_{u}\in \wka,l=1,2$, we have
$D\mathcal{S}_{( \kappa_{o}, b_{o})}(\kappa,\tilde{h}) \in \lka.$ Then $Q_{( \kappa_{o}, b_{o})}\circ D\mathcal{S}_{( \kappa_{o}, b_{o})}(\kappa,\tilde{h})\in K\times \wka.$ By \eqref{definition_I'} we have $\hat h_{0}=0.$

Next we estimate  $\| (\kappa, \tilde{h}) \|_{k,2,\alpha}.$
From \eqref{definition_I'}, by $I'_{r}(\kappa,h)=0,$ we have
$$\| (\kappa, \tilde{h}) \|_{k,2,\alpha}
 \leq C  \|d\mathfrak{i}_{( \kappa_{o}, b_{o})}(\kappa ,\tilde{h})+ D_u(\tilde{h})\|_{k-1,2,\alpha}  $$
$$
=\left\| d\mathfrak{i}_{( \kappa_{o}, b_{o})}(\kappa ,\tilde{h})+  D_u(\tilde{h})-\left(  \beta_{1;2}  \left(
    D_{u_{(r)}} h+d\mathfrak{i}_{( \kappa_{o},b_{(r)})}(\kappa ,h)\right) ,  \beta_{2;2}  \left(
    D_{u_{(r)}} h+d\mathfrak{i}_{( \kappa_{o},b_{(r)})}(\kappa ,h)\right)  \right)  \right\|_{k-1,2,\alpha}
$$
for some constant $C>0,$  where we used $(\kappa,h)\in \ker D\mathcal {S}_{b_{(r)}}$ in the last inequality.
 We choose $r>4R_{0}$.
As $$d\mathfrak{i}_{( \kappa_{o},b_{(r)})}(\kappa ,h) = d\mathfrak{i}_{( \kappa_{o}, b_{o})}(\kappa ,\tilde{h}),\;\;\;d\mathfrak{i}_{( \kappa_{o},b_{(r)})}(\kappa ,h)|_{|s_i|\geq R_{0}}= d\mathfrak{i}_{( \kappa_{o}, b_{o})}(\kappa ,\tilde{h})|_{|s_i|\geq R_{0}}=0$$ and $\beta_{i;2} |_{|s_{i}|\leq r-1}=1$ we have $(\beta_{1;2} d\mathfrak{i}_{( \kappa_{o},b_{(r)})}(\kappa ,h),\beta_{2;2} d\mathfrak{i}_{( \kappa_{o},b_{(r)})}(\kappa ,h))= d\mathfrak{i}_{( \kappa_{o}, b_{o})}(\kappa ,\tilde{h})$.
Set
$(\bar{\partial}\beta) h=(\bar{\p} \beta_{1;2} h_{1}, \bar{\p} \beta_{2;2} h_{2}).$ Therefore
\begin{eqnarray*}
 \| (\kappa, \tilde{h}) \|_{k,2,\alpha}
 &\leq &   C  \|(\bar{\partial}\beta)  h \|_{k-1,2,\alpha} +C \sum_{i=1}^{2}\|
\beta_{i;2} (F^{1}_{u_{i}} - F^{1}_{u_{(r)}}) h  \|_{k-1,2,\alpha} \\
&& +C \sum_{i=1}^{2}\|
\beta_{i;2}(F^{2}_{u_{i}} - F^{1}_{u_{(r)}}) \p_{t}h  \|_{k-1,2,\alpha} .
\end{eqnarray*}
Note that
 $F^l_{u} = F^l_{u_{(r)}},l=1,2$  in $\{|s_{i}|\leq \; \tfrac{r}{2}\} . $
By exponential decay of $F^l_{u},F^l_{u_{(r)}},$ $l=1,2,$ in $\{\frac{ r}{2}\leq s_{1}\leq \frac{3r}{2}\}$, there exists a constant  $C>0$ such that
$$\sum_{i=1}^{2}\|
\beta_{i;2} (F^{1}_{u_{i}} - F^{1}_{u_{(r)}}) h  \|_{k-1,2,\alpha}  +\sum_{i=1}^{2}\|
\beta_{i;2}(F^{2}_{u_{i}} - F^{1}_{u_{(r)}}) \p_{t}h  \|_{k-1,2,\alpha}  \leq Ce^{-\mathfrak{c}\tfrac{r}{2}}\|\beta h\|_{k,2,\alpha}.$$
 Since $(\bar{\partial}\beta_{1;2}) h_{1}$ supports in
$r-1\leq s_1 \leq r+ 1$, and over this part
$$|\bar{\partial}\beta_{1;2}|\leq 4,\;\;\;\;r-  1\leq |s_2| \leq r+1,\;\;\;e^{2\alpha|s_1|}\leq e^{2\alpha}e^{2\alpha|s_2|},\;\;\;\beta_{1;2}+\beta_{2;2}\geq 1,\;\;\;h_{1}=h_{2},$$
we obtain
\begin{align*}
&\|(\bar{\partial}\beta_{1;2})h_{1}\|_{k-1,2,\alpha}\leq  C \|h_{1}|_{r-1\leq s_{1}\leq r+1}\|_{\Sigma_{1},k-1,2,\alpha} \leq  C \|\sum\beta_{i}h_{i}|_{r-1\leq s_{1}\leq r+1}\|_{\Sigma_{1},k-1,2,\alpha}\\
&\leq   C\sum \|\beta_{i;2}h_{i}|_{r-1\leq |s_{i}|\leq r+1}\|_{k-1,2,\alpha}=   C\|h|_{r-1\leq s\leq r+1}\|_{k-1,2,\alpha,r}\leq Ce^{-(\fc -\alpha)\frac{r}{2}}  \|h\|_{k-1,2,\alpha,r}
\end{align*}
where we have used Corollary \ref{tube_ker_c-1} with $R'=r-1,R=\frac{r-1}{2}$ and $\eta=0$ in the last inequality.
Similar inequality for $(\bar{\partial}\beta_{2;2})h_{2}$
also holds. So we have
$$\|(\bar{\partial}\beta)h\|_{k-1,2,\alpha} \leq Ce^{-(\fc -\alpha)\frac{r-1}{2}}  \|h\|_{k-1,2,\alpha,r}=Ce^{-(\fc -\alpha)\frac{r-1}{2}} \|\tilde{h}\|_{k,2,\alpha}.$$
Hence
\begin{equation}\label{delta_I'}
\| (\kappa, \tilde{h}) \|_{k,2,\alpha}\leq C(e^{-(\fc -\alpha)\frac{r}{2}}+
e^{- \tfrac{\mathfrak{c}r}{2}})\|\tilde{h}\|_{k,2,\alpha} \leq   1/2  \|\tilde{h}\|_{k,2,\alpha}  \end{equation}
for some constant $C>0$, here we choosed    $r$   large enough  such
that   $Ce^{- (\fc -\alpha)\frac{r}{2}} < 1/4$. Then $I'_{r}(\kappa,h)=0$ and
\eqref{delta_I'} gives us $$ |\kappa|=0, \;\;  \|\tilde{h}\|_{k,2,\alpha}=0 .$$
Note that $\tilde \beta_{i;r}h_{i}|_{|s_{i}|\leq r}=h|_{|s_{i}|\leq r}.$ It follows that $\kappa = 0, \;\; h=0$. So $I'_r$ is injective.
\vskip 0.1in
\noindent
{\bf Step 2}. Let $(\kappa,h+\hat h_{0})\in \ker D\mathcal{S}_{(\kappa_{o},b_{o})}$ with $I_r(\kappa,h+h_{0})=0$.
Since $\|Q_{( \kappa_{o},b_{(r)})}\|$ is uniformly bounded, from \eqref{definition_I}  and \eqref{right_estimate}, we have
$$\|(\kappa, h_{(r)})\|_{k,2,\alpha,r} =\|I_r(\kappa,h+h_{0}) - (\kappa, h_{(r)})\|_{k,2,\alpha,r}\leq C\|D\mathcal {S}_{( \kappa_{o},b_{(r)})} (\kappa, h_{(r)})\|$$
for some constant $C>0.$ By a similar culculation as in \eqref{approximate_difference}  we obtain
\begin{align}\label{eqn_a_DK}
&D\mathcal {S}_{( \kappa_{o},b_{(r)})} (\kappa, h_{(r)})=\bar{\partial}_{j_{o},J_0}h_{(r)}+ F^{1}_{u_{(r)}}h_{(r)}+ F^{2}_{u_{(r)}}\p_{t}h_{(r)}+d\mathfrak{i}_{( \kappa_{o},b_{(r)})}(\kappa,h_{(r)})  \\
&=\sum (\bar{\partial}\beta_{i;r}) h_{i}+  \sum \beta_{i;r}(F^{1}_{u_{(r)}}-F^{1}_{u_{i}})h_{i} +\sum \beta_{i;r}(F^{2}_{u_{(r)}}-F^{2}_{u_{i}})\p_{t}h_{i}+(F_{u_{(r)}}^{1}-\sum \beta_{i;r} F^{1}_{u_{i}})\hat h_{0}  \nonumber
\end{align} where we used $D\mathcal {S}_{( \kappa_{o},b_{(r)})}(\kappa,h+\hat h_{0})=0.$ Then we have
\begin{equation}\label{delta_I}
\| (\kappa, h_{(r)})\|_{k,2,\alpha,r}\leq \frac{C}{r}
(\|h\|_{k,2,\alpha} + |h_0|)\end{equation} for some constant $C>0$.
Since $d\mathfrak{i}_{ (\kappa_{o},u)}|_{|s_i|\geq R_0}=d\mathfrak{i}_{ (\kappa_{o},u_{(r)})}|_{|s_i|\geq R_0}=0,$  for any $(\kappa,h)\in \ker D\mathcal S_{( \kappa_{o}, b_{o})},$ restricting in $\{|s_i|\geq R_0\}$, we have
$$
 \overline{\partial }_{J_0}(h+\hat h_0) +F^{1}_{u}(h+\hat h_0)+F^{2}_{u}\p_{t}(h+\hat h_0)=D\mathcal S_{( \kappa_{o}, b_{o})}(\kappa,h+\hat h_0) =0.
$$
Let $ \epsilon'\in (0,1)$ be a constant. Applying Corollary \ref{tube_ker_c-1} with $R= \max (\frac{-\ln\epsilon'+\ln 2\mathsf C}{\alpha},R_0+2)$ and $R'=2R$, we conclude that
 the restriction of $h$ to $|s_i|\geq 2R$ satisfies
$$\|h|_{|s_i|\geq 2R}\|_{k,2,\alpha}\leq 2\mathsf C e^{-\alpha R}(\|h\|_{k,2,\alpha}+|h_0|) \leq \epsilon'(\|h\|_{k,2,\alpha}+|h_0|),$$
therefore
\begin{equation}\label{lower_bound_h_r}
\|h_{(r)}\|_{k,2,\alpha,r}\geq \|h|_{|s_i|\leq 2R_1}\|_{k,2,\alpha} + |h_0|
\geq (1 - \epsilon')(\|h\|_{k,2,\alpha} + |h_0|),
\end{equation}
for $r>4R$.  Then \eqref{delta_I} and \eqref{lower_bound_h_r} give us $h=0$ and $h_{0}=0$,
and so $\kappa=0$. Hence $I_r$ is injective.
\vskip 0.1in
\noindent
The {\bf step 1} and {\bf step 2} together show that both $I_{r}$ and $I'_r$
are isomorphisms for $r$ large enough.
\end{proof}

\subsection{Gluing maps}\label{gluing_maps}
\vskip
0.1in \noindent
Choose $R_{0}$ large such that
\begin{equation}\label{eqn_ass_R}
 \sum_{i=1}^{2}E(u_i; |s_{i}|\geq \tfrac{R_{0}}{2})\leq \tfrac{\hbar}{8}.
 \end{equation}
where $\hbar$ is the constant in Theorem \ref{tube_exponential_decay}.

\begin{lemma}\label{lem_est_al_hol} Suppose that $\mathcal S (\kappa_{o},b_{o})=0.$   Then there exists a  constant $\mathsf C>0$  independent of $r$ such that for $r>R_0$
\begin{equation}\label{est_app_hol}
\|\mathcal S (\kappa_{o},b_{(r)})\|_{k-1,2,\alpha,r}\leq
\mathsf Ce^{-(\fc-\alpha)\frac{r}{2}}.
\end{equation}
 \end{lemma}

\begin{proof} Since $u_{(r)}|_{\{|s_{i}|\leq \frac{r}{2}\}}=u|_{\{|s_{i}|\leq \frac{r}{2}\}}$, we have $\mathcal S (\kappa_{o},b_{(r)})|_{\{|s_{i}|\leq \frac{r}{2}\}}=0.$
Note that $\mathfrak{i}(\kappa_{o},b_{(r)})=0$ in  $\{ \tfrac{r}{2}\leq  s_{1}\leq \tfrac{3r}{2}\}$. So we get
\begin{align}\label{eqn_al_S}
\mathcal S (\kappa_{o},b_{(r)})=  &\beta\left(3-\tfrac{4s_1}{r}\right) \bar{\partial}_{j_{o},J}(u_1(s_1,t_1) ) +
\beta\left(\tfrac{4s_1}{r}-5\right)\bar{\partial}_{j_{o},J}u_2(s_1-2r,t_1-\tau) ) \\\nonumber
&+\tfrac{\p }{\p s_{1}}(\beta\left(3-\tfrac{4s_1}{r}\right)) (u_1(s_1,t_1)-u_1(q))+
\tfrac{\p }{\p s_{1}}(\beta\left(\tfrac{4s_1}{r}-5\right)) (u_2(s_1-2r,t_1-\tau)-u_2(q)).
\end{align}
By \eqref{eqn_ass_R} and Theorem \ref{tube_exponential_decay} we can obtain
\begin{equation}\label{eqn_exp_decay}
 \left|d u (s_{i},t_{i})\right|\leq \mathsf Ce^{-\fc|s_{i}|}\;\;\;\;\;\;\;\;\;\mbox{ for any }\tfrac{r}{2}\leq  |s_{i}|\leq \tfrac{3r}{2}.
\end{equation}
Then \eqref{est_app_hol} follows from the exponential decay of $|du|$.
\end{proof}

\v
For fixed $(r)$ we consider the family of maps:
\begin{align*}& \mathcal{F}_{(r)}: K  \times W^{k,2,\alpha}_{r,u_{(r)}}
\rightarrow L^{k-1,2,\alpha}_{r,u_{(r)}},\;\;\;\mathcal{F}_{(r)}(\kappa,h)=P_{b, b_{(r)}}\left(\bar{\partial}_{j_{o},J}v
+ \mathfrak{i}(\kappa,b)\right),
\end{align*}
where $b=(r,\tau,a_{o}, v),\;\;v=\exp_{u_{(r)}}h$  and
\begin{equation*}
P_{b, b_{(r)}}=\Phi_{u_{(r)}}(h)^{-1}:W_{r}^{k-1,2,\alpha}(\Sigma_{(r)}, v^{*}TM\otimes \wedge^{0,1}_{j_{o}}T^{*}\Sigma_{(r)})\to \lkar.
\end{equation*}
Let $\mathbf g_{o}$ be the metrics on $(\Sigma,j_{o}) $ as in Section \S\ref{s_intro_1}. We denote by $c_{\mathbf{g}_{o}} $ the norm of the
Sobolev embedding $W^{k,2}(\Sigma  ,\mathbf{g}_{o})\to C^{k-2}(\Sigma ,\mathbf{g}_{o})$.
\v
For every $(r)$, $\mathcal{F}_{(r)}(\kappa,h)$ is a smooth function of $(\kappa,h)$.
 Consider the path $\mathbb R\to K    \times  \wkr : \lambda\to (\lambda \hat \kappa, \lambda \zeta).$

\begin{lemma} \label{difference_DF_estimates_gw}
\begin{description}
\item[1.] $\frac{d}{d\lambda } \mathcal{F}_{(r)}(\lambda\hat \kappa,\lambda\zeta) |_{\lambda=0}= d \mathcal{F}_{(r)}(0)(\hat \kappa,  \zeta)= D\mathcal{S}_{ ( \kappa_{o},b_{(r)}) }(\hat \kappa,  \zeta).$
\item[2.] For every constant $\mathsf{d}_{1}>0$ there exists a constant $\mathsf C>0$ such that the following holds for every metric $\mathbf g_{o,(r)}$ on $\Sigma_{(r)}$ with  $c_{\mathbf g_{o,(r)}}<\mathsf{d}_{1}$:
if $u_{(r)}\in W^{k,2,\alpha}_{r}(\Sigma_{r},M)$ and
$(\kappa,h)\in {K} \times  \wkr$ satisfying
\begin{equation}\label{ass_du}
\|du_{(r)}\|_{k-1,2,\alpha,r}\leq \mathsf{d}_{1} ,\;\;\; \|(\kappa,h)\|_{k,2,\alpha,r}\leq \mathsf{d}_{1}
\end{equation}
then
\begin{equation} \label{difference_of_DF_estimates}
\|d\mathcal{F}_{(r)}(\kappa ,h)-d\mathcal{F}_{(r)}(0  )\| \leq \mathsf C\|(\kappa,h)\|_{k,2,\alpha,r}.
\end{equation}
Here $\|\cdot\|$ denotes the operator norm on $L(K\times \wkr, \lkar)$.
\end{description}
\end{lemma}
\v\n
The proof is similar to the proofs of Proposition 3.1.1 and Proposition 3.5.3 in \cite{MS}, we omit it here.
\v
Now we check that $\mathcal F_{(r)}$ satisfies the assumption of Theorem \ref{details_implicit_function_theorem}.
There exists two constants  $\epsilon_{o}>0$ and $C_{o}>0$ such that  $\mathbf{g}_{o}$ is a complete Riemannian metric with injectivity radius $\inj(\Sigma,\mathbf{g}_{o})>\epsilon_{o}$ and sectional curvature $|Rm(\Sigma,\mathbf{g}_{o})|\leq C_{o}$. Then there exists a constant $C>0$ depending only on $\epsilon_{o}$ and $C_{o}$ such that $c_{\mathbf g_{o}}<C$
 ( for the Sobolev embedding theorem see \cite{A} ).
Then we have
\begin{equation}
\|h\|_{C^{k-2}(\Sigma_{(r)})}\leq \sum \|\beta_{i;2} h \|_{C^{k-2}(\Sigma_{i})}\leq \sum c_{\mathbf{g}_{o}}\|\beta_{i;2} h \|_{W^{k,2}(\Sigma_{i})}\leq 2c_{\mathbf{g}_{o}}\|  h \|_{W^{k,2}(\Sigma_{(r)})}.
\end{equation}
It follows that $c_{\mathbf g_{o,(r)}}<C .$
On the other hand,  by Theorem \ref{tube_exponential_decay} we have
$$
\|du_{(r)}\|_{k-1,2,\alpha,r}\leq \sum\|\beta_{i;2}du \|_{\Sigma_{i},k-1,2,\alpha}+Ce^{-(\fc-\alpha)\frac{r}{2}}\leq C.
$$
Choosing $\mathsf{d}_{1} $ large enough, by  Lemma \ref{difference_DF_estimates_gw}  we have
\begin{equation}\label{uniform_bound_difference_dF}
 \|d\mathcal{F}_{(r)}(\kappa,h) - D\mathcal{S}_{( \kappa_{o},b_{(r)})}\| \leq \mathsf C\|(\kappa,h)\|_{k,2,\alpha,r}.
\end{equation}
By Lemma \ref{lem_est_al_hol} we have
\begin{equation}
\|\mathcal{F}_{(r)}(0)\|_{k-1,2,\alpha,r}= \|\mathcal{S}(\kappa_{o},b_{(r)}) \|_{k-1,2,\alpha,r}\leq   \mathsf C e^{-(\fc-\alpha)\frac{r}{2}}.
\end{equation}
Then $\mathcal{F}_{(r)}$ satisfies the conditions in  Implicit function theorem \ref{details_implicit_function_theorem} when $r$ large enough and  $\|(\kappa,h)\|_{k,2,\alpha,r}$ small enough. Hence the zero set of $\mathcal{F}_{(r)}$ is locally the form
$(\kappa_{r},\zeta_{r}) + Q_{( \kappa_{o},b_{(r)})}\circ f_{(r)}(\kappa_{r},\zeta_{r})$,\;i.e
\begin{equation}\label{j_holomorphic_curve}
\mathcal{F}_{(r)}\left((\kappa_{r},\zeta_{r}) + Q_{( \kappa_{o},b_{(r)})}\circ f_{(r)}(\kappa_{r},\zeta_{r})\right)=0\end{equation}
where $(\kappa_{r},\zeta_{r}) \in \ker\;D{\mathcal S}_{( \kappa_{o},b_{(r)})}$ with $\|(\kappa_{r},\zeta_{r})\| $ small.
\v
Since $I_r$ is an isomorphism for $r$ large enough we
have a gluing map
$$I_{r}+Q_{( \kappa_{o},b_{(r)})}\circ f_{(r)} \circ I_{r}$$
from $O$ into the  space of  perturbed $(j_o,J)$-holomorphic maps, where $O$ is a neighborhood of $0$ in $\ker D \mathcal{S}_{(\kappa_o,b_o)}$.

\section{\bf Estimates of derivatives with respect to gluing parameters $r$}\label{diff gluing parameters}

In this section we prove Theorem \ref{lem_est_rI}. We fix an arbitrary $a_o=(a_{o1},a_{o2})\in \mathbf{A}_1\times \mathbf{A}_2$.\v
{\bf Assumption ($\ast$).} Let $\Sigma$ be a marked nodal Riemann surface with one nodal point $q$. Let $(\kappa_o, b_o)=(\kappa_o, a_o, u)\in K\times O(\rho)$ be a perturbed $(j_o,J)$-holomorphic map from $\Sigma$ to $M$, where $u=(u_{1},u_{2}):\Sigma_1\cup \Sigma_2\to M$ with $u_{1}(q)=u_{2}(q)$ and
$$D{\mathcal S}_{(\kappa_o,b_o)}|_{K\times  \wka} :K\times  \wka \rightarrow \lka $$ is surjective.
Denote by $Q_{(\kappa_{o},b_o)}:\lka \rightarrow K \times \wka$ a right inverse of $D{\mathcal S}_{(\kappa_o,b_o)} .$ Let $ Q_{(\kappa_{o},b_{(r)})}$ be the right inverse of D${\mathcal S}_{(\kappa_o,b_{(r)})}$ defined in \eqref{express_right_inverse}.

\v
To simplify notations we denote
$$D:=D{\mathcal S}_{( \kappa_{o},b_{(r)})},\;\;Q:=Q_{(\kappa_{o},b_{(r)})},$$$$Q':=  Q'_{(\kappa_{o},b_{(r)})},\;\;
(Q^{\prime})^{*}:=(Q^{\prime})^{*}_{(\kappa_{o},b_{(r)})},
\;\;Q^{*}:=Q^{*}_{(\kappa_{o},b_{(r)})}.$$

\v
The main result of this section is to prove Theorem \ref{lem_est_rI}. We first prove some lemmas.

\subsection{Estimates for $\frac{\p}{\p r} ((Q')^{*}P_r)$}\label{notations}
We define an operator $X:\lka\to \lkar$ as
\begin{equation}\label{ass_X}
X(\eta_{1},\eta_{2})=DQ'P_{r}(\eta_{1},\eta_{2})-P_{r}(\eta_{1},\eta_{2}).
\end{equation}
By Lemma \ref{DQ'}, we have
\begin{align}\label{def.X}
X(\eta_1,\eta_2)=&\sum (\bar{\partial}\beta_{i;r}) h_{i}+\sum  \beta_{i;r}(F^{1}_{u_{(r)}}-F^{1}_{u_{i}})h_{i}
+ \sum \beta_{i;r}(F^{2}_{u_{(r)}}-F^{2}_{u_{i}})\p_{t}h_{i} \\ \nonumber
&+\left(\sum \beta_{i;r}\beta_{i;2}-1\right)\sum \beta_{i;2}\eta_{i},\end{align}
where $(\kappa,h)=Q_{(\kappa_{o},b_o)}H_{r}P_{r}(\eta_{1},\eta_{2}).$  Then  $\supp X(\eta_1,\eta_2) \subset \{\frac{r}{2}\leq |s_{i}|\leq \frac{3r}{2}\}$.

\v
\begin{lemma}\label{lem_def_HP} For any $(\eta_1,\eta_2)\in
\lka$ the following estimates hold:
\begin{itemize}
 \item[(a)] $\| (H_{r}P_r)(\eta_{1},\eta_{2} )\|_{k-1,2,\alpha} \leq  \mathsf C\sum_{i=1}^{2} \left\|\eta_{i}|_{|s_{i}|\leq  r+1}\right\|_{\Sigma_{i},k-1,2,\alpha},$ \\
     $\| (H_{r}P_r)(\eta_{1},\eta_{2} )|_{|s_{i}|\geq R}\|_{k-1,2,\alpha} \leq \mathsf C\sum_{i=1}^{2} \left\|\eta_{i}|_{R\leq |s_{i}|\leq  r+1}\right\|_{\Sigma_{i},k-1,2,\alpha},$\;\;\;\;\ for any $R\geq R_{0},$
 \item[(b)] $\|\frac{\p}{\p r}(H_{r}P_r)(\eta_{1},\eta_{2} )\|_{k-2,2,\alpha} \leq  \mathsf C\sum_{i=1}^{2} \left\|\eta_{i}|_{r-1 \leq |s_{i}|\leq  r+1}\right\|_{\Sigma_{i},k-1,2,\alpha},$
  \item[(c)]    $\|  (Q')^{*}P_r  (\eta_{1},\eta_{2} )\|_{k,2,\alpha} \leq   \mathsf C\sum\limits_{i=1}^{2} \left\|H_{r}P_{r}(\eta_{1},\eta_{2})|_{|s_{i}|\leq  r+1}\right\|_{\Sigma_{i},k-1,2,\alpha} ,$
 \item[(d)] $\|\frac{\p}{\p r} ((Q')^{*}P_r)  (\eta_{1},\eta_{2} )\|_{k-1,2,\alpha} \leq \mathsf C\left(e^{-(\fc -\alpha)\frac{r}{4}}\sum\|\eta_{i}|_{|s_{i}|\leq r+1}\|_{k-1,2,\alpha}+ \sum_{i=1}^{2} \left\|\eta_{i}|_{ \frac{r}{4} \leq |s_{i}|\leq  r+1}\right\|_{k-1,2,\alpha}\right)$,
 \item[(e)] $ \|(Q')^{*}P_{r}(\eta_{1},\eta_{2})|_{\frac{r}{2}\leq |s_{i}|\leq \frac{3r}{2}}\|_{k,2,\alpha}
\leq  C\left[e^{-(\fc-\alpha)\frac{r}{4} } \|H_{r}P_{r}(\eta_{1},\eta_{2}) \|_{k-1,2,\alpha}+ \left\|H_{r}P_{r}(\eta_{1},\eta_{2})\mid_{ | s_{i}| \geq  \frac{r}{4} }  \right\|_{k-1,2,\alpha}\right].
$
 \end{itemize}
\end{lemma}
\n{\bf Proof.}
By definition
\begin{equation}\label{eqn_est_HP_0}
(H_{r}P_r)(\eta_{1},\eta_{2} )=(\tilde \eta_{1},\tilde \eta_{2}) ,\end{equation} where
$$\tilde \eta_{1}=\beta_{1;2}\left[\beta_{1;2}\eta_1(s_1,t_1)+\beta_{2;2}\eta_{2}(s_1-2r, t_1-\tau)\right],\;\;\;\tilde \eta_{2}=\beta_{2;2}\left[\beta_{1;2}\eta_1(s_2+2r,t_2+\tau)+\beta_{2;2}\eta_{2}(s_2, t_2)\right].
$$
Note that
\begin{equation}\label{eqn_eta_1}
\left\|\eta_{2}(s_1-2r,t_{1}-\tau)|_{r-1\leq s_{1}\leq r+1}\right\|_{\Sigma_1,k-1,2,\alpha} \leq  e^{2\alpha}\left\|\eta_{2}(s_2,t_{2})|_{r-1\leq s_{1}\leq r+1}\right\|_{\Sigma_{2},k-1,2,\alpha}.
\end{equation}
Then (a) follows from \eqref{eqn_eta_1} and the definition of $\beta_{i;2}$.
\v
By \eqref{beta} we have
\begin{align*}\label{eqn_eta_2}
\frac{\p \tilde \eta_{1}}{\p r}= &\tfrac{\p\beta_{1;2}^2}{\p r} \eta_1(s_1,t_{1})+ \tfrac{\p\left(\beta_{1;2}\sqrt{1- \beta_{1;2}^2(s_1)}\right) }{\p r}\eta_{2}(s_1-2r,t_1-\tau)
 -2\beta_{1;2}\sqrt{1- \beta_{1;2}^2(s_1)}\tfrac{\p( \eta_{2})}{\p s_{2}}(s_1-2r,t_1-\tau).
\end{align*}
Then $$\left\|\frac{\p \tilde \eta_{1}}{\p r}\right\|_{k-2,2,\alpha}\leq C\sum_{i=1}^{2} \left\|\eta_{i}|_{r-1 \leq |s_{i}|\leq  r+1}\right\|_{\Sigma_{i},k-1,2,\alpha}.$$
A same estimate for $\frac{\p\tilde \eta_{2}}{\p r} $ also holds.
Then  (b) follows from $\frac{\p}{\p r}(H_{r}P_r)(\eta_{1},\eta_{2} )=(\frac{\p}{\p r}\tilde \eta_{1},\frac{\p}{\p r}\tilde \eta_{2})$.
\v Denote $(\kappa,h_{1},h_{2})=Q_{(\kappa_{o},b_o)}H_{r}P_{r}(\eta_1,\eta_2).$ Recall that (cf. \eqref{eqn_def_Q*})
\begin{equation}\label{eqn_(Q')^*}
(Q^{\prime})^{*} P_r (\eta_{1},\eta_{2} )=(\kappa,\beta_{1;r} h_1,\beta_{2;r} h_2).\end{equation}
Then (c) follows from (a), $|\beta_{i;r}|\leq 1$ and $\|Q_{(\kappa_{o},b_o)}\|\leq C$.

Taking the derivative $\frac{\p}{\p r}$ of \eqref{eqn_(Q')^*} we obtain
\begin{equation} \label{eqn_Q'*_1}
\tfrac{\p}{\p r}((Q^{\prime})^{*} P_r) (\eta_{1},\eta_{2} )=(\tfrac{\p \kappa}{\p r},\beta_{1;r}\tfrac{\p h_{1}}{\p r},\beta_{2;r}\tfrac{\p h_{2}}{\p r})+(0,h_{1}\tfrac{\p \beta_{1;r}}{\p r}, h_{2}\tfrac{\p\beta_{2;r}}{\p r}).
\end{equation}
On the other hand, since $\frac{\p}{\p r}(\kappa,h_1,h_2)= Q_{(\kappa_{o},b_{o})}\left(\frac{\p}{\p r}(H_{r}P_{r})(\eta_1,\eta_{2})\right),$ we have
\begin{align}\label{eqn_est_II_1}
 \left\|\frac{\p}{\p r}(\kappa,h_1,h_2)\right\|_{k-1,2,\alpha}
   \leq C \left\|(\eta_1,\eta_{2})|_{r-1 \leq |s_{i}|\leq r+1}\right\|_{k-1,2,\alpha},
\end{align}
where we have used the bound of right inverse and (b) in the inequality. Since $h_{i}\tfrac{\p \beta_{i;r}}{\p r}\subset \{ \frac{r}{2}\leq |s_{i}|\leq \frac{3r}{2}\},$
we have
\begin{equation}\label{eqn_adH}
\|(0,h_{1}\tfrac{\p \beta_{1;r}}{\p r}, h_{2}\tfrac{\p\beta_{2;r}}{\p r})\|_{k-1,2,\alpha} \leq C\|(h_{1}, h_{2})|_{\frac{r}{2}\leq |s_{i}|\leq \frac{3r}{2}}\|_{k-1,2,\alpha}.
\end{equation}
Since $d\mathfrak{i}_{(\kappa_{o},b_{o})}|_{R_{0}\leq |s_{i}|\leq 2r-R_{0}}=0 $  we have
\begin{equation}\label{eqn_a4.20}
D_{u}(h_{1},h_{2})=H_{r}P_{r}(\eta_{1},\eta_{2}),\;\mbox{ for }\;R_{0}\leq |s_{i}|\leq 2r-R_{0} .
\end{equation}
By  $(\kappa,h_{1},h_{2})=Q_{(\kappa_{o},b_o)}H_{r}P_{r}(\eta_1,\eta_2)$
and $ \kappa |_{R_{0}\leq |s_{i}|\leq 2r-R_{0}}=0$,
applying \eqref{f_tube_estimate_curve} of Corollary \ref{tube_ker_c-1} with $R'=\frac{r}{2}$, $R=\frac{r}{4}$ and $\eta=H_{r}P_{r}(\eta_{1},\eta_{2})$,  we conclude that
 \begin{align} \label{eqn_HX}
   \|(h_{1}, h_{2})|_{\frac{r}{2}\leq |s_{i}|\leq \frac{3r}{2}}\|_{k-1,2,\alpha}\leq &C\left(e^{-(\fc -\alpha)\frac{r}{4}}\|H_{r}P_{r}(\eta_{1},\eta_{2})\|_{k-1,2,\alpha}+\|H_{r}P_{r}(\eta_{1},\eta_{2})|_{|s_{i}|\geq \frac{r}{4}}\|_{k-1,2,\alpha}\right) \\
 \nonumber \leq& C\left(e^{-(\fc -\alpha)\frac{r}{4}}\sum\|\eta_{i}|_{|s_{i}|\leq r+1}\|_{k-1,2,\alpha}+\sum\|\eta_{i}|_{\frac{r}{4}\leq |s_{i}|\leq r+1}\|_{k-1,2,\alpha}\right),
 \end{align}
where we have used (a) in the last inequality. By \eqref{eqn_Q'*_1}, \eqref{eqn_est_II_1}, \eqref{eqn_adH} and  \eqref{eqn_HX} we get (d).
(e) follows from \eqref{eqn_HX} and $\| (\beta_{1;r}h_{1},\beta_{2;r}h_{2})|_{\frac{r}{2}\leq |s_{i}|\leq \frac{3r}{2}}\|_{k,2,\alpha}\leq \| (h_{1},h_{2})|_{\frac{r}{2}\leq |s_{i}|\leq \frac{3r}{2}}\|_{k,2,\alpha}$.    $\;\;\;\Box$

\v
\begin{lemma} \label{lem_est_HDS} There exists a constant $\mathsf C>0,$
  such that for any $(\kappa,  h +\hat h_{0})\in K\times \cwk$ with $\supp\; h_{i}\subset \{|s_{i}|\leq \frac{3r}{2}\}$, we have
$$
\|H_{r}D(Id,\phi_{r})(\kappa, (h_{1} +\hat h_{0},h_{2}+\hat h_{0}) )\|_{k-1,2,\alpha} \leq \mathsf C  \|(\kappa, h_{1}+\hat h_{0}  ,h_{2}+\hat h_{0} )\| _{k,2,\alpha}.
$$
In particular, if $(\kappa, h +\hat h_{0} ) $ satisfies  $D(\kappa,h+\hat h_{0})|_{|s_{i}|\leq \frac{r}{2}}=0$, we have
$$
\|H_{r}D(Id,\phi_{r})(\kappa, h_{1}\beta_{1;r}+\hat{h}_{0}, h_{2}\beta_{2;r}+\hat{h}_{0}  )\|_{k-1,2,\alpha} \leq \mathsf C  (\| (h_{1}   ,h_{2})|_{\frac{r}{2}\leq |s_{i}|\leq \frac{3r}{2}}
\| _{k,2,\alpha}+e^{(\fc-\alpha)\frac{r}{2}}|h_{0}|).
$$ \end{lemma}
\begin{proof}
By the same calculation as  in \eqref{approximate_difference} we have, in $\{|s_{i}|\leq r+1\}$,
\begin{align*}
&H_{r}D(Id,\phi_{r})(\kappa, (  h_{1}+\hat h_{0}, h_{2}+\hat h_{0})) =H_{r}D (\kappa, \sum h_{i}+\hat h_{0}):=(\tilde \eta_{1},\tilde \eta_{2})
\end{align*}
where
$$\tilde \eta_{i} =\beta_{i;2}\left(\sum_{i=1}^{2} D_{u_{i}}  h_{i} +d\mathfrak{i}_{(\kappa_o, b_{o})}(\kappa,h +\hat h_{0} )    + \sum_{i=1}^{2} (F^{1}_{u_{(r)}}-F^{1}_{u_{i}}) h_{i}+F^{1}_{u_{(r)}}\hat h_{0}  +  \sum_{i=1}^{2}  (F^{2}_{u_{(r)}}-F^{2}_{u_{i}})\p_{t}h_{i} \right) . $$
Then the first inequality follows. Note that $(\kappa,h_{(r)})=(Id,\phi_{r})(\kappa, h_{1}\beta_{1;r}+\hat{h}_{0}, h_{2}\beta_{2;r}+\hat{h}_{0}  ).$ Applying \eqref{eqn_a_DK}, the exponential decay of $F^{k}_{u_{(r)}},F^{k}_{u},k=1,2,$
and $$\beta_{i;r}|_{|s_{i}|\leq \frac{r}{2}}=1,\;\;\;\beta_{i;r}|_{|s_{2-i}|\leq \frac{r}{2}}=0,\;\;\;D(h+\hat h_{0})|_{|s_{i}|\leq \frac{r}{2}}=0,\;\;F^{l}_{u_{(r)}}-F^{l}_{u_{i}}|_{|s_{i}|\leq \frac{r}{2}}=0,\;\;\;i,l=1,2,$$  we can prove the second inequality.
\end{proof}

\v
\begin{lemma}\label{lem_HX} There exists a constant $\mathsf C>0,$
  such that  for any $(\eta_1,\eta_2)\in
\lka$ the following estimates hold:
\begin{itemize}
\item[(1)] $\|H_{r}X(\eta_1,\eta_2)\|_{k-1,2,\alpha}  \leq \mathsf C  \left[ e^{-(\fc-\alpha)\frac{r}{4}} \|  (\eta_1,\eta_{2})|_{ |s_{i}|\leq  r+1}\|_{k-1,2,\alpha} + \left\|(\eta_1,\eta_{2})|_{\frac{r}{4}\leq |s_{i}|\leq  r+1}\right\|_{k-1,2,\alpha} \right],$
\item[(2)] $\| \frac{\p}{\p r} \left(H_{r} X\right)(\eta_1,\eta_2)\|_{k-2,2,\alpha} \leq
   \mathsf C \left[e^{-(\fc-\alpha)\frac{r}{4}} \|  (\eta_1,\eta_{2})|_{ |s_{i}|\leq  r+1}\|_{k-1,2,\alpha}   +\left\|(\eta_1,\eta_{2})|_{\frac{r}{4} \leq |s_{i}|\leq  r+1}\right\|_{k-1,2,\alpha}  \right]. $
   \end{itemize}
\end{lemma}
\v\n{\bf Proof.} Denote $(\kappa,h_{1},h_{2})=Q_{(\kappa_{o},b_o)}H_{r}P_{r}(\eta_1,\eta_2).$  Note that  $\supp\; H_{r}X\subset \{\frac{r}{2}\leq |s_{i}|\leq r+1 \}.$ From the definition of $X$ (see \eqref{def.X}) we have
$$\|H_{r}X(\eta_1,\eta_2)\|_{k-1,2,\alpha}\leq C \|(h_{1},h_{2})|_{\frac{r}{2}\leq |s_{i}|\leq r+1}\|_{k,2,\alpha}+C\|(\eta_{1},\eta_{2})|_{\frac{r}{2}\leq |s_{i}|\leq r+1}\|_{k,2,\alpha}.$$
Then (1) follows from \eqref{eqn_HX} and (a) in Lemma \ref{lem_def_HP}.
Taking derivative $\frac{\p}{\p r}$ to $H_{r}X(\eta_{1},\eta_{2})$ we get
$$
\frac{\p (H_{r}X)}{\p r}(\eta_{1},\eta_{2})=(\sum\xi^{i}_{1} h_{i}+\sum\zeta^{i}_{1}\eta_{i},\sum\xi^{i}_{2} h_{i}+\sum\zeta^{i}_{2}\eta_{i})+(\sum\rho^{i}_{1} \frac{\p h_{i}}{\p r},\rho_{2}^{i} \frac{\p h_{i}}{\p r})$$
$$+(\sum\chi^{i}_{1} \frac{\p h_{i}}{\p t},\chi_{2}^{i} \frac{\p h_{i}}{\p t})+(\sum\lambda^{i}_{1} \frac{\p^2 h_{i}}{\p r\p t},\lambda_{2}^{i} \frac{\p^2 h_{i}}{\p r\p t}),
$$
where
$$
\xi_{l}^{i}=\frac{\p}{\p r}\left(\beta_{l;2}(\bar{\partial}\beta_{i;r} +\beta_{i;r}(F^{1}_{u_{(r)}}-F^{1}_{u_{i}}))\right),\;\;
\rho_{l}^i=\beta_{l;2}(\bar{\partial}\beta_{i;r} +\beta_{i;r}(F^{1}_{u_{(r)}}-F^{1}_{u_{i}}))
$$
$$
\zeta_{l}^{i}=\frac{\p}{\p r}\left(\beta_{l;2}\beta_{i;2}(\sum_{e=1,2} \beta_{e;r}\beta_{e;2} -1)\right),\;\;
\chi^{i}_{l}=\frac{\p}{\p r}\left(\beta_{l;2}\beta_{i;r}(F^{2}_{u_{(r)}}-F^{2}_{u_{i}})\right),\;\;\lambda_{l}^{i}=\beta_{l;2}\beta_{i;r}(F^{2}_{u_{(r)}}-F^{2}_{u_{i}}).
$$
 By the same proof of (d) in Lemma \ref{lem_def_HP}, we have
\begin{align}\label{eqn_est_II_2}
&  \left\|\frac{\p (H_rX)}{\p r}(\eta_1,\eta_{2})\right\|_{k-2,2,\alpha}\\
\nonumber \leq & C\left(\left\|\frac{\p}{\p r}(h_1,h_2)|_{\frac{r}{2} \leq |s_{i}|\leq  r+1}\right\|_{k-1,2,\alpha}+\left\|(h_1,h_2)|_{\frac{r}{2} \leq |s_{i}|\leq  r+1}\right\|_{k,2,\alpha}+\|(\eta_{1},\eta_{2})|_{\frac{r}{2}\leq |s_{i}|\leq r+1}\|_{k,2,\alpha}\right) \\
\nonumber \leq &C\left[\left\|(\eta_1,\eta_{2})|_{\frac{r}{4} \leq |s_{i}|\leq r+1}\right\|_{k-1,2,\alpha}+e^{-(\fc-\alpha)\frac{r}{4}} \|  (\eta_1,\eta_{2})|_{  |s_{i}|\leq r+1}\|_{k-1,2,\alpha}  \right],
\end{align}
where we have used \eqref{eqn_est_II_1} and \eqref{eqn_HX} in the last inequality. Then (2) is proved.  $\Box$

\v

\subsection{Estimates of $\left\|\frac{\p}{\p r}\left(H_{r}\circ (D{\mathcal S}_{( \kappa_{o},b_{(r)})}\circ Q_{( \kappa_{o},b_{(r)})}^ {\prime})^{-1} P_{r}\right)\right\|_{k-2,2,\alpha}$}

\begin{lemma} \label{est_DQ} There exists a constant $\mathsf C>0,$
  such that for any $(\eta_1,\eta_{2})\in \lka$	
the following estimates hold:
\begin{itemize}
\item[(A)]	$H_{r}  (DQ')^{-1} P_{r}(\eta_1,\eta_{2})|_{|s_i|\leq \frac{r}{2}}=(\eta_1,\eta_{2})|_{|s_i|\leq \frac{r}{2}} $,
\item[(B)]$\left\|H_{r}  (DQ') ^{-1} P_{r}(\eta_1,\eta_{2})|_{\frac{r}{2}  \leq |s_{1}|\leq \frac{3r}{2}}\right\|_{k-1,2,\alpha}  $
     $$\leq \mathsf C\left[ e^{-(\fc-\alpha)\frac{r}{4}} \|  (\eta_1,\eta_{2})|_{ |s_{i}|\leq  r+1}\|_{k-1,2,\alpha} + \left\|(\eta_1,\eta_{2})|_{\frac{r}{4}\leq |s_{i}|\leq  r+1}\right\|_{k-1,2,\alpha} \right],$$
\item[(C)]
$\left\|\frac{\p}{\p r}\left(H_{r}  (DQ') ^{-1} P_{r}\right)(\eta_1,\eta_{2})\right\|_{k-2,2,\alpha}$ $$
 \leq \mathsf C \left[ e^{-(\fc-\alpha)\frac{r}{4}} \|  (\eta_1,\eta_{2})|_{ |s_{i}|\leq  r+1}\|_{k-1,2,\alpha} + \left\|(\eta_1,\eta_{2})|_{\frac{r}{4}\leq |s_{i}|\leq  r+1}\right\|_{k-1,2,\alpha} \right].$$
	 \end{itemize}
\end{lemma}
\n
{\bf Proof.} For any $(\eta_{1},\eta_{2})\in \lka ,$
let $\eta_{r}:=P_{r}(\eta_{1},\eta_{2})=\sum_{i=1}^{2} \beta_{i;2}\eta_{i}.$
Denote
$$Q_{(\kappa_{o},b_o)}\left(H_{r}   \eta_{r} \right)=(\kappa,h_1,h_2).$$ Then $ Q'\eta_{r}=(\kappa,\sum_{i=1}^{2} \beta_{i;r}h_{i}).$ Let   $\tilde \eta_{r}= (DQ')^{-1}\eta_{r}.$ By the definition of $X$ and the invertibility of $DQ'$, we have
\begin{align}\label{app_diff_0}
 \eta_{r}-\tilde \eta_{r}=(DQ')\tilde \eta_{r}-\tilde \eta_{r}= X(H_{r}\tilde \eta),\;\;\;\;\eta_{r}-\tilde \eta_{r} =(DQ')^{-1}(DQ'-I)\eta_{r}=(DQ')^{-1}  X(\eta_1,\eta_2).
\end{align}
It follows   that $\tilde \eta_{i}|_{|s_{i}|\leq \frac{r}{2}}=\eta_{i}|_{|s_{i}|\leq \frac{r}{2}}.$ Then (A) follows.

\v  By \eqref{app_diff_0} and $\|H_{r}\eta\|_{k-1,2,\alpha}  =\|\eta\|_{k-1,2,\alpha,r}$, we have
\begin{align*}
&\left\|H_{r}\tilde\eta_{r}|_{\frac{r}{2}  \leq |s_{1}|\leq \frac{3r}{2}}\right\|_{k-1,2,\alpha}
\leq  \|H_{r} \eta|_{|s_{i}|\geq \frac{r}{2}}\|_{k-1,2,\alpha}+\| (DQ')^{-1}X(\eta_1,\eta_2)\|_{k-1,2,\alpha,r} \\
\leq &C \left[\|\eta_{i}|_{\frac{r}{2}\leq  |s_{i}|\leq r+1}\|_{k-1,2,\alpha} +\| X(\eta_1,\eta_{2}) \|_{k,2,\alpha,r}\right]  \\=&C \left[\|\eta_{i}|_{ \frac{r}{2}\leq |s_{i}|\leq r+1}\|_{k-1,2,\alpha} +\| H_{r}X(\eta_1,\eta_{2}) \|_{k,2,\alpha}\right]\\
 \leq &C \left[ e^{-(\fc-\alpha)\frac{r}{4}} \|  (\eta_1,\eta_{2})|_{ |s_{i}|\leq  r+1}\|_{k-1,2,\alpha} + \left\|(\eta_1,\eta_{2})|_{\frac{r}{4}\leq |s_{i}|\leq  r+1}\right\|_{k-1,2,\alpha} \right],
\end{align*}
where we have used  (a) of Lemma \ref{lem_def_HP}   and the bound  of $\|(DQ')^{-1}\|$ in the second inequality. We used   (1) of Lemma \ref{lem_HX}  in the last inequality.
Then (B) follows.

\v
We prove (C).  Multiplying $H_{r}$ on both sides of \eqref{app_diff_0}  and taking the derivative $\frac{\p}{\p r}$ we have
\begin{align}\label{eqn_r_H_2}
&\frac{\p}{\p r}(H_{r}(\tilde\eta_{r}))=\frac{\p}{\p r}(H_{r}(\eta_{r}))-\frac{\p}{\p r} \left[H_{r}(D  Q')^{-1} P_{r} \circ H_{r} X(\eta_1,\eta_2)\right].
\end{align} On the other hand, by
$H_{r}(DQ') P_{r} \circ H_{r}(DQ')^{-1}  X(\eta_1,\eta_2)=H_{r}X(\eta_1,\eta_2),$  we get
\begin{align}\label{eqn_r_H_0}
&H_{r}  P_{r}\circ \frac{\p}{\p r} \left[H_{r}  (D Q')^{-1} X(\eta_1,\eta_2) \right] \\\nonumber
=&H_{r}  (D Q')^{-1}  P_r\circ\frac{\p}{\p r}\left[H_r  X(\eta_1,\eta_2)\right]  -H_{r}  (D Q')^{-1}  P_r \frac{\p}{\p r}\left[H_{r}(D  Q') P_{r}\right]\circ H_{r}  (D Q')^{-1}  X(\eta_1,\eta_2).
\end{align}
Note that
\begin{align*}
&\frac{\p}{\p r} \left[H_{r}  (D Q')^{-1} X(\eta_1,\eta_2) \right] =H_{r}  P_{r}\circ \frac{\p}{\p r} \left[H_{r}  (D Q')^{-1} X(\eta_1,\eta_2) \right]+\frac{\p}{\p r}(H_{r}P_{r})H_{r}  (D Q')^{-1} X(\eta_1,\eta_2).
\end{align*}
Then inserting \eqref{eqn_r_H_0} into \eqref{eqn_r_H_2} we get
\begin{align*}
&\frac{\p}{\p r}(H_{r}\tilde\eta) =\frac{\p}{\p r}(H_{r}P_r)(\eta_{1},\eta_{2}) + (I)+ (II) + (III),
\end{align*}
where
\begin{align*}
 &(I)=-\frac{\p( H_{r}P_{r} )}{\p r} \circ H_{r} (D  Q')^{-1} P_{r}  \circ H_{r} X(\eta_1,\eta_2),\;\;\;\;\;\;\;\\&(II)= -H_{r}(D  Q')^{-1} P_{r} \circ \frac{\p (H_rX)}{\p r}(\eta_1,\eta_2)\\
  &(III)=H_{r}(D  Q')^{-1} P_{r}\circ \frac{\p}{\p r}\left(H_{r}(D  Q') P_{r} \right)\circ H_{r}(D  Q')^{-1} P_{r}  \circ H_{r} X(\eta_1,\eta_2).
\end{align*}
By (b) of Lemma \ref{lem_def_HP} we have,
\begin{align}\label{eqn_est_HP}
\left\|\frac{\p}{\p r}(H_{r}P_r)(\eta_{1},\eta_{2} )\right \|_{k-2,2,\alpha} \leq C \|(\eta_{1},\eta_{2})|_{r-1 \leq |s_{i}|\leq r+1}\|_{k-1,2,\alpha}.
\end{align}
\v
Next we calculate (I), (II) and (III).
\v\n
{\bf Calculation for (I).}
Using \eqref{eqn_est_HP} with $(\eta_{1},\eta_{2} )$ replaced by $\left(H_{r}(D  Q')^{-1} P_{r}  \circ H_{r} X(\eta_{1},\eta_{2})\right)$, we obtain that
\begin{align} \label{eqn_est_I0}
  \| H_{r} (D  Q')^{-1} P_{r}  \circ H_{r} X(\eta_{1},\eta_{2})\|_{k-1,2,\alpha}
\leq  C  \|   H_{r} X(\eta_{1},\eta_{2})\|_{k-1,2,\alpha}
\end{align}
where we have used (A), (B) with $(\eta_1,\eta_{2})$ replaced by  $(H_{r}X(\eta_1,\eta_2))$ and $\supp\; H_{r}X(\eta_1,\eta_2)\subset \{\frac{r}{2} \leq |s_{i}|\leq \frac{3r}{2}\}.$ Then
\begin{align} \label{eqn_est_I}
\| \mbox{I} \|_{k-2,2,\alpha}   \leq C  \|   H_{r} X(\eta_{1},\eta_{2})\|_{k-1,2,\alpha}
\end{align}
\v\n
{\bf Calculation for (II).} Applying (A) and (B) with $(\eta_1,\eta_{2})$ replaced by $\frac{\p (H_rX)}{\p r}(\eta_1,\eta_2) $, we have
\begin{equation}\label{eqn_est_II_2}
\|\mbox{II}\|_{k-2,2,\alpha}\leq C \left\|\frac{\p (H_rX)}{\p r} (\eta_{1},\eta_{2})\right\|_{k-2,2,\alpha}.
\end{equation}
\v\n
{\bf Calculation for (III).} Set $\xi=H_{r}(D  Q')^{-1} P_{r}  \circ H_{r} X (\eta_{1},\eta_{2}).$ Multiplying $H_{r}$ on both sides of \eqref{ass_X} and taking the derivative $\frac{\p}{\p r}$ we have
\begin{align}\label{eqn_HX_r1}
\frac{\p (H_rX)}{\p r} \xi+\frac{\p(H_{r}P_{r})}{\p r} \xi=\frac{\p}{\p r}(H_{r}(D  Q') P_{r})\xi,
\end{align}
where $(\eta_1,\eta_{2})$ is replaced by $\xi.$
Since $\supp \;\frac{\p}{\p r}(H_{r}(D  Q') P_{r})(\eta_{1},\eta_{2})\subset\{\frac{r}{2} \leq |s_{i}|\leq r+1\}$, using \eqref{eqn_HX_r1}, (b) of Lemma \ref{lem_def_HP} and (2) of Lemma \ref{lem_HX} we get
\begin{align}\label{eqn_H_r}
\left\|\frac{\p}{\p r}(H_{r}(D  Q') P_{r})\xi\right\|_{k-2,2,\alpha}
\leq  C  \left[\left\|\xi|_{\frac{r}{4} \leq |s_{i}|\leq  r+1}\right\|_{k-1,2,\alpha}+e^{-(\fc-\alpha)\frac{r}{4}} \|  \xi|_{  |s_{i}|\leq  r+1}\|_{k-1,2,\alpha}  \right].
\end{align}
 Applying (A) and (B) with $(\eta_1,\eta_{2})$ replaced by $ \frac{\p}{\p r}\left(H_{r}(D  Q')P_{r} \right)\xi ,$ we have
\begin{align}\label{eqn_III_est}
&\|\mbox{III}\|_{k-2,2,\alpha}\leq C \left \|\frac{\p}{\p r}\left(H_{r}(D  Q') P_{r} \right)\xi \right\|_{k-2,2,\alpha} \\
\nonumber
&\leq C \left[\| H_{r}X(\eta_{1},\eta_{2}) \|_{k-1,2,\alpha}+e^{-(\fc-\alpha)\frac{r}{4}} \|   H_{r}X(\eta_{1},\eta_{2})  \|_{k-1,2,\alpha}\right] \nonumber
\end{align}
where we have used \eqref{eqn_H_r} and \eqref{eqn_est_I0} in the last inequality.

Combining the estimates
\eqref{eqn_est_HP}, \eqref{eqn_est_I}, \eqref{eqn_est_II_2} and \eqref{eqn_III_est}  , we obtain
\begin{align*}
&\|\frac{\p}{\p r}\left(H_{r}(D  Q')^{-1} P_{r}\right)(\eta_1,\eta_{2})\|_{k-2,2,\alpha} \\
\leq&   C  \left[\|(\eta_{1},\eta_{2})|_{\frac{r}{2} \leq |s_{i}|\leq  r+1}\|_{k-1,2,\alpha} +\left\|\frac{\p (H_rX)}{\p r} (\eta_{1},\eta_{2})\right\|_{k-2,2,\alpha}\right]  +C  \| H_{r}X(\eta_{1},\eta_{2})  \|_{k-1,2,\alpha}  .
\end{align*}
where we have used $\supp \;H_{r}X(\eta_1,\eta_2)\subset \frac{r}{2} \leq |s_{i}|\leq  r+1$.
The lemma follows from (1) and (2) of Lemma \ref{lem_HX} $\Box$

\v

\subsection{Estimates of $\frac{\partial I_{r}^*}{\partial r}$}

\begin{lemma}\label{lem_est_rI-1} There exists a constant $\mathsf C>0,$ independent of $r,$ such that  for any $(\kappa,h+\hat h_{0})\in \ker D\mathcal {S}_{( \kappa_{o}, b_{o})}$.
 \begin{equation}
 \label{partial_I_estimate} \left\|\tfrac{\partial}{\partial r} I^*_{r}(\kappa,h+\hat h_{0})
 \right\|_{k-1,2,\alpha}
\leq \mathsf C \|h_{i}|_{\frac{r}{2}\leq |s_i|\leq \frac{3r}{2}}\|_{k  ,2,\alpha}+\mathsf C e^{(\fc-\alpha)\frac{r}{2}} |\hat h_0|.
 \end{equation}
\end{lemma}
\begin{proof}
 Recall that $I^{*}_{r}(\kappa,h+h_{0})=(\kappa,h_{r}^*)-Q^{*} D (\kappa,h_{(r)})
.$ By  $P_{r}H_{r}=Id$ we have
\begin{align*}\tfrac{\partial }{\partial r}I^*_{r}( \kappa,h,h_{0})=&\left(0,\tfrac{\p \beta_{1;r}}{\p r}h_1,\tfrac{\p \beta_{2;r}}{\p r}h_2\right)+\tfrac{\partial }{\partial r}\left(Q^*P_{r}\right)\circ H_{r} D(\kappa,h_{(r)}) \\
&+Q^*P_{r}\circ\tfrac{\partial }{\partial r}\left(H_{r}D(\kappa,h_{(r)}) \right).
\end{align*}
Note that $d\mathfrak{i}_{( \kappa_{o},b_{(r)})}(\kappa,h_{(r)})=d\mathfrak{i}_{( \kappa_{o}, b_{o})}(\kappa,h_{i}+\hat h_0).$
By   \eqref{eqn_a_DK}, equalities $F^{k}_{u_{(r)}}|_{|s_{i}|\leq \frac{r}{2}}=F^{k}_{u_{i}}|_{|s_{i}|\leq \frac{r}{2}},i,k=1,2$ and the definition of $\beta_i$ we have $\supp \; D(\kappa,h_{(r)})\subset \left\{  \frac{r}{2}\leq |s_i|\leq \frac{3r}{2}\right\}.$
Then
\begin{equation}\label{HD-1}
\|H_{r} \circ D(\kappa,h_{(r)})\|_{k-1,2,\alpha}\leq C \|h_{i}|_{\frac{r}{2}\leq |s_i|\leq \frac{3r}{2}}\|_{k,2,\alpha}+Ce^{-(\fc-\alpha)\frac{r}{2}} |  h_0|,
\end{equation}
\begin{equation}\label{HD-2}
\left\|\frac{\p}{\p r}\left(H_{r}D(\kappa,h_{(r)})\right)\right\|_{k-2,2,\alpha}\leq C \|h_{i}|_{\frac{r}{2}\leq |s_i|\leq \frac{3r}{2}}\|_{k  ,2,\alpha}+Ce^{-(\fc-\alpha)\frac{r}{2}} |  h_0|.
\end{equation}
Note that
\begin{equation}\label{Q*P-1}
Q^*P_r=(Q')^*P_r\circ H_r(DQ')^{-1}P_r\end{equation}
\begin{equation}\label{Q*P-2}
\frac{\partial}{\partial r}(Q^*P_{r})=\frac{\partial}{\partial r}((Q')^*P_{r})\circ H_r(DQ')^{-1}P_r
+ (Q')^*P_{r}\frac{\partial}{\partial r}\left(H_r(DQ')^{-1}P_r\right).\end{equation}
Then lemma follows from \eqref{HD-1}, \eqref{HD-2} and
 Lemma \ref{lem_def_HP}, Lemma \ref{est_DQ}.
\end{proof}

\subsection{Estimates of $\frac{\partial }{\partial r}\left[I^*_r(\kappa,\zeta)+Q^*_{( \kappa_{o},b_{(r)})}f_{(r)}(I_r(\kappa,\zeta))\right]$}
\label{exp est}
Let $\mathsf d>0$ be a small constant such that
$$
\|(\kappa,\zeta)\|_{k,2,\alpha}\leq \mathsf d.
$$
We set
\begin{align} \label{kernel_express}
(\kappa^*_{r},\xi^*_{(r)}):=I^*_r(\kappa,\zeta)+Q^*_{( \kappa_{o},b_{(r)})}\circ f_{(r)}(I_r(\kappa,\zeta)), \\ \label{ker}
(\kappa_{r},\xi_{(r)}):=I_r(\kappa,\zeta)+Q_{( \kappa_{o},b_{(r)})}\circ f_{(r)}(I_r(\kappa,\zeta)).
\end{align}
where $\kappa_{r}\in K ,\; \xi^*_{(r)}\in \wka.$ Denote  $v_{(r)}=\exp_{u_{(r)}}\xi_{(r)}$ and $b_r=(r,\tau,a_{o},v_{(r)})$. We have
\begin{equation}
\label{local_kernel_express}
\bar{\partial }_{j_{o},J} v_{(r)}+ \mathfrak{i}( \kappa_{o}+\kappa_{r},v_{(r)}) =0.
\end{equation}

\v

From  the Implicity Function Theorem (see Theorem \ref{details_implicit_function_theorem}), Lemma \ref{lem_est_I_r}, Lemma \ref{difference_DF_estimates_gw} and Lemma \ref{lem_est_al_hol} we have
\begin{align}
\label{local_calculation_2}
 &\|f_{(r)}(I_r(\kappa,\zeta))\|_{k-1,2,\alpha,r} =\left\|DQf_{(r)}(I_r(\kappa,\zeta))  \right\|_{k-1,2,\alpha,r} \leq C\left\|Qf_{(r)}(I_r(\kappa,\zeta))  \right\|_{k,2,\alpha,r}
  \\ \leq & C\left\|\mathcal F_{(r)}(I_r(\kappa,\zeta))  \right\|_{k-1,2,\alpha,r} \leq C\left\|\mathcal F_{(r)}(0) \right\|_{k-1,2,\alpha,r}+C\left\|d\mathcal F_{(r)}(\theta I_r(\kappa,\zeta)) (I_r(\kappa,\zeta))\right\|_{k-1,2,\alpha,r} \nonumber \\
 \leq & C\left(\|(\kappa,\zeta)\|_{k,2,\alpha}+e^{-\frac{(\fc-\alpha)r}{2}}\right). \nonumber \end{align}
where we use  the  intermediate value theorem in the third inequality and  $\theta\in(0,1).$
  It follows that
$\|(\kappa_r,\xi_{(r)})\|\leq   C\mathsf{d}$ as $r$ large enough.

For any small $(\kappa,\zeta)\in \ker D{\mathcal S}_{( \kappa_{o}, b_{o})}$, $\exp_{u_{i}}\zeta$ converges to a point as
  $|s_i|\to \infty$ ( see Lemma \ref{tube_ker_c-1}).
It follows that $F^{l}_{\exp_{u_{i}}\zeta}, l=1,2$, converges to zero exponentially.
By Theorem \ref{tube_exponential_decay} and the definition of $u_{(r)}$ we have
  \begin{align} \label{exp_J_map_3}
&\left|d u_{(r)}\right|\leq \mathsf Ce^{-\fc |s_{i}|},\;\;\;\mbox{for any}\; \frac{r}{4}\leq |s_{i}|  \leq  {r}.
\end{align}
 By choosing $\mathsf{d}$ small and $R_{0}$ large we have
 $$E (v_{(r)};|s_{i}|\geq R_{0})\leq \hbar,$$
where $\hbar$ is the constant in Theorem \ref{tube_exponential_decay}. Then we have for any $\tfrac{r}{4}\leq |s_{i}|\leq r$
\begin{align}\label{exp_J_map_1}
|d v_{(r)} |\leq \mathsf Ce^{-\fc |s_{i}|}.
\end{align}
By \eqref{exp_J_map_3}, \eqref{exp_J_map_1}, \eqref{expression_of_S} and \eqref{expression_of_F}  we conclude that in the part $ \tfrac{r}{2}\leq |s_i|\leq \tfrac{3r}{2} $
\begin{equation}
\label{local_calculation_3}
 |F^{l}_{v_{(r)}}|\leq C e^{-\mathfrak{c} \tfrac{r}{2}}, \;\;\;\;|F^{l}_{u_{(r)}}|\leq C e^{-\mathfrak{c}  \tfrac{r}{2}},\;\;\;\;l=1,2,
\end{equation}
 for some constant $C>0$, when $r>4R_{0}$.

\v
Next we   recall a fact about the exponential map on a compact Riemannian manifold $M$ (see \cite{MS}, Page 362, Remark 10.5.5). There
are two smooth families of endomorphisms
$$
E_{i}(p,\xi) : T_{p}M\rightarrow T_{\exp _{p}\xi} M ,\;\;\; i=1,2,
$$
that are characterized by the following property.
Let $\gamma :\mathbb R\to M $ be any smooth path in $M$ and $v(t)\in T_{\gamma(t)}M$ be any smooth vector field along this path then the derivative of the path $t\to \exp_{\gamma(t)}(v(t))$ is given by the formula
$$
\frac{d}{dt} \exp_{\gamma}(v)=E_{1}(\gamma,v)\dot\gamma+E_{2}(\gamma, v)\tilde \nabla_{t}v,
$$
where $\dot \gamma =\frac{d\gamma}{dt}.$
We have
$$
E_{1}(p,0)=E_{2}(p,0)=Id:T_{p}M\to T_{p}M,\;\;\;\forall p\in M,
$$
and $E_{i}(p, \xi)$ are uniformly invertible for sufficiently small $\xi.$ Since $M$ is compact, there exists a constant $ \epsilon$ such that for any $p\in M$ and $\xi\in T_{p}M$ with $|\xi|_{T_{p}M}\leq \epsilon$, $E_{i}(p, \xi)$ are  uniformly invertible.

\begin{lemma}\label{lem_est_f_r}
There exist two constant $ \mathsf C>0$ such that for any $(\kappa,\zeta)\in \ker D\mathcal{S}_{( \kappa_{o}, b_{o})}$
we have
\begin{equation}\label{decay_exp_J_map}
\left\|H_{r}f_{r}\circ I_{r}(\kappa,\zeta)|_{  |s_{i}|\geq \frac{r}{4}}\right\|_{k-1,2,\alpha }\leq   \mathsf Ce^{-(\fc-\alpha)\frac{r}{4}} (1+\|(\kappa,\zeta)\|),\;\;\;\;\forall \;r\geq 8R_{0}.
\end{equation}
\end{lemma}
\begin{proof}  Let $D\mathcal{S}_{( \kappa_{o}, b_{(r)})}$ act on \eqref{ker}. Since $d\mathfrak{i}_{(\kappa_{o},b_{o})}|_{R_{0}\leq |s_{i}|\leq 2r-R_{0}}=0 $  we get,  in  $\{R_{0} \leq |s_{i}|\leq 2r-R_{0}\}$
\begin{equation}\label{express_f_r}
f_{(r)}(I_r(\kappa,\zeta))=\bar{\p }_{j_{o},J_{0}} \xi_{(r)}+F^{1}_{u_{(r)}}\xi_{(r)}+F^{2}_{u_{(r)}}\p_{t}\xi_{(r)},\;\;\;\;\;\;\;\bar{\p}_{j,J} v_{(r)}=0,
\end{equation}
 Since  $$d v_{(r)}=E_{1}(u_{(r)},\xi_{(r)})(d u_{(r)})+E_{2}(u_{(r)},\xi_{(r)})(\tilde \nabla  \xi_{(r)}),$$   by \eqref{exp_J_map_3},  \eqref{exp_J_map_1}  and the elliptic estimate  we have
\begin{equation}
\left\|E_{2}(u_{(r)},\xi_{(r)})(\tilde \nabla \xi_{(r)})\right\|_{k-1,2,\alpha}\leq Ce^{-\fc |s_{i}|},\;\;\;\;{\frac{r}{4}\leq |s_{i}|\leq r}.
\end{equation}
Note that $[E_{2}(u_{(r)},\xi_{(r)})]^{-1}$ is uniformly bounded as $\|\xi_{(r)}\|_{k,2,\alpha,r}$ is small.
Then
\begin{equation}\label{a_nabla_h}
\left\| \tilde  \nabla  \xi_{(r)} |_{ |s_{i}|\geq \frac{ r}{4}}\right\|_{k-1,2,\alpha,r}\leq Ce^{-(\fc-\alpha)\frac{r}{4} },\end{equation}
for some constant $C>0.$
Note that
$$\tilde \nabla_{s} \xi_{(r)}=\frac{\p \xi_{(r)}}{\p s} -\sum_{i,j,l}\tilde{\Gamma}^{l}_{ij}\frac{\p u^{i}_{(r)}}{\p s} \xi_{(r)}^{j}\frac{\p }{\p x^{l}},$$
where $\tilde \nabla_{\frac{\p }{\p x^{i}}}\frac{\p }{\p x^{j}}= \sum_{l}\tilde{\Gamma}^{l}_{ij}\frac{\p }{\p x^{l}}.$
The lemma follows from \eqref{express_f_r}, \eqref{a_nabla_h} and the exponential decay of $F^{k}_{u_{(r)}},k=1,2.$
\end{proof}

\begin{lemma}\label{lem_xi^*_app_1}
There exists a constant $\mathsf C>0$ such that for any $(\kappa,\zeta)\in \ker D\mathcal{S}_{( \kappa_{o}, b_{o})}$   we have
$$\left\|\frac{\p}{\p s_{i}}(\xi^*_{(r)})_{i}|_{\frac{r}{2}\leq |s_{i}|\leq \frac{3r}{2}}\right\|_{k-1,2,\alpha}\leq \mathsf C e^{-(\fc-5\alpha)\frac{r}{4}}(\|\zeta\|_{k,2,\alpha}+1).$$
\end{lemma}
\begin{proof} From the definition of $I^*_{r}$ ( see \eqref{def_I*}) and \eqref{kernel_express} we have, in $\{|s_{i}|\geq \frac{r}{4}\}$,
$$(\kappa_{r},\xi^*_{(r)})=(\kappa,\zeta^{*}_{r})- Q^{*}P_{r}\circ H_{r}D(Id,\phi_{r})(\kappa,\zeta^{*}_{r}) + Q^* P_{r}\circ H_{r}f_{r}I_{r}(\kappa,\zeta)
=$$$$(\kappa,\zeta^{*}_{r})- (Q')^*P_r\circ H_r(DQ')^{-1}P_r\circ H_{r}D(Id,\phi_{r})(\kappa,\zeta^{*}_{r}) + (Q')^*P_r\circ H_r(DQ')^{-1}P_r\circ H_{r}f_{r}I_{r}(\kappa,\zeta).
$$
Note that
\begin{align}\label{eqn_xi_r*}
\left\| \zeta^{*}_{r}|_{\frac{r}{2}\leq |s_{i}|\leq \frac{3r}{2}}\right\|_{k,2,\alpha} \leq Ce^{\alpha r}   \left\| \zeta|_{\frac{r}{2}\leq |s_{i}|\leq \frac{3r}{2}}\right\|_{k,2,\alpha} \leq C e^{-(\fc-5\alpha)\frac{r}{4}}\left\| \zeta\right\|_{k,2,\alpha},
\end{align}
where we have applied Lemma \ref{tube_ker_c-1} with $R'=\frac{r}{2}$, $R=\frac{r}{4}$ and $\eta=0$ in the last inequality.
\v
By (e) of Lemma \ref{lem_def_HP},   Lemma \ref{est_DQ}, \eqref{local_calculation_2} and Lemma \ref{lem_est_f_r} we have
\begin{equation}
\|(Q')^*P_r\circ H_r(DQ')^{-1}P_r\circ H_{r}f_{r}I_{r}(\kappa,\zeta)|_{\frac{r}{2}\leq |s_{i}|\leq \frac{3r}{2}}\|_{k,2,\alpha}\leq Ce^{-(\fc-\alpha)\frac{r}{4}}(\|(\kappa,\zeta)\|+1).
\end{equation}
Similar, by (e) of Lemma \ref{lem_def_HP},   Lemma \ref{est_DQ} and  \eqref{HD-1}  we have
$$
\|(Q')^*P_r\circ H_r(DQ')^{-1}P_r\circ H_{r}D(Id,\phi_{r})(\kappa,\zeta^{*}_{r})|_{\frac{r}{2}\leq |s_{i}|\leq \frac{3r}{2}}\|_{k,2,\alpha}\leq Ce^{-(\fc-\alpha)\frac{r}{4}}(\|(\kappa,\zeta)\|+1).
$$
Using  Lemma \ref{lem_est_HDS} and Lemma \ref{lem_est_f_r} we have
\begin{align*}
&\left\|\frac{\p}{\p s_{i}}(\xi^*_{(r)})_{i}|_{\frac{r}{2}\leq |s_{i}|\leq \frac{3r}{2}}\right\|_{k-1,2,\alpha}
\leq   C\left(\left\| \zeta^{*}_{r}|_{\frac{r}{2}\leq |s_{i}|\leq \frac{3r}{2}}\right\|_{k,2,\alpha} +e^{-(\fc-\alpha)\frac{r}{4}}(\|(\kappa,\zeta)\|_{k,2,\alpha}+1)\right)\\
\leq &  Ce^{-(\fc-5\alpha)\frac{r}{4}}(\|(\kappa,\zeta)\|_{k,2,\alpha}+1)
\end{align*}
where we have used  \eqref{eqn_xi_r*} in the last inequality.
\end{proof}

\begin{lemma}\label{difference}
For any $\epsilon>0$,  there are two constants  $\mathsf C>0$ and $R_{1}>R_0$ depending only on $\epsilon,k,\alpha $ and the geometry of $M$, such that   for any  $r>R_{1}$  and $(\kappa,\zeta)\in \ker D\mathcal{S}_{( \kappa_{o}, b_{o})}$ with $\|(\kappa,\zeta)\|\leq \mathsf d,$  we have
 \begin{equation}
\label{local_calculation_4}
\left\|H_{r}\circ D(Id, \phi_{r})\circ \left(\frac{\partial}{\partial r}\left(\kappa^*_{r}, \xi^*_{(r)}\right)\right)\right\|_{k-2,2,\alpha} \leq \mathsf{C}\left( \mathsf {d}\left\|\frac{\partial}{\partial r}\left(\kappa^*_{r}, \xi^*_{(r)}\right)\right\|_{k-1,2,\alpha} + e^{-(\fc-5\alpha)\frac{r}{4} } \right).
\end{equation}
\end{lemma}
\v\n
{\bf Proof.} We estimate   $\left\|\beta_{1;2}  D (Id,\phi_{r})  \left(\frac{\partial}{\partial r}\left(\kappa^*_{r}, \xi^*_{(r)}\right)\right)\right\|_{k-1,2,\alpha}.$\;  The estimates of  $\left\|\beta_{2;2}  D (Id,\phi_{r})  \left(\frac{\partial}{\partial r}\left(\kappa^*_{r}, \xi^*_{(r)}\right)\right)\right\|_{k-1,2,\alpha}$ is the same. Let   $b=(r,\tau,j_{o},v_{(r)}).$
First we construct  $\tilde{u}_{(r)}$ and $\tilde{\xi}_{(r)}$   defined over $\Sigma_{1}$ as follows:
\begin{align}\label{defn-1}
\tilde{u}_{(r)}&=\left\{
\begin{array}{ll}
u_{(r)}, & if\ s_{1}\in \Sigma_{1}(r+1), \\
u_1(q)+\beta(r+2-s_1) (u_{(r)}(s_1,t_1)-u_{1}(q)), & if\ s_{1}\geq r+1
\end{array}
\right.
\\\label{defn-xi1}
\tilde{\xi}_{(r)}&=\left\{
\begin{array}{ll}
\xi_{(r)}, & if\ s_{1}\in \Sigma_{1}(r+1), \\
\beta(r+2-s_1) \xi_{(r)}(s_1,t_1), & if\ s_{1}\geq r+1
\end{array}
\right. .
\end{align}
Define $\tilde v_{(r)}=\exp_{\tilde{u}_{(r)}}\tilde{\xi}_{(r)}.$ So the meaning of $\frac{\p \tilde u_{(r)}}{\p r} ,\frac{\p \tilde v_{(r)}}{\p r} $ and $\nabla   \tilde \xi_{(r)} $ is clear.
 Set
$$ \Lambda_{r}:=P_{ b,  b_{(r)}}\left(\bar{\partial }_{j,J} v_{(r)}+ \mathfrak{i}(\kappa_o+\kappa_{r},b_{r})\right),\;\;\;\;\tilde\Lambda_{r}:=P_{\tilde b,\tilde b_{(r)}}\left(\bar{\partial }_{j,J} \tilde v_{(r)}+ \mathfrak{i}(\kappa_o+\kappa_{r},\tilde v_{(r)})\right).$$
where $\tilde b=(z,a_{o},\tilde v_{(r)})$  and $\tilde b_{(r)}=(z,a_{o},\tilde u_{(r)})$. We calculate $\frac{\p}{\p r}\left(\beta_{1;2}\Lambda_{r}\right)$. By
\eqref{local_kernel_express} we have $\Lambda_{r}=0$ and $\beta_{1;2}\Lambda_{r}=0.$ Then
  \begin{align*}
  \frac{\p}{\p r}\left(\beta_{1;2}\Lambda_{r}\right)&=  \frac{\p}{\p r}\left(\beta_{1;2}\tilde\Lambda_{r}\right)= \beta_{1;2}P_{\tilde b,\tilde b_{(r)}}\left(D_{\tilde v_{(r)}}\frac{\p\tilde v_{(r)}}{\p r}+d\mathfrak{i}_{(\kappa_o+\kappa_{r},\tilde b_{r})}\left(\frac{\p\kappa_{r}}{\p r}, \frac{\p \tilde v_{(r)}}{\p r} \right)\right) \\
  &= \beta_{1;2}P_{\tilde b,\tilde b_{(r)}}\left(D_{\tilde v_{(r)}}\left(\left(E_{1}(\tilde u_{(r)}),\tilde \xi_{(r)}\right)\frac{\p \tilde u_{(r)}}{\p r} +E_{2}\left(\tilde u_{(r)}),\tilde \xi_{(r)}\right)\tilde \nabla_{r} \tilde \xi_{r} \right)\right. \\
  &\;\;\;\left.+d\mathfrak{i}_{(\kappa_o+\kappa_{r},\tilde b_{r})}\left(\frac{\p\kappa_{r}}{\p r},\frac{\p \tilde v_{(r)}}{\p r} \right)\right)
  \end{align*}
 Since $$d\mathfrak{i}_{(\kappa_o+\kappa_{r},\tilde b_{r})}\left.\left(\frac{\p\kappa_{r}}{\p r},\frac{\p \tilde v_{(r)}}{\p r} \right)\right|_{|s_{i}|\geq R_{0}}=0,\;\;\;\frac{\p \tilde u_{(r)}}{\p r} |_{|s_{i}|\leq R_{0}}=0, \;\;\;\; \left.\tilde\nabla_{r} \tilde \xi_{r}\right|_{\Sigma(R_{0})}=\left.\phi_{r}\frac{\p \xi^{*}_{(r)}}{\p r}\right|_{\Sigma(R_{0})}$$ we can conclude that $$d\mathfrak{i}_{(\kappa_o+\kappa_{r},\tilde b_{r})}\left(\frac{\p\kappa_{r}}{\p r},\frac{\p \tilde v_{(r)}}{\p r} \right)=d\mathfrak{i}_{(\kappa_o+\kappa_{r},v)}
\left(\tfrac{\p \kappa_{r}}{\p r},E_{2}(u_{(r)},\xi_{(r)})\phi_{r} \left(\tfrac{\p \xi^*_{(r)}}{\p r}\right)  \right).$$
Note that in $\{s_{1}\leq r+1\}$
\begin{equation}\label{rel.eqn_xi}
\tilde \xi_{(r)}=\xi_{(r)}=\xi_{1}^*(s_{1},t_{1})+\xi_{2}^*(s_{1}-2r,t_{1}-\tau).
\end{equation}
Taking derivative $\nabla_{r}$ of \eqref{rel.eqn_xi}  we get, in $\{s_{1}\leq r+1\}$
$$
\tilde\nabla_{r} \tilde \xi_{(r)}=\phi_{r}\tilde \nabla_{r}   \xi^{*}_{(r)}-2\tilde \nabla_{s_{2}}(\xi^{*}_{(r)})_{2}.
$$
Then  by $v_{(r)}|_{|s_{i}|\leq r+1}=\tilde v_{(r)}|_{|s_{i}|\leq r+1}$ and  $u_{(r)}|_{|s_{i}|\leq r+1}=\tilde u_{(r)}|_{|s_{i}|\leq r+1}$ we get,
\begin{align}\label{eqn_beta_P1}
\tfrac{\p}{\p r}\left(\beta_{1;2} \Lambda_{r} \right)=P_{ b,b_{(r)}}\beta_{1;2}(s_1,t_1)\left((E)+ (G)+ (H) + (F)\right),\end{align}
where $$(E)=D_{v_{(r)}}\left(E_{2}(u_{(r)},\xi_{(r)})\phi_{r} \frac{\p  \xi^*_{(r)}}{\p r}  \right),\;\;\;\;(F)=d\mathfrak{i}_{(\kappa_o+\kappa_{r},\tilde b_{r})}
\left(\tfrac{\p \kappa_{r}}{\p r},E_{2}(u_{(r)},\xi_{(r)})\phi_{r} \left(\tfrac{\p \xi^*_{(r)}}{\p r}\right)  \right),$$
$$(G)=-2D_{v_{(r)}}\left(E_{2}(u_{(r)},\xi_{(r)})\nabla_{s_2} (\xi_{(r)}^{*})_{2}\right),\;\;\;\;$$$$
(H)=D_{v_{(r)}}\left(E_{1}({u_{(r)}},\xi_{(r)}) \tfrac{\p (\tilde u_{(r)} )}{\p r}  +E_{2}(u_{(r)},\xi_{(r)})\phi_{r}\left(\tilde\nabla_{\frac{\p}{\p r}}  \xi^*_{(r)}-  \frac{\p  \xi^*_{(r)}}{\p r}   \right)\right).
$$

 \v
By $\frac{\p }{\p r}(\beta_{1;2}\Lambda_{r} )=0$ and \eqref{eqn_beta_P1}  we have
\begin{align*}\label{est_o}
&\left\|\beta_{1;2}D(Id,\phi_{r})\left(\tfrac{\partial}{\partial r}\left(\kappa^*_{r}, \xi^*_{(r)}\right)\right)\right\|_{k-2,2,\alpha}\\ =&
\left\| \beta_{1;2} D(Id,\phi_{r})\left(\tfrac{\partial}{\partial r}\left(\kappa^*_{r}, \xi^*_{(r)}\right)\right)-\tfrac{\p}{\p r}\left(\beta_{1;2}  \Lambda_{r} \right)\right\| _{k-2,2,\alpha}
  \leq (\mbox{I}) + (\mbox{II})+(\mbox{III}) +(\mbox{IV}) ,
\end{align*}
where
\begin{equation}
(\mbox{I})=\left\|\beta_{1;2}\left(P_{b,b_{(r)}}D_{v_{(r)}}\left(E_{2}(u_{(r)},\xi_{(r)})\phi_{r}\frac{\p}{\p r} \xi^*_{(r)}  \right)-D_{u_{(r)}}\left(\phi_{r}\tfrac{\p\xi^*_{(r)} }{\p r} \right)\right)\right\|_{k-2,2,\alpha},\;\;
\end{equation}
\begin{equation}
(\mbox{II})= \left\| \beta_{1;2}  P_{b,b_{(r)}}(H)\right\|_{k-2,2,\alpha},\;\;\;\;
(\mbox{III})= \left\| \beta_{1;2}  P_{b,b_{(r)}}(G)\right\|_{k-2,2,\alpha},\;\;\end{equation}
\begin{equation}(\mbox{IV})=\left\|\beta_{1;2}P_{b,b_{(r)}}(F)
- \beta_{1;2}d \mathfrak{i}_{( \kappa_{o},u_{(r)})}\left(\tfrac{\p}{\p r}\kappa^*_{r}, \phi_{r}\tfrac{\p \xi^*_{(r)}}{\p r}\right)\right\|_{k-2,2,\alpha}.
\end{equation}
We calculate $(\mbox{I})$. There is a constant $C>0$ depending only on  the geometry of $M$ such that
\begin{align*}
(\mbox{I})\leq &C\|\xi_{(r)}\|_{k,2,\alpha,r} \left\|\beta_{1;2}\phi_{r}\tfrac{\p}{\p r}   \xi^*_{(r)}   \right\|_{k-1,2,\alpha}
\leq  C \|\xi_{(r)}\|_{k,2,\alpha,r} \left\| \tfrac{\p}{\p r}   \xi^*_{(r)}   \right\|_{k-1,2,\alpha}.
\end{align*}
Essentially, this estimate has been obtained by McDuff and D. Salamon ( see (3.5.5), P68,  \cite{MS})
we omit the proof here.
\v\n
Now we calculate $(\mbox{II})$. Since
$$ \tilde\nabla_{\frac{\p}{\p r}}  \xi^*_{(r)}-  \frac{\p  \xi^*_{(r)}}{\p r}=\sum \Gamma_{ij}^{k}\frac{\p u^{i}_{(r)}}{\p r} (\xi^{*}_{(r)})^{j}\frac{\p}{\p x^k},$$
by the definition of $u_{(r)}$ we conclude that $$\supp \beta_{1;2}P_{b,b_{(r)}}(H)\subset \{\tfrac{r}{2}\leq |s_{i}|\leq r+1\}.$$ Then by  \eqref{exp_J_map_3} and \eqref{local_calculation_3} we have
\begin{equation}\label{est_II}
(\mbox{II})\leq Ce^{-(\fc-\alpha)\frac{r}{2}}
\end{equation} for a constant $C>0$.
\v
It follows from Lemma \ref{lem_xi^*_app_1} that
$$
(\mbox{III})\leq  C\| \tfrac{\p}{\p s_{2}}(\xi_{(r)}^*)_{2} |_{ \frac{r}{2}\leq |s_{2}|\leq r+1}\|_{k-1,2,\alpha}\leq Ce^{-(\fc-5\alpha)\frac{r}{4}}(\|\zeta\|_{k,2,\alpha}+1).
$$
Finally we estimate $(\mbox{IV})$. Let $\kappa(\lambda)= \kappa_{o}+\lambda\kappa_{r}$, $v(\lambda)=\exp_{u(r)} (\lambda \xi_{(r)}) $ and $b(\lambda)=(\kappa(\lambda),v(\lambda)), \lambda\in[0,1]. $ In the following we omit the restriction $\{|s_{i}|\leq r+1\}.$ Since $\mathfrak{i}$ and parallel translation are smooth with respect to $(\kappa,u),$ we have
\begin{align}
(\mbox{IV})&=\left\|\beta_{1;2} P_{b,b_{(r)}}(F)-\beta_{1;2} d \mathfrak{i}_{( \kappa_{o},u_{(r)})}(\tfrac{\p\kappa_{r}}{\p r},\phi_{r}\tfrac{\p}{\p r}\xi_{(r)}^*)\right\|_{k-2,2,\alpha,r} \nonumber \\
&=\left\| \beta_{1;2} \int_{0}^1  \tfrac{d}{d\lambda} \left( P_{b(\lambda),b_{(r)}}d\mathfrak{i}_{(\kappa(\lambda),v(\lambda))}\left(\tfrac{\p \kappa_r}{\p r},E_{u_{(r)}}(\lambda\xi_{(r)})\phi_{r}\left(\tfrac{\p \xi^*_{(r)}}{\p r}\right)\right) \right)d\lambda \right\|_{k-2,2,\alpha,r}\nonumber \\
&\leq C\|(\kappa_{r},\xi_{(r)})\|_{k,2,\alpha}  \left\|\beta_{1;2}\left(\tfrac{\p \kappa_r}{\p r},  \phi_{r}\tfrac{\p \xi^*_{(r)}}{\p r}\right)  \right\|_{k-1,2,\alpha} \nonumber\\
&\leq C\|(\kappa_{r}^*,\xi^*_{(r)})\|_{k,2,\alpha}\left\| \left(\tfrac{\p \kappa_r^*}{\p r},  \tfrac{\p \xi^*_{(r)}}{\p r}\right)  \right\|_{k-1,2,\alpha}. \nonumber
\end{align}
Then the lemma follows from  the estimates of (I), (II), (III), (IV) and $\|(\kappa_{r}^*,\xi^*_{(r)})\|_{k,2,\alpha}\leq C\mathsf{d}$. $\Box$
\v

\v
\n{\bf Proof of Theorem \ref{lem_est_rI}.}
From \eqref{kernel_express} we have
\begin{align}
\label{partial_r_kernel_estimate}
 \frac{\partial}{\partial r} (\kappa^*_{r}, \xi^*_{(r)})= \frac{\partial}{\partial r} I^*_{r}(\kappa,\zeta) + \frac{\partial}{\partial r}\left (Q^* P_{r}\right) H_{r} f_{(r)}I_{r}(\kappa,\zeta)  + Q^* P_{r}\frac{\partial}{\partial r}  (H_{r} f_{(r)}I_{r}(\kappa,\zeta)).
\end{align}
Then multiplying $H_{r} D(Id,\phi_{r})$ on both sides of \eqref{partial_r_kernel_estimate} we get
 \begin{align*}
 H_{r} D(Id,\phi_{r})\left(\frac{\partial }{\partial r}(\kappa^*_{r},\xi^*_{(r)})\right)=&H_{r}D(Id,\phi_{r})\left(\frac{\partial }{\partial r}I^*_{r}(\kappa,\zeta)\right)+H_{r}P_{r}\frac{\partial}{\partial r}(H_{r}f_{(r)}I_{r}(\kappa,\zeta)) \nonumber\\&+ H_{r}D(Id,\phi_{r})\left(\frac{\partial}{\partial r} (Q^*P_{r})H_{r}f_{(r)}I_{r}(\kappa,\zeta)\right).
 \end{align*}
In the above calculation we have used $D(Id,\phi_{r})\circ Q^{*} =DQ=Id$.
It follows together with \eqref{local_calculation_4} that
\begin{equation}\label{operator_partial_right_inverse_kernel}
\left\|H_{r}P_{r}\frac{\partial}{\partial r}(H_{r}f_{(r)}(I_{r}(\kappa,\zeta)))\right\|_{k-2,2,\alpha}
\leq C\mathsf{d}\left\|\frac{\partial}{\partial r} (\kappa^*_{r},\xi^*_{(r)})\right\|_{k-1,2,\alpha}
+  Ce^{-(\fc-5\alpha)\tfrac{r}{4} }+ (A) + (B),\end{equation}
where
\begin{align*}(A)= &\left\| H_{r}D(Id,\phi_{r})\left(\frac{\partial}{\partial r}I^*_{r}(\kappa,\zeta)\right)\right\|_{k-2,2,\alpha},\;\;\;
 (B)=&\left\| H_{r} D(Id,\phi_{r})\left( \frac{\partial}{\partial r}(Q^*P_{r})\circ H_{r}f_{(r)}(I_{r}(\kappa,\zeta))\right)\right\|_{k-2,2,\alpha}.
 \end{align*}
By Lemma \ref{lem_est_HDS}, Lemma \ref{lem_est_rI-1}, taking  $R=\frac{r}{4}, R'=\frac{r}{2}$ in Lemma \ref{tube_ker_c-1}   we conclude that
\begin{align*}
 (A)\leq  &C \left\|\frac{\partial}{\partial r}I^*_{r}(\kappa,\zeta)\right\|_{k-1,2,\alpha}\leq   C   \|\zeta_{i}|_{\frac{r}{2}\leq |s_i|\leq \frac{3r}{2}}\|_{k  ,2,\alpha}+Ce^{-(\fc-\alpha)\frac{r}{2}} |\hat \zeta_0|\\
 \leq  &C e^{-(\fc-\alpha)\frac{r}{4}}( \|\zeta_{i}|_{|s_i|\geq \frac{r}{4}}\|_{k  ,2,\alpha}+|\hat \zeta_0|)
 \end{align*}
where $\zeta=(\zeta_{1}+\hat \zeta_0,\zeta_{2}+\hat \zeta_0)$. By \eqref{Q*P-2},  Lemma \ref{lem_def_HP},
Lemma \ref{est_DQ}, Lemma \ref{lem_est_f_r} and \eqref{local_calculation_2}  we get
\begin{equation}\label{eqn_B_1}
\left\|\frac{\partial}{\partial r}(Q^*P_{r})\circ H_{r}f_{(r)}(I_{r}(\kappa,\zeta))\right\|_{k-1,2,\alpha}
 \leq Ce^{-(\fc -\alpha)\frac{r}{4}}(\|(\kappa,\zeta)\|_{k,2,\alpha} + 1) .
\end{equation}
By Lemma \ref{lem_est_HDS} we have
$$
(B)\leq \left\|\frac{\partial}{\partial r}(Q^*P_{r})\circ H_{r}f_{(r)}(I_{r}(\kappa,\zeta))\right\|_{k-1,2,\alpha}\leq Ce^{-(\fc -\alpha)\frac{r}{4}}(\|(\kappa,\zeta)\|_{k,2,\alpha} + 1)
$$
Inserting the estimates of $(A)$ and $(B)$ into \eqref{operator_partial_right_inverse_kernel} we have
\begin{equation}\label{estimate_operator_partial_I}
\left\| H_{r}P_{r}\frac{\partial}{\partial r}(H_{r}f_{(r)}(I_{r}(\kappa,\zeta)))\right\|_{k-1,2,\alpha}\leq C  e^{-(\mathfrak{c} -5\alpha)\frac{r}{4} }  +C\mathsf{d}\left\| \frac{\partial}{\partial r}\left(\kappa^*_{r}, \xi^*_{(r)}\right)\right\|_{k-1,2,\alpha}.
\end{equation}
By \eqref{Q*P-1}, Lemma \ref{lem_def_HP} and Lemma \ref{est_DQ}  we get
\begin{align*}\left\|Q^*P_{r}\frac{\partial}{\partial r}\left( H_{r} f_{(r)}(I_{r}(\kappa,\zeta))\right)\right\|_{k-1,2,\alpha} &=\left\|Q^*P_{r}\left(H_{r}P_{r}\frac{\partial}{\partial r}\left( H_{r} f_{(r)}(I_{r}(\kappa,\zeta))\right)\right)\right\|_{k-1,2,\alpha} \\
\leq &
C\left\|H_rP_{r}\frac{\partial}{\partial r}\left( H_{r} f_{(r)}(I_{r}(\kappa,\zeta))\right)\right\|_{k-1,2,\alpha}.
\end{align*}
 Using \eqref{eqn_B_1}, \eqref{estimate_operator_partial_I} and Lemma \ref{lem_est_f_r}, Lemma \ref{lem_est_rI-1} we get
\begin{equation}\label{exp_ker_coord_est}
\left\|\frac{\partial}{\partial r}\left(\kappa^*_{r}, \xi^*_{(r)}\right)\right\|_{k-1,2,\alpha } \leq Ce^{-(\mathfrak{c} -5\alpha)\frac{r}{4} } +C\mathsf{d} \left\|\frac{\partial}{\partial r}\left(\kappa^*_{r}, \xi^*_{(r)}\right)\right\|_{k-1,2,\alpha } \end{equation}
Choose  $\mathsf{d}$   small  such that $4C\mathsf{d}<1.$ Then Theorem \ref{lem_est_rI} is proved.
$\Box$
\v

\n{\bf Proof of Corollary \ref{coordinate_decay}.} It is easy to see that, restricting to $\Sigma(R_0)$, we have $$I^*_r(\kappa,\zeta)+Q^*_{( \kappa_{o},b_{(r)})}f_{(r)}(I_r(\kappa,\zeta))=I_r(\kappa,\zeta)+Q_{( \kappa_{o},b_{(r)})}f_{(r)}(I_r(\kappa,\zeta)).$$
So we have an estimate for $\left\|\left.\frac{\p}{\p r}\left[I_r(\kappa,\zeta)+Q_{( \kappa_{o},b_{(r)})}f_{(r)}(I_r(\kappa,\zeta))\right]\right|_{\Sigma(R_0)}\right\|_{k,2}$. By Sobolev embedding theorem and the standard elliptic estimates we get Theorem \ref{coordinate_decay}. $\Box$
\begin{remark}
Repeating the all arguments in this section,  one can prove that there exists a constant $\mathsf C>0 $ such that
\begin{equation}
\left\|\frac{\partial }{\partial \tau}\left[I^*_r(\kappa,\zeta)+Q^*_{( \kappa_{o},b_{(r)})}f_{(r)}(I_r(\kappa,\zeta))\right]
 \right\|_{k-1,2,\alpha}\leq \mathsf C e^{-(\fc-5\alpha)\tfrac{r}{4} }(\mathsf{d}+1)
\end{equation}
 for any $(\kappa,\zeta)\in \ker D\mathcal {S}_{( \kappa_{o}, b_{o})}$. Since we need only a bound for $\|\frac{\p}{\p \tau}(\cdot)\|$, the calculations are much simpler. For example, consider (b) in Lemma \ref{lem_def_HP}, we have
\begin{equation}\label{eqn_HP_tau}
\left\|\frac{\p}{\p \tau}(H_{r}P_r)(\eta_{1},\eta_{2} )\right\|_{k-2,2,\alpha} \leq  C\sum_{i=1}^{2} \left\|\eta_{i}|_{r-1 \leq |s_{i}|\leq  r+1}\right\|_{\Sigma_{i},k-1,2,\alpha}.
\end{equation}
In fact, by \eqref{beta} we have
\begin{align}\label{eqn_eta_tau}
\frac{\p \tilde \eta_{1}}{\p \tau}=
 - \beta_{1;2}\sqrt{1- \beta_{1;2}^2(s_1)}\tfrac{\p( \eta_{2})}{\p s_{2}}(s_1-2r,t_1-\tau).
\end{align}
Then \eqref{eqn_HP_tau} follows from $H_{r}P_{r}(\eta_{1},\eta_{2})=(\tilde \eta_{1},\tilde \eta_{2})$ and \eqref{eqn_eta_tau}.
\end{remark}

\section{Extension}\label{Extension}
In this section we extend the Theorem \ref{coordinate_decay} to more general setting.

\subsection{Gluing several nodes}

\v
Let $(\Sigma, j,{\bf y}, {\bf q})$ be a marked nodal Riemann surface of genus $g$ with $n$ marked points ${\bf y}=(y_1,...,y_n)$ and $\mathfrak{e}$ nodal points ${\bf q}=(q_{1},\cdots,q_{\mathfrak{e}})$. Suppose that $\Sigma$ has $\iota$ smooth components $\Sigma_{i}$.
We assume that every component $(\Sigma_{\nu},j_{\nu},{\bf y}_{\nu},{\bf q}_{\nu})$ is stable. Let $\mathbf{A}=\mathbf{A}_1\times \mathbf{A}_2\times...\times \mathbf{A}_{\iota}$ be the space of complex structures ( incluing marked points ). Let $u=(u_{1},\cdots,u_{\iota})$, where $u_{\nu}:\Sigma_{\nu} \to M$ be $(j_{\nu},J)$-holomorphic map.
\v
For every node $q_i$ we choose the holomorphic cylindrical coordinates near the node $q_i$. We glue $\Sigma$ and $u$ at each node $q_i$ with parameter $(r_i, \tau_i)$ as in \S\ref{s_intro_2} to get $\Sigma_{(\mathbf{r})}$ and the pregluing map $u_{(\mathbf{r})}$.  Denote $z_{i}=e^{-r_{i}-2\pi \sqrt{-1}\tau_{i}}$ and $\mathbf z=(z_{1},\cdots,z_{\mathfrak{e}})$. Set
 $$b_o=(a_{o},0, u),\;\;\; b_{(\mathbf{r})}:=(a_{o},u_{(\mathbf r)}).$$
We can define $\mathcal B_{(\mathbf r)},$ $W^{k,2,\alpha}_{\mathbf{r},u_{(\mathbf{r})}}$ and
$L^{k-1,2,\alpha}_{\mathbf{r},u_{(\mathbf{r})}}$ as in section \ref{s_intro_2}. The Weil-Pertersson  metric induces a distance $d_{\mathbf{A}}(a_o,a)$ on $\mathbf{A}$.
Set
 \begin{align}\label{neig.-1}
O_{b_{o}}(\mathrm{R},\delta,\rho ):=&  \left\{(a ,\mathbf{z},v_{(\mathbf r)})| (a,\mathbf z)\in \mathbf{A} \times \mathbb{D}_{\mathbf{r}} ,\;v_{(\mathbf r)}\in \mathcal B_{(\mathbf r)},\;|r_{i}|< e^{-2\mathrm{R}},\; \right.\\
 \nonumber & \left.d_{\mathbf{A}}(j_{a},j_{o})<\delta,\; \| h_{(\mathbf r)}\|_{k,2,\alpha,\mathbf r}<\rho   \right\},
\end{align}
where $v_{({\mathbf r})}=\exp_{u_{({\mathbf r})}}(h_{({\mathbf r})})$. Denote by $\mathbf{g}_{a}$ the metric on $(\Sigma,j_{a}),$ and $|\mathbf{r}|:=min\{r_1,...,r_{\mathfrak{e}}\}.$

\begin{lemma}\label{isomor of ker}
For $|\mathbf{r}|>R_0$ there is an isomorphism
$$I_{(\mathbf{r})}: \ker D \mathcal{S}_{( \kappa_{o}, b_{o})}\longrightarrow \ker D \mathcal{S}_{( \kappa_{o},b_{(\mathbf{r})})}.$$
\end{lemma}

In order to get a global regularization we need to take a sum of several $K_{b_i}$. So we consider the following setting.
Let $K$ be a $N$-dimensional linear space. Let
$$\mathfrak{i}: K\times \mathbf{A}\times W^{k,2,\alpha}\left(\Sigma(R_0),(u\mid_{\Sigma(R_0)})^*TM \right)$$$$\to
W^{k-1,2,\alpha}\left(\Sigma(R_0),(u\mid_{\Sigma(R_0)})^*TM\otimes \wedge^{0,1}_{j_a}T^{*}\Sigma(R_0)\right)
$$
be a smooth map such that
 $D_{v} + d\mathfrak{i}_{(\kappa,a,v\mid_{\Sigma(R_0)})}$
is surjective for any $(\kappa,b)\in K\times  O_{b_o}(\mathrm{R},\delta,\rho)$, where $b=(a,\mathbf z,v)$, $v=\exp_{u_{(\mathbf{r})}}h$.
\v\n
Define a thickned Fredholm system
$(K\times O_{b_{o}}(\mathrm{R},\delta,\rho), K\times \E|_{O_{b_{o}}(\mathrm{R},\delta,\rho)}, \mathcal{S})$ with
 \begin{equation}\label{regu.equ}
\mathcal{S}(\kappa,b)=\bar{\partial}_{j_a,J}v +  \mathfrak{i}(\kappa, b).
\end{equation}

For fixed $(\mathbf{r})$ we consider the family of maps:
\begin{align*}& \mathcal{F}_{(\mathbf{r})}: K   \times \mathbf{A} \times W^{k,2,\alpha}\left(\Sigma_{(\mathbf r)},u_{(\mathbf r)}^*TM \right)\to W^{k-1,2,\alpha}\left(\Sigma_{(\mathbf r)},(u_{(\mathbf r)}^*TM\otimes \wedge^{0,1}_{j_a}T^{*}\Sigma_{(\mathbf r)}\right),\;\;\;\;\;\\
& \mathcal{F}_{(\mathbf{r})}(\kappa,a,h)=\Psi_{j_{a},j_{a_{o}}}\Phi_{u_{(\mathbf{r})}}(h)^{-1}\left(\bar{\partial}_{j_{a},J}v
+ \mathfrak{i}(\kappa,b)\right),
\end{align*}
where $b=( a,\mathbf{z}  , v),\;v=\exp_{u_{(\mathbf{r})}}h$  and
$ \Psi_{j_{a},j_{a_{o}}} $ is defined in section \S\ref{An isomorphism}.

By implicit function theorem (Theorem \ref{details_implicit_function_theorem}, Theorem \ref{smooth_implicit_function_theorem}),
there exist   $R>0$, a small neighborhood of $O_{a_{o}}(\delta)\subset \mathbf A$ and  a small neighborhood $O_{(\mathbf{r})}$ of $0 \in \ker\;D\mathcal S|_{b_{(\mathbf{r})}}$ and a unique smooth map
$$f_{(\mathbf{r})}: O_{a_{o}}(\delta)
\times O_{(\mathbf{r})}\rightarrow W^{k-1,2,\alpha}(\Sigma_{(\mathbf{r})},u_{(\mathbf{r})}^{*}TM\otimes \wedge^{0,1}T\Sigma_{(\mathbf{r})})$$ such that for any $(a,(\kappa,\zeta))\in O_{a_{o}}(\delta)
\times O_{(\mathbf{r})}$ and $|\mathbf r|>R,$
\begin{equation}\label{glu_solu_J_hol}
\mathcal  F_{(\mathbf r)} \left(a, I_{\mathbf{r}}(\kappa,\zeta)+Q_{(\kappa_{o},b_{(\mathbf{r})})}\circ f_{a,(\mathbf{r})} \circ I_{\mathbf{r}}(\kappa,\zeta) \right) =0.
\end{equation}

Denote by $Q_{(\kappa_{o},b_{(\mathbf{r})})}$ the right inverse of $D{\mathcal
S}_{(\kappa_o,b_{(\mathbf{r})})}$.
Then Theorem \ref{coordinate_decay} can be directly extended as
\begin{theorem}\label{coordinate_decay-2}
 Let $l\in \mathbb Z^+$ be a fixed integer. There exists positive  constants  $\mathsf C_{2,l}, \hbar, R_{0}$ such that for any $(\kappa,\zeta)\in \ker D \mathcal{S}_{(\kappa_{o},b_{o})}$ with $\|(\kappa,\zeta)\|< \mathsf{d}$, restricting to the compact set $\Sigma(R_0)$, for any $a\in \bigotimes_{i=1}^{\mathfrak{e}} O_i$, the following estimate hold
$$
\left\|\frac{\partial }{\partial r_{i}}\left(I_{\mathbf{r}}(\kappa,\zeta)+Q_{(\kappa_{o},b_{(\mathbf{r})})}\circ f_{a,(\mathbf{r})} \circ I_{\mathbf{r}}(\kappa,\zeta) \right) \right\|_{C^{l}(\Sigma(R_0))}+$$
$$\left\|\frac{\partial }{\partial \tau_{i}}\left(I_{\mathbf{r}}(\kappa,\zeta)+Q_{(\kappa_{o},b_{(\mathbf{r})})}\circ f_{a,(\mathbf{r})} \circ I_{\mathbf{r}}(\kappa,\zeta) \right) \right\|_{C^{l}(\Sigma(R_0))} \leq  C_{2,l}e^{-(\fc-5\alpha)\tfrac{r_{i}}{4} },
$$
$i=1,\cdots,\mathfrak{e}$.
\end{theorem}

\subsection{Estimates of higher derivatives}
Let $(s_{l}^{i},t_{l}^{i}),l=1,2$ be the cylinder coordinates near the node $q_{i}$. Set $$V_{i}:=\cup_{l=1}^{2}\left\{\left.\left(s_{l}^{i},t_{l}^{i}\right)\in\Sigma\;\right|\;\tfrac{r_{i}}{2} \leq |s_{l}^{i}|\leq \tfrac{3r_{i}}{2} \right\}.$$ Denote $$Glu_{a,(\mathbf{r})}(\kappa,\zeta)=I_{(\mathbf{r})}(\kappa,\zeta)+Q_{(\kappa_{o},b_{(\mathbf{r})})}\circ f_{a,(\mathbf{r})} \circ I_{\mathbf{r}}(\kappa,\zeta),$$
$$Glu^{*}_{a,(\mathbf{r})}(\kappa,\zeta)=I^{*}_{(\mathbf{r})}(\kappa,\zeta)+Q^{*}_{(\kappa_{o},b_{(\mathbf{r})})}\circ f_{a,(\mathbf{r})} \circ I_{\mathbf{r}}(\kappa,\zeta).$$
In this subsection we prove
\begin{theorem}\label{coordinate_decay-2}
	There exists positive  constants  $\mathsf C, \mathsf{d}, R_{0}$ such that for any $(\kappa,\zeta)\in \ker D \mathcal{S}_{(\kappa_{o},b_{o})}$ with $\|(\kappa,\zeta)\|< \mathsf{d}$, for any $X_{i}\in \{\frac{\p}{\p r_{i}},\frac{\p}{\p \tau_{i}}\},i=1,\cdots,\mathfrak{e}$, the following estimate hold
	$$ \left\|X_{i}X_{j}\left(Glu^{*}_{a,(\mathbf{r})}(\kappa,\zeta) \right) \right\|_{k-2,2,\alpha}+\left\|\left.X_{i}\left(Glu^{*}_{a,(\mathbf{r})}(\kappa,\zeta) \right)\right|_{V_{j}} \right\|_{k-2,2,\alpha}
	\leq  \mathsf{C} e^{-(\fc-5\alpha)\tfrac{r_{i}+r_{j}}{4} },$$
	$1\leq i\neq j\leq\mathfrak{e},$
	for any $a\in \bigotimes_{i=1}^{\mathfrak{e}} O_i$.
In particular, restricting to the compact set $\Sigma(R_0)$ and for any $l\in \mathbb Z^+$,
$$ \left\|X_{i}X_{j}\left(Glu_{a,(\mathbf{r})}(\kappa,\zeta) \right)  \right\|_{C^{l}(\Sigma(R_{0}))}
	\leq  C_{l}e^{-(\fc-5\alpha)\tfrac{r_{i}+r_{j}}{4} },$$
for some  constant $\mathsf{C}_{l}$.
\end{theorem}
\v\n
{\bf Proof} We give a sketch of the proof. Denote $\eta=(\eta_{1},\cdots,\eta_{\iota}) .$
	Set
	$$
	D_{l}^{i}(R_{0})=\left\{\left.\left(s_{l}^{i},t_{l}^{i}\right)\in\Sigma\;\right| |s_{l}^{i}|\geq R_{0} \right\},\;\;\;\;D^{i}(R_{0})=\cup_{l=1}^{2}D_{l}^{i}(R_{0}).
	$$
	   Denote
$$
\beta_{1,i;R}(s^{i}_{1})=\beta\left(\frac{1}{2}+\frac{r_{i}-s^{i}_1}{R}\right),\;\;\;\;\beta_{2,i;R}(s^{i}_{2})
=\sqrt{1-\beta^2\left(\frac{1}{2}-\frac{s^{i}_{2}+r_{i}}{R}\right)}.
$$
 We can define $h_{\mathbf r}^{*},$  $h_{(\mathbf r)},\tilde \xi_{(\mathbf r)} ,$ $H_{\mathbf r}$ and $P_{\mathbf r},\cdots$  as   before. Let $\eta_{l}^{i}=\eta|_{D^{i}_{l}(R_{0})},l=1,2.$ Obviously
 $H_{\mathbf r}P_{\mathbf r}(\eta)|_{D^{i}(R_{0})}=(\beta_{1,i;2}(\sum_{l=1}^{2}\beta_{l,i;2}\eta_{l}^{i} ),\beta_{2,i;2}(\sum_{l=1}^{2}\beta_{l,i;2}\eta_{l}^{i} ) ) $.
  It is easy to see that for any $1\leq i\neq j,\ell\leq \mathfrak{e},$   and $l=1,2,$
$$
\frac{\p(H_{\mathbf r}P_{\mathbf r})}{\p r_{i} }(\eta)|_{V_{j}}=0,\;\;\;\;\;\frac{\p^2(H_{\mathbf r}P_{\mathbf r})}{\p r_{i} \p r_{j}}(\eta)=0,\;\;\;\;\;\frac{\p^2\beta_{l,\ell;r_{\ell}}}{\p r_{i} \p r_{j}}=\frac{\p^2\beta_{l,\ell;2}}{\p r_{i} \p r_{j}}=0,$$$$\frac{\p\beta_{l,j,r_{j}}}{\p r_{i}}=\frac{\p\beta_{l,j,2}}{\p r_{i}}=0,\;\;\;supp \frac{\p\beta_{l,i,r_{i}}}{\p r_{i}}\subset V_{i},\;\;\;\;supp \frac{\p\beta_{l,i,2}}{\p r_{i}}\subset V_{i}.
$$
In the following we assume that  $1\leq i\neq j\leq \mathfrak{e}.$
Let $(\kappa,h)=(\kappa,h_{1},\cdots,h_{\iota})=Q_{(\kappa_{o},b_{o})}(H_{\mathbf r}P_{\mathbf r})(\eta) .$ Then we have $\frac{\p^2 \kappa}{\p r_{i}\p r_{j}}=0$ and $\frac{\p^2 h}{\p r_{i}\p r_{j}}=0.$ It follows that
$$
\frac{\p^2 h^{*}_{\mathbf r}}{\p r_{i} \p r_{j}}=0,\;\;\;\; \frac{\p^2(H_{\mathbf r}D(\kappa_{\mathbf r},h_{(\mathbf r)})) }{\p r_{i} \p r_{j}}= \frac{\p^2(H_{\mathbf r}D(\kappa_{\mathbf r},h_{(\mathbf r)})) }{\p r_{i} \p r_{j}}=0
$$
Let $h_{l}^{i}=h|_{D^{i}_{l}(R_{0})},l=1,2.$ Then $(h_{1}^{i},h^{i}_{2})$ is the restriction of $h$ near the node $q_{i}$. Obviously, $(Q^{\prime})^{*}P_{\mathbf r}(\eta)|_{D^{i}}=(\kappa,\beta_{1,{i},r_{i}}h_{1}^{i},\beta_{2,{i},r_{i}}h_{2}^{i}).$ Taking the derivative $\frac{\p^2}{\p r_{i}\p r_{j}}$ of $(Q^{\prime})^{*}P_{\mathbf r}$ we obtain
\begin{align*}
\left.\frac{\p}{\p r_{j}}((Q^{\prime})^{*} P_{\mathbf r}) (\eta)\right|_{V_{i}}&=\left. \left(0, \beta_{1,{i},r_{i}}\frac{\p h_{1}^{i}}{\p r_{j}}, \beta_{2,{i},r_{i}}\frac{\p h_{2}^{i}}{\p r_{j}}\right)\right|_{V_{i}},\\
\left.\frac{\p^2}{\p r_{i}\p r_{j}}((Q^{\prime})^{*} P_{\mathbf r}) (\eta)\right|_{D^{\ell}}&=\left.\delta_{\ell,i}\left(0,\frac{\p\beta_{1,{i},r_{i}}}{\p r_{i}}\frac{\p h_{1}^{i}}{\p r_{j}},\frac{\p\beta_{2,{i},r_{i}}}{\p r_{i}}\frac{\p h_{2}^{i}}{\p r_{j}}\right)\right|_{D^{\ell}}+\left.\delta_{\ell,j}\left(0,\frac{\p\beta_{1,{j},r_{j}}}{\p r_{j}}\frac{\p h_{1}^{j}}{\p r_{i}},\frac{\p\beta_{2,{j},r_{j}}}{\p r_{j}}\frac{\p h_{2}^{j}}{\p r_{i}}\right)\right|_{D^{\ell}}.
\end{align*}
Applying \eqref{f_tube_estimate_curve} of Lemma \ref{tube_ker_c-1} we have
\begin{align} \label{sec_app_right}
&\left\|\left.\frac{\p}{\p r_{j}} ((Q^{\prime})^{*}P_{\mathbf r})(\eta)\right|_{V_{i}}\right\|_{k-1,2,\alpha} + \left\|\frac{\p^2}{\p r_{i}\p r_{j}}((Q^{\prime})^{*} P_{\mathbf r}) (\eta)\right\|_{k-2,2,\alpha} \\
\leq &Ce^{-\frac{(\fc-\alpha)r_{j}}{4}} \left\|\eta|_{V_{i}}\right\|_{k-1,2,\alpha}
+Ce^{-\frac{(\fc-\alpha)r_{i}}{4}} \left\|\eta|_{V_{j}}\right\|_{k-1,2,\alpha}.\nonumber
\end{align}
Similar  we  obtain that
\begin{align} \label{sec_dq}
&\left\|\left.\frac{\p}{\p r_{j}}\left(H_{\mathbf r}  (DQ') ^{-1} P_{\mathbf r}\right)(\eta)\right|_{V_{i}}\right\|_{k-1,2,\alpha} + \left\|\frac{\p^2}{\p r_{i}\p r_{j}}\left(H_{\mathbf r}  (DQ') ^{-1} P_{\mathbf r}\right) (\eta)\right\|_{k-2,2,\alpha} \\
\leq &Ce^{-\frac{(\fc-\alpha)r_{j}}{4}} \left\|\eta|_{V_{i}}\right\|_{k-1,2,\alpha}
+Ce^{-\frac{(\fc-\alpha)r_{i}}{4}} \left\|\eta|_{V_{j}}\right\|_{k-1,2,\alpha}.\nonumber
\end{align}
and
\begin{align*}
& \left\|\left.\frac{\p}{\p r_{j}}I^*_{\mathbf r}(\kappa,h+\hat h_{0})\right|_{V_{i}}
\right\|_{k-1,2,\alpha} +\left\|\frac{\p^2}{\p r_{i}\p r_{j}}I^*_{\mathbf r}(\kappa,h+\hat h_{0})
\right\|_{k-2,2,\alpha}
\\
\leq &\mathsf C\left( e^{-\frac{(\fc-\alpha)r_{j}}{4}}\|h|_{V_{i}}\|_{k  ,2,\alpha}+e^{-\frac{(\fc-\alpha)r_{i}}{4}}\|h|_{V_{j}}\|_{k  ,2,\alpha}\right)+\mathsf C e^{(\fc-\alpha)\frac{r_{i}+r_{j}}{2}} |\hat h_0|.
\end{align*}
Note that, restricting in $V_{i},$
$$
\tilde{\nabla}_{\frac{\p}{\p r_{j}}}\tilde \xi_{(\mathbf r)}=\phi_{\mathbf r}\tilde{\nabla}_{\frac{\p}{\p r_{j}}}  \xi^{*}_{(\mathbf r)},\;\;\;\;\;\frac{\p u_{(\mathbf r)}}{\p r_{j}}=0.
$$
Then by the same calculation of Lemma \ref{difference} we have
\begin{align}\label{eqn_diff_rest}
\left\|\left.H_{\mathbf r}\circ D(Id, \phi_{\mathbf r})\circ \frac{\partial}{\partial r_{j}}\left(\kappa^*_{\mathbf r}, \xi^*_{(\mathbf r)}\right)\right|_{V_{i}}\right\|_{k-1,2,\alpha}
\leq \mathsf{C}  \mathsf {d}\left\|\left.\frac{\partial}{\partial r_{j}}\left(\kappa^*_{\mathbf r}, \xi^*_{(\mathbf r)}\right)\right|_{V_{i}}\right\|_{k-1,2,\alpha}.
\end{align}
Using \eqref{sec_app_right}, \eqref{sec_dq}, \eqref{eqn_diff_rest} and   the same proof of Theorem \ref{lem_est_rI} word by word, we have
\begin{equation}\label{eqn_est_sol_rest}
\left\|\frac{\partial }{\partial r_{i}}\left.\left(\kappa^*_{\mathbf r}, \xi^*_{(\mathbf r)}\right)\right|_{V_{j}} \right\|_{k-1,2,\alpha}+\left\|\left.  \frac{\p(H_{\mathbf r} f_{(\mathbf r)}I_{\mathbf r}(\kappa,\zeta))}{\p r_{i}}\right|_{V_{j}}\right\|_{k-2,2,\alpha}\leq Ce^{-(\fc-5\alpha)\frac{r_{i}+r_{j}}{4}}.
\end{equation}
Using \eqref{eqn_est_sol_rest}, Theorem  Theorem \ref{lem_est_rI} and the Cauchy-Schwarz inequality, by the same argument of Lemma \ref{difference}, we have
\begin{align}\label{sec_est_diff}
 \left\|H_{\mathbf r} D(Id, \phi_{\mathbf r})\circ  \frac{\p^2}{\p r_{i}\p r_{j}} \left(\kappa^*_{\mathbf r}, \xi^*_{(\mathbf r)}\right)\right\|_{k-2,2,\alpha}
\leq  \mathsf{C}  \mathsf {d}\left\|\frac{\p^2}{\p r_{i}\p r_{j}}\left(\kappa^*_{\mathbf r}, \xi^*_{(\mathbf r)}\right)\right\|_{k-2,2,\alpha}+\mathsf{C} e^{-(\fc-5\alpha)\frac{r_{i}+r_{j}}{4}}.
\end{align}
Taking the derivative $\frac{\p^2}{\p r_{i}\p r_{j}}$ of  \eqref{kernel_express} and multiplying $H_{\mathbf r} D(Id,\phi_{r})$ on both sides  we get
\begin{align*}
&H_{\mathbf r} D(Id,\phi_{\mathbf r})\circ\frac{\p^2(\kappa^*_{\mathbf r}, \xi^*_{(\mathbf r)})}{\p r_{i}\p r_{j}} \\
= &H_{\mathbf r} D(Id,\phi_{\mathbf r})\circ\frac{\p^2I^*_{\mathbf r}(\kappa,\zeta) }{\partial r_{i}\p r_{j}} + \frac{\p^2(Q^* P_{\mathbf r} )}{\p r_{i}\p r_{j}} \circ H_{\mathbf r} f_{(\mathbf r)}I_{\mathbf r}(\kappa,\zeta)  + H_{\mathbf r} P_{\mathbf r} \frac{\p^2(H_{\mathbf r} f_{(\mathbf r)}I_{\mathbf r}(\kappa,\zeta))}{\p r_{i}\p r_{j}} \\
&+ H_{\mathbf r} D(Id,\phi_{\mathbf r})\circ\frac{\p (Q^* P_{\mathbf r})}{\p r_{i}}\frac{\p(H_{\mathbf r} f_{(\mathbf r)}I_{\mathbf r}(\kappa,\zeta))}{\p r_{j}}+H_{\mathbf r} D(Id,\phi_{\mathbf r})\circ \frac{\p (Q^* P_{\mathbf r})}{\p r_{j}}\frac{\p(H_{\mathbf r} f_{(\mathbf r)}I_{\mathbf r}(\kappa,\zeta))}{\p r_{i}}  .
\end{align*}
Using \eqref{operator_partial_right_inverse_kernel}, \eqref{sec_est_diff}, $\frac{\p(H_{\mathbf r}P_{\mathbf r})}{\p r_{j}}\subset V_{j}$, Lemma \ref{lem_est_f_r} and  $$\frac{\partial}{\partial r_{j}}(H_{\mathbf r}f_{(\mathbf r)}(I_{\mathbf r}(\kappa,\zeta)))=\frac{\p(H_{\mathbf r}P_{\mathbf r})}{\p r_{j}}\circ H_{\mathbf r}f_{(\mathbf r)}(I_{\mathbf r}(\kappa,\zeta))+H_{\mathbf r}P_{\mathbf r}\frac{\partial}{\partial r_{j}}(H_{\mathbf r}f_{(\mathbf r)}(I_{\mathbf r}(\kappa,\zeta))),$$ by the same  argument of \eqref{eqn_B_1} we get
$$
\left\|\frac{\p (Q^* P_{\mathbf r})}{\p r_{i}}\frac{\p(H_{\mathbf r} f_{(\mathbf r)}I_{\mathbf r}(\kappa,\zeta))}{\p r_{j}}\right\|_{k-2,2,\alpha}\leq Ce^{-(\fc-5\alpha)\frac{r_{i}+r_{j}}{4}}.
$$
Then repeating the proof of  \eqref{estimate_operator_partial_I} we have
\begin{equation*}
\left\| H_{\mathbf r} P_{\mathbf r} \frac{\p^2(H_{\mathbf r} f_{(\mathbf r)}I_{\mathbf r}(\kappa,\zeta))}{\p r_{i}\p r_{j}}\right\|_{k-2,2,\alpha}\leq C  e^{-(\mathfrak{c} -5\alpha)\frac{r_{i}+r_{j}}{4} }  +C\mathsf{d}\left\| \frac{\partial^2}{\p r_{i}\p r_{j}}\left(\kappa^*_{\mathbf r}, \xi^*_{(\mathbf r)}\right)\right\|_{k-2,2,\alpha}.
\end{equation*}
Then as in the proof of \eqref{exp_ker_coord_est} we conclude that
\begin{equation*}
\left\|\frac{\partial^2}{\p r_{i}\p r_{j}}\left(\kappa^*_{\mathbf r}, \xi^*_{(\mathbf r)}\right)\right\|_{k-2,2,\alpha } \leq Ce^{-(\mathfrak{c} -5\alpha)\frac{r_{i}+r_{j}}{4} } +C\mathsf{d} \left\|\frac{\partial^2}{\p r_{i}\p r_{j}}\left(\kappa^*_{\mathbf r}, \xi^*_{(\mathbf r)}\right)\right\|_{k-2,2,\alpha }. \end{equation*}
Choose  $\mathsf{d}$   small  such that $4C\mathsf{d}<1.$ Since  $$I^*_{\mathbf r}(\kappa,\zeta)+Q^*_{( \kappa_{o},b_{(\mathbf r)})}f_{(\mathbf r)}(I_{\mathbf r}(\kappa,\zeta))=I_{\mathbf r}(\kappa,\zeta)+Q_{( \kappa_{o},b_{(\mathbf r)})}f_{(\mathbf r)}(I_{\mathbf r}(\kappa,\zeta))\;\;\mbox{ on }\;\Sigma(R_0)$$
Theorem \ref{coordinate_decay-2} holds.\;\;\;$\Box$

\section{Appendix}
\subsection{Linearized operator }\label{Linearized operator}
Choose local normal coordinates $(x^{1},\cdots,x^{2m})$ in a neighborhood  $O_{u(q)}$ of $u(q)$ such that
$$
(x^{1},\cdots,x^{2m})(u(q))=0,\;\;\;\;\;\left.J\frac{\p}{\p x^{i}}\right|_{0}=\left.\frac{\p}{\p x^{m+i}}\right|_{0},\;\;\;\;\;\;\;
\left.J\frac{\p}{\p x^{m+i}} \right|_{0}=\left.-\frac{\p}{\p x^{i}}\right|_{0},\;\;\;\;\;\;i\leq m.
$$
For any $h\in W^{k,2}(\Sigma, u^{*}TM)$ we can write
$h=\sum_{i=1}^{2m} h^{i}\frac{\p }{\p x^{i}},$ with $h^{i}\in W^{k,2}(\Sigma, \mathbb R).$
For fixed $j,$ denote by $D_{u}^{(j)}$ the linearized operator of $\bar{\p}_{j,J}$ at $u.$  Let $(s,t)$ be the local coordinates on $\Sigma$ with $j\frac{\p}{\p s}=\frac{\p }{\p t}.$
 Since
\begin{equation}\label{eqn_line_j}
D_{u}^{(j)} h=\frac{1}{2}\left(\nabla h+J(u)\nabla h \circ j    \right)-\frac{1}{4} J(u)\nabla_{h}J\left(du-J(u)du \circ j    \right)
\end{equation}
we have
\begin{align*}
&D_{u}^{(j)}h\left(\frac{\p}{\p s}\right)=\frac{1}{2}\sum_{i=1}^{2m}\left(\frac{\p h^{i}}{\p s} +J_{0}\frac{\p h^{i}}{\p t}   \right)\frac{\p }{\p x^{i}} -\frac{1}{4} J(u)\nabla_{h}J\left(\frac{\p u^{i}}{\p s} -J(u)\frac{\p u^{i}}{\p t}  \right)\frac{\p }{\p x^{i}} \\
&+\frac{1}{2}\sum_{i,k=1}^{2m} h^{i}\left(\frac{\p u^{k}}{\p s}\nabla_{ \frac{\p }{\p x^{k}}} \frac{\p }{\p x^{i}}+J(u)\frac{\p u^{k}}{\p t}\nabla_{ \frac{\p }{\p x^{k}}} \frac{\p }{\p x^{i}}  \right)+\frac{1}{2}\sum_{i=1}^{2m}(J(u(q)-J_{0})\frac{\p h^{i}}{\p t}\frac{\p }{\p x^{i}}
\end{align*}
Let $(J_{i}^{k})$  be the matrix such that  $\sum_{k=1}^{2m} J_{i}^{k}\frac{\p}{\p x^{k}}:=J(\frac{\p}{\p x^{i}})$.
Let $\Gamma^{k}_{il}$ be the Christoffel symbol of a connection $\nabla$ with respect to a local frame $(\frac{\p}{\p x^{i}}),$ i.e., $\nabla_{ \frac{\p }{\p x^{k}}} \frac{\p }{\p x^{i}}=\sum_{l=1}^{2m} \Gamma_{ki}^{l}\frac{\p }{\p x^{l}}$. Then we can write $D_{u}h\left(\frac{\p}{\p s}\right)$ as
\begin{align*}
&2D_{u}^{(j)}h\left(\frac{\p}{\p s}\right)= \sum_{i=1}^{2m}\left(\frac{\p h^{i}}{\p s} +J_{0}\frac{\p h^{i}}{\p t}   \right)\frac{\p }{\p x^{i}} -\frac{1}{2} \sum_{i,k,l,a=1}^{2m} h^{l}J_{k}^{i}(u)\left(\nabla_{\frac{\p}{\p x^{l}}}J\right)_{a}^{k}\left(\frac{\p u^{a}}{\p s} -\sum_{e=1}^{2m}J^{a}_{e}(u)\frac{\p u^{e}}{\p t}  \right)\frac{\p }{\p x^{i}} \\
&+ \sum_{i,k,l=1}^{2m} h^{l}\left(\frac{\p u^{k}}{\p s}\Gamma_{kl}^{i}+ \sum_{a=1}^{2m}  J_{a}^{i}(u)\frac{\p u^{k}}{\p t}\Gamma_{kl}^{a} \right)\frac{\p }{\p x^{i}}+ \sum_{i,k=1}^{2m}(J_{l}^{i}(u)-(J_{0})_{l}^{i})\frac{\p h^{l}}{\p t}\frac{\p }{\p x^{i}}.
\end{align*}
$2D_{u}^{(j)}h\left(\frac{\p}{\p s}\right)$ may simply be written as $D_{u}^{(j)}h$ when no amlargeuity can arise.
In the matrix form,  $D_u^{(j)}$ can be written as
$$
D_u^{(j)} \begin{pmatrix} h^1 \\
     \vdots \\ h^{2m}
     \end{pmatrix}
  =  \frac{\partial  }{\partial s}\begin{pmatrix} h^1 \\
     \vdots \\
     h^{2m}
     \end{pmatrix}
   +J_{0}\frac{\partial }{\partial t}\begin{pmatrix} h^1 \\
     \vdots \\
     h^{2m}
     \end{pmatrix}
    +F^{1,(j)}_{u} \begin{pmatrix} h^1 \\
     \vdots \\
     h^{2m}
     \end{pmatrix}+F^{2,(j)}_{u}\frac{\partial }{\partial t} \begin{pmatrix} h^1 \\
     \vdots \\
     h^{2m}
     \end{pmatrix}
$$
where   $F_{u}^{1,(j)},F_{u}^{2,(j)}$ are matrixs given by
\begin{align}\label{expression_of_S}
(F_{u}^{1,(j)})_{l}^{i}&=\sum_{k=1}^{2m} \left(\frac{\p u^{k}}{\p s}\Gamma_{kl}^{i}+ \sum_{a=1}^{2m}  J_{a}^{i}(u)\frac{\p u^{k}}{\p t}\Gamma_{kl}^{a} \right)-\frac{1}{2} \sum_{k, a=1}^{2m} J_{k}^{i}(u)\left(\nabla_{\frac{\p}{\p x^{l}}}J\right)_{a}^{k}\left(\frac{\p u^{a}}{\p s} - \sum_{e=1}^{2m} J^{a}_{e}(u)\frac{\p u^{e}}{\p t}  \right)\\\label{expression_of_F}
(F_{u}^{2,(j)})_{l}^{i}&=J_{l}^{i}(u)-(J_{0})_{l}^{i}
\end{align}
If there is no confusion we  will denote $D_{u}^{(j)},F^{i,(j)}_{u}$ by $D_{u},F^{i}_{u}$, $i=1,2$.
We view $D_u$ as a operator
$D_u:W^{k,2,\alpha}(\Sigma, \mathbb R^{2m})\longrightarrow L^{k-1,2,\alpha}(\Sigma, \mathbb R^{2m}) $
\begin{equation}\label{loc_exp_Du}
D_u=\frac{\partial }{\partial s}+J_{0} \frac{\partial }{\partial t} + F_{u}^{1}+F_{u}^{2} \frac{\partial }{\partial t}.
\end{equation}
\v
\subsection{Implicit function theorem}\label{Implicit function theorem}

The following theorem is a restatement of Theorem A.3.3 and Proposition A.3.4 in \cite{MS}.

\begin{theorem}\label{details_implicit_function_theorem}
Let  $(A,\|\cdot\|_{A})$,  $(X,\|\cdot\|_{X})$ and $(Y,\|\cdot\|_{Y})$ be Banach spaces, $U\subset X$  be open sets and  $V\subset A$, $U\subset X$  be open sets and $F: V\times U\rightarrow Y$ be a continuously differentiable map.   For any $(a,x)\in V\times U$ define
$$
D_{a}F(a,x)(g)=\frac{d}{dt}F(a+tg ,x)|_{t=0},\;\;D_{x}F(a,x)(h)=\frac{d}{dt}F(a ,x+th)|_{t=0},\;\;\;\forall\;g\in A,\; h\in X.
$$
Suppose that   $D_{x}F(a_{o},x_{o})$  is surjective and has a bounded linear right inverse $Q_{(a_{o},x_{o})}:Y\longrightarrow X$ with $\|Q_{(a_{o},x_{o})}\|\leq \mathsf C$ for some constant $\mathsf C>0$.  Choose a positive constant $\delta >0$  such that
\begin{equation}\label{c_of_differential}
\|D_{x}F(a,x)- D_{x}F(a_{o},x_{o})\|\leq \frac{1}{2\mathsf C},\;\;\;\;\forall \;x\in B_{\delta}(x_{o},X),\;a\in B_{\delta}(a_{o},A).
\end{equation}
where $B_{\delta}(a_{o},A)=\{a\in A|\; \|a-a_{o}\|_{A} \leq \delta\},B_{\delta}(x_{o},X)=\{x\in X|\; \|x-x_{o}\|_{X} \leq \delta\}.$ Suppose that $x_1\in X$ and $a\in  B_{\delta}(a_{o},A)$ satisfies
\begin{equation}\label{small_value_of_F}
\|F(a,x_1)\|_{Y}<\frac{\delta}{4\mathsf C},\;\;\;\; \|x_{1}-x_{o}\|_{X}\leq \frac{\delta}{8}.
\end{equation}
Then there exists a unique $x\in X$ such that
\begin{equation}
F(a,x)=0,\;\;\;\;x-x_1\in \im\; Q,\;\;\;\;\|x-x_{o}\|_{X}\leq \delta,\;\;\;\; \|x-x_{1}\|_{X}\leq 2\mathsf C\|F(a,x_{1})\|_{Y}.
\end{equation}
Moreover, if $\|F(a_{o},x_{o})\|_{Y}\leq \frac{\delta}{4\mathsf C},$ there exist a constant $\delta'>0$ and a unique family differential map  $f_{a}:\ker D_{x}F(a_{o},x_{o})\rightarrow Y$ such that for any $(a,x)\in F^{-1}(0)\cap (B_{\delta'}(a_{o},A)\times B_{\delta'}(x_{o},X)),$ we have
 \begin{equation}\label{eqn_implicit_F}
 F(a,x)=0\Longleftrightarrow
 x=x_{o}+\zeta + Q_{(a_{o},x_{o})}\circ f_{a}(\zeta),\;\;\; \zeta \in \ker\;  D_{x}F(a_{o},x_{o})
 \end{equation}
 \end{theorem}
The following is a version of the implicit function theorem with parameters. For the proof please see \cite{LS-2}.
\begin{theorem}\label{smooth_implicit_function_theorem}
$F$ satisfies the assumption of Theorem \ref{details_implicit_function_theorem}.
   If $F:V\times U\longrightarrow Y$ is of class $C^{\ell}$, where $\ell$ is a positive integer, then  there exists a constant $\delta'>0$ such that  $F^{-1}(0)|_{B_{\delta'}(a_{o},A)\times  B_{\delta'}(x_{o},X)}$ is $C^{\ell}$ manifold,  and $\xi\rightarrow x_{o} + \xi + Q \circ f_{a}(\xi)$ is a $C^{\ell}$-chart of  $F^{-1}(0)|_{B_{\delta'}(a_{o},A)\times  B_{\delta'}(x_{o},X)}$.
 In particular,
\begin{equation}\label{eqn_impli_est_F}
\|D_{a} \left(x_{o}+\zeta + Q_{(a_{o},x_{o})}\circ f_{a}(\zeta)\right)\|\leq C,
\end{equation}
where $C>0$ is a constant depending only on $C_{1}$, $\mathsf C,\delta'$, $\|f_a\|$ and  $\|D^2_{ax}F(a,x_{o})\|$.
 \end{theorem}

\subsection{An isomorphism between $u^{*}TM\otimes \wedge^{0,1}_{j}T^{*}\Sigma$ and $u^{*}TM\otimes \wedge^{0,1}_{j_{o}}T^{*}\Sigma$ }\label{An isomorphism}
\v
Let $\mathcal{J}(\Sigma)\subset End(T\Sigma)$ denote the manifold of complex structures on $\Sigma$ and $j_o\in \mathcal{J}(\Sigma)$. For any $j\in \mathcal{J}(\Sigma)$ near $j_o$
we can write $j=(I + H)j_o(I + H)^{-1}$ where $H\in T_{j_o}\mathcal{J}(\Sigma)$.
We define two maps
$$\Psi_{j_{o},j}:u^{*}TM\otimes \wedge^{0,1}_{j_{o}}T^{*}\Sigma  \to  u^{*}TM\otimes \wedge^{0,1}_{j}T^{*}\Sigma$$ and $$\Psi_{j,j_{o}}:u^{*}TM\otimes \wedge^{0,1}_{j}T^{*}\Sigma  \to  u^{*}TM\otimes \wedge^{0,1}_{j_{o}}T^{*}\Sigma$$ by
\begin{align*}
 \Psi_{j_{o},j}(\eta)= \frac{1}{2}(\eta-\eta \cdot j_{o} j),\;\;\;\; \Psi_{j,j_{o}}(\varpi)= \frac{1}{2}(\varpi-\varpi \cdot jj_{o}).
\end{align*}
Since $J\eta=-\eta j_{o}$ and $J\varpi=-\varpi j,$ one can check that $J \Psi_{j_{o},j}(\eta) =-\Psi_{j_{o},j}(\eta) j$ and $J \Psi_{j,j_{o}}(\varpi) =-\Psi_{j,j_{o}}(\varpi) j_{o}.$ Then $\Psi_{j_{o},j}$ and $\Psi_{j,j_{o}}$ are well defined. %Now we consider $\Psi_{j,j_{o}}\Psi_{j_{o},j}(\eta).$
\begin{lemma}\label{isomorphism}
$\Psi_{j_{o},j}$ is an isomorphism when $|H|$ small enough.
\end{lemma}
\v\n
{\bf Proof.} By the definition we have
\begin{align*}
\Psi_{j,j_{o}}\Psi_{j_{o},j}(\eta)=\frac{1}{4}\left(2\eta-\eta\cdot(jj_{o}+j_{o}j)\right) ).
\end{align*}
 A direct calculation gives us
$$ 1-C|H|\leq \|\Psi_{j,j_{o}}\Psi_{j_{o},j}\|\leq 1+C|H| $$
where $|H| =\sup\limits_{p\in \Sigma,X\in T_{p}\Sigma}\left\{   \left. (HX,X)_{\mathbf{g}_{o}(p)}\;\right| (X,X)_{\mathbf{g}_{o}(p)}=1\right\}.$
Then $\Psi_{j,j_{o}}\Psi_{j_{o},j}$ is an isomorphism as $|H|$ small enough. In particular, $\Psi_{j_{o},j}$ is injective and $\Psi_{j,j_{o}}$ is surjective.  Similarly  $\Psi_{j_{o},j}\Psi_{j,j_{o}}$ is also an isomorphism. Hence  $\Psi_{j_{o},j}$  and $\Psi_{j,j_{o}}$ are isomorphisms.
$\Box$
\v
$\Psi_{j_{o},j}$ induces an isomorphism  $$\Psi_{j_{o},j}: W^{k-1,2,\alpha}(\Sigma, u^{*}TM\otimes \wedge^{0,1}_{j_{o}}T^{*}\Sigma)\to W^{k-1,2,\alpha}(\Sigma, u^{*}TM\otimes \wedge^{0,1}_{j_{a}}T^{*}\Sigma)$$
in a natural way.


\begin{thebibliography}{L3}

\bibitem{A} T. Aubin, Nonlinear analysis on manifolds. Monge-Ampššre equations, Grundlehren der Mathematischen Wissenschaften [Fundamental Principles of Mathematical Sciences], 252, Berlin, New York: Springer-Verlag.
\bibitem{D} S. K, Donaldson,  Floer homology groups in Yang-Mills theory. With the assistance of M. Furuta and D. Kotschick. Cambridge Tracts in Mathematics, 147. Cambridge University Press, Cambridge, 2002. viii+236 pp.

\bibitem{Gr} M. Gromov, Pseudo holomorphic curves in symplectic manifolds, Invent. math., 82 (1985), 307-347.

\bibitem{R}  A-M. Li and Y. Ruan, Symplectic surgery and Gromov-Witten invariants of Calabi-Yau 3- folds, Invent. Math. 145, 151-218(2001)
\bibitem{LS-1} A-M. Li and L. Sheng, Virtual Neighborhood Technique for Holomorphic Curve Moduli Spaces, Preprint.
\bibitem{LS-2} A-M. Li, Li, Sheng , A Finite Rank Bundle over $J$-Holomorphic map Moduli Spaces, Preprint.

\bibitem{MS} D.  McDuff and D. Salamon, $J$-holomorphic curves and symplective topology,   Colloquium Publications, vol. 52, Amer. Math. Soc., Providence, RI, 2004.

\bibitem{R1} Y. Ruan, Topological Sigma model and Donaldson type invariants in Gromov theory, Math. Duke J. 83(1996), 461-500.
\bibitem{RT1}  Y. Ruan,  G. Tian, A mathematical theory of quantum cohomology. J. Differential Geom.  42  (1995),  no. 2, 259-367.


\bibitem{SL} S. Lang, Real analysis, second edition,   Addison-Wesley Publishing Company, Advanced Book Program, Reading, MA, 1983. xv+533 pp.


\end{thebibliography}
\end{document}